\documentclass[10pt,reqno]{amsart}
\RequirePackage[OT1]{fontenc}
\RequirePackage{amsthm,amsmath}
\RequirePackage[numbers]{natbib}
\RequirePackage[colorlinks,linkcolor=blue,citecolor=blue,urlcolor=blue]{hyperref}
\usepackage{amssymb}
\usepackage{anysize}
\usepackage{hyperref}
\usepackage{extarrows}
\usepackage{multirow}
\usepackage{natbib}
\usepackage{stmaryrd}
\usepackage{color}
\usepackage{amsmath}
\usepackage{enumitem}

\allowdisplaybreaks

\numberwithin{equation}{section}

\theoremstyle{plain} 
\newtheorem{thm}{Theorem}[section]
\newtheorem{lem}[thm]{Lemma}
\newtheorem{cor}[thm]{Corollary}

\newtheorem{assumption}[thm]{Assumption}
\newtheorem{defn}[thm]{Definition}

\theoremstyle{remark}
\newtheorem{rem}[thm]{Remark}

\renewcommand{\Re}{\mathrm{Re}\,}
\renewcommand{\Im}{\mathrm{Im}\,}

\newcommand{\im}{\mathrm{Im}\,}

\newcommand{\E}{{\mathbb E }}
\newcommand{\R}{{\mathbb R }}
\newcommand{\N}{{\mathbb N}}
\newcommand{\Z}{{\mathbb Z}}
\renewcommand{\P}{{\mathbb P}}
\newcommand{\C}{{\mathbb C}}

\newcommand{\ii}{\mathrm{i}}
\newcommand{\deq}{\mathrel{\mathop:}=}
\newcommand{\e}[1]{\mathrm{e}^{#1}}
\newcommand{\ntr}{\mathrm{tr}\,}
\newcommand{\dd}{\mathrm{d}}
\newcommand{\ie}{\emph{i.e., }}
\newcommand{\eg}{\emph{e.g., }}
\newcommand{\cf}{\emph{c.f., }}
\newcommand{\PP}{\Phi}
\newcommand{\dL}{\mathrm{d}_{\mathrm{L}}}
\newcommand{\wt}{\widetilde}

\newcommand{\bs}{\boldsymbol}

\newcommand{\la}{\langle}
\newcommand{\ra}{\rangle}

\renewcommand{\mathbf}[1]{\bs{#1}}
 
\marginsize{33mm}{33mm}{33mm}{33mm}  

\begin{document}

 \begin{minipage}{0.85\textwidth}
 \vspace{2.5cm}
 \end{minipage}
\begin{center}
\large\bf
Local single ring theorem on optimal scale
\end{center}

\renewcommand{\thefootnote}{\fnsymbol{footnote}}	
\vspace{1cm}
\begin{center}
 \begin{minipage}{0.32\textwidth}
\begin{center}
Zhigang Bao\footnotemark[1]  \\
\footnotesize {HKUST}\\
{\it mazgbao@ust.hk}
\end{center}
\end{minipage}
\begin{minipage}{0.32\textwidth}
\begin{center}
L\'aszl\'o Erd{\H o}s\footnotemark[1]  \\
\footnotesize {IST Austria}\\
{\it lerdos@ist.ac.at}
\end{center}
\end{minipage}
\begin{minipage}{0.33\textwidth}
 \begin{center}
Kevin Schnelli\footnotemark[1]\\
\footnotesize 
{KTH Royal Institute of Technology}\\
{\it schnelli@kth.se}
\end{center}
\end{minipage}
\footnotetext[1]{Partially supported by ERC Advanced Grant RANMAT No.\ 338804.}

\renewcommand{\thefootnote}{\fnsymbol{footnote}}	

\end{center}

\vspace{1cm}

\begin{center}
 \begin{minipage}{0.8\textwidth} \footnotesize 
 
 Let $U$ and $V$ be two independent $N$ by $N$ random matrices that are distributed according to Haar measure on $U(N)$. Let $\Sigma$ be a non-negative deterministic $N$ by $N$ matrix. The {\it single ring theorem} \cite{GKZ11}
 asserts that the empirical eigenvalue distribution of the matrix $X\deq U\Sigma V^*$ converges weakly, in the limit of large $N$, to a deterministic measure which is supported on a single ring centered at the origin in~$\C$.  Within the bulk regime, \ie in the interior of the single ring, we establish the convergence of the empirical eigenvalue distribution on the optimal local scale of order $N^{-1/2+\varepsilon}$ and establish the optimal convergence rate. The same results hold true when~$U$ and~$V$ are Haar distributed on $O(N)$.

\end{minipage}
\end{center}

 \vspace{2mm}
 
 {\small
\footnotesize{\noindent\textit{Date}: March 1, 2019}
 
 \vspace{2mm}

 }

\thispagestyle{headings}
\section{Introduction and main result}

Consider the $N\times N$
random matrix of the form
\begin{align}
X\equiv X_N=U\Sigma V^*, \label{091101}
\end{align}
where $U\equiv U_N$ and $V\equiv V_N$ are two independent sequences of random matrices, which are both Haar distributed on either the unitary group, ${U}(N)$, of degree $N$, or on the orthogonal group, ${O}(N)$, of degree $N$. Moreover, let $\Sigma\equiv \Sigma_N$ be a sequence of $N\times N$ deterministic non-negative definite diagonal matrices. Note that in general $X$ is not hermitian and most of its eigenvalues are genuinely complex numbers. In fact, almost surely the matrix $X$ is not normal. Let $\lambda_j(X)$, 
 $j=1,2,\ldots, N$,  be the eigenvalues of $X$ and let 
 \begin{align}
 \mu_X\deq \frac{1}{N}\sum_{j=1}^N \delta_{\lambda_j(X)}
 \end{align}
 be the (normalized) empirical spectral distribution of $X$. We define $\mu_\Sigma$ analogously.

\begin{assumption}\label{ass one}
We assume that the sequence $(\Sigma_N)$ is uniformly bounded, \ie
 there exists a finite constant $S_+$ such that
\begin{align}\label{le big S_+} 
0\le \Sigma_N\le S_+\,.
\end{align}
\end{assumption}

From this assumption it follows that there is a constant $0< s_+<\infty$ such that, for all $N\in\N$,
\begin{align}\label{le s-s+}
\mathrm{supp}\,\mu_\Sigma\subset[0,s_+]\,.
\end{align}

We first consider the situation where there exists a limiting measure $\mu_\sigma$\footnote{We will often use the convention that capital letters
 indicate  random matrices and the corresponding small letters indicate their limiting objects.} of $\mu_\Sigma$, \ie
\begin{align}\label{levy distance first time}
 \mathrm{d}_{\mathrm{L}}(\mu_\Sigma,\mu_\sigma)\rightarrow 0\,,
\end{align}
as $N\rightarrow\infty$, where $\mathrm{d}_{\mathrm{L}}$ denotes the L\'evy distance. Given such a $\mu_\sigma$ on $[0,\infty)$, we define
\begin{align}\label{le radii}
r_-\deq\Big(\int_{\R^+} x^{-2} {\rm d} \mu_\sigma (x)\Big)^{-\frac12}\,,\qquad r_+\deq\Big(\int_{\R^+}  x^2 {\rm d} \mu_{\sigma}(x)\Big)^{\frac12}\,,
\end{align}
where we set $r_-=0$ in case the integral in its definition diverges. Note that if $\mu_\sigma$ is supported more than one point, we have $r_-<r_+$ as  follows from Schwarz inequality. We let
\begin{align}
\mathcal{R}_\sigma\equiv \mathcal{R}(\mu_\sigma)\deq\{w \in \mathbb{C}\,:\, r_- < |w| <  r_+\} \label{102911}
\end{align}
be the ring in $\C$ with radii $r_-$ and $r_+$. In case $r_-=0$, $\mathcal{R}_\sigma$ is the punctuated disc of radius~$r_+$.

For a probability measure $\mu$ on $\R$ we denote by $\mu^{\mathrm{sym}}$ its symmetrization, \ie $\mu^{\mathrm{sym}}(A)\deq\frac{1}{2}\big[ \mu(A) + \mu(-A)\big]$ for any Borel set $A\subset\R$. For $r\in\R^+$, set
\begin{align}
  \mu_{\sigma,r}\deq\mu_{\sigma}^{\mathrm{sym}}\boxplus  \delta_{r}^\mathrm{sym}, \label{090210}
\end{align}
where $\boxplus$ denotes the free additive convolution of probability measures on $\R$; see Subsection~\ref{FAC}.

 Given a probability measure $\mu$ on $\mathbb{R}$,  its {\it Stieltjes transform}, $m_\mu$, on the complex upper half-plane $\mathbb{C}^+\deq\{z\in \mathbb\,:\, \Im z>0\}$ is defined by 
\begin{align}\label{le stieltjes transform}
 m_\mu(z)\deq \int_{\mathbb{R}} \frac{{\rm d} \mu(x)}{x-z}\,,\qquad\qquad z\in \mathbb{C}^+\,.
 \end{align}

\begin{thm}[Single ring theorem, \cite{GKZ11}]\label{global single ring} Assume that Assumption~\ref{ass one} holds and that there is a compactly supported probability measure $\mu_\sigma$ on $[0,\infty)$, which is supported at more than one point, such that~\eqref{levy distance first time} holds. Assume in addition that there are constants $k,k_1>0$ such that
\begin{align}\label{le regularity condition}
 \im m_{\mu_\Sigma}(z)\le k_1 
\end{align}
on $\{z\in\C^+\,:\, \im z> N^{-k}\}$. Then the empirical spectral distribution $\mu_X$ converges weakly (in probability) to a deterministic probability measure $\rho_\sigma$ supported on $\overline{\mathcal{R}_\sigma}$. The limiting measure is absolutely continuous with respect to Lebesgue measure and given by
\begin{align}
\rho_\sigma(w)\,\dd^2w=\frac{1}{2\pi} \Delta_w\Big( \int_{ \R} \log|s| \mu_{\sigma,|w|}({\rm d}s)\Big)\,\dd^2w \,, \qquad\qquad  w\in \mathcal{R}_\sigma\,,\label{081401}
\end{align}
where $\Delta_w=4\partial_w\partial_{\overline{w}}$ is the Laplacian on $\C$ and $\dd^2 w\equiv\dd w\wedge\dd \overline{w}$ is Lebesgue measure on~$\C$.
\end{thm}
\begin{rem}In 
 Theorem~\ref{global single ring}, $U$ and $V$ may be both Haar distributed on $U(N)$ or on $O(N)$. 
 \end{rem}
 \begin{rem}\label{remRV}
 In its original form Theorem~\ref{global single ring} was proved by Guionnet, Krishnapur and Zeitouni in~\cite{GKZ11} under a further assumption on the smallest singular value of the matrix $X-z$, $z\in\C$. This hard-to-check condition was removed by Rudelson and Vershynin in~\cite{RV14} (\cf Theorem~\ref{lem.081003} below), which yields~Theorem~\ref{global single ring}.   
\end{rem}
\begin{rem}
The measure $\rho_\sigma$ was first computed in~\cite{HL00}. It has a direct interpretation
in free probability theory, in fact it is the Brown measure of the free product of a
Haar unitary and an element $\sigma$ on a noncommutative probability space; see~\cite{HL00} for more details.
\end{rem}

\subsection{Local single ring law}

To state our results, we use the following definition on high-probability estimates from~\cite{EKY}. In Appendix~\ref{appendix A} we collect some of its properties.

\begin{defn}\label{definition of stochastic domination}
Let $\mathcal{X}\equiv \mathcal{X}^{(N)}$, $\mathcal{Y}\equiv \mathcal{Y}^{(N)}$ be two sequences of
 nonnegative random variables. We say that~$\mathcal{Y}$ stochastically dominates~$\mathcal{X}$ if, for all (small) $\epsilon>0$ and (large)~$D>0$,
\begin{align}
\P\big(\mathcal{X}^{(N)}>N^{\epsilon} \mathcal{Y}^{(N)}\big)\le N^{-D},
\end{align}
for sufficiently large $N\ge N_0(\epsilon,D)$, and we write $\mathcal{X} \prec \mathcal{Y}$.
 When
$\mathcal{X}^{(N)}$ and $\mathcal{Y}^{(N)}$ depend on a parameter $v\in \mathcal{V}$ (typically an index label or a spectral parameter), then $\mathcal{X}(v) \prec \mathcal{Y} (v)$, uniformly in $v\in \mathcal{V}$, means that the threshold $N_0(\epsilon,D)$ can be chosen independently of $v$. 
\end{defn}

Motivated by~\eqref{081401} we introduce a probability measure $\rho_\Sigma $ on $\C$ by requiring
\begin{align}\label{102910}
 \dd\rho_\Sigma(w)=\frac{1}{2\pi}\Delta_w \Big( \int_\R\log|s| \dd \mu_{\Sigma,|w|}(s)\Big)\dd^2 w\,,\qquad\qquad w\in\C\,,
\end{align}
where
\begin{align}
 \mu_{\Sigma,r}\deq\mu_{\Sigma}^{\mathrm{sym}}\boxplus \delta_{r}^{\mathrm{sym}}\,,\qquad\qquad r\ge0\,,
\end{align}
and $\Delta_w$ is the Laplacian on $\C$ in the sense of distributions.
\begin{rem} The fact that formula~\eqref{102910} defines a probability measure follows from previous work on the subject which we shortly summarize here.

Consider a non-commutative $W^*$-probability space $(\mathcal{M},\tau)$, with $\tau$ a trace. Let $u$ be a Haar unitary element and let $t=t^*$ be $*$-free from $u$ and such that the distribution of $t$, \ie its spectral measure, is given by $\mu_\Sigma$. Let $\widetilde\mu_{\Sigma,w}$ be the spectral measure of $|ut-w\mathrm{id}|$, with $\mathrm{id}$ the unit in $\mathcal{M}$ and $w\in\C$. Then the Brown measure for the product $ut$ is given by the Riesz measure associated to the subharmonic function
\begin{align}
\C\ni w\mapsto\int_\R\log|s| \dd \widetilde\mu_{\Sigma,w}(s)\,,
\end{align}
\cf Section 2 of~\cite{HL00}. Haagerup and Larsen showed in Proposition 3.5 in~\cite{HL00} that $\widetilde\mu_{\Sigma,w}=\mu_{\Sigma,|w|}$. Hence $\rho_\Sigma$ in~\eqref{102910} can be characterized as the Brown measure of $ut$ which by construction is a probability measure.

\end{rem}

The main result of this paper is the following local single theorem in the bulk. Notice that~\eqref{levy distance first time} is not assumed, we only require that $\mathrm{d}_\mathrm{L}(\mu_\Sigma,\mu_\sigma)\le b$, for some small constant $b>0$, for $N$ sufficiently large.

\begin{thm} \label{main theorem} Suppose that Assumption~\ref{ass one} holds. Let $\mu_\sigma$ be a compactly supported probability measure on $[0,\infty)$ which is supported at more than one point. Fix any (small) $\tau>0$ and define
\begin{align}\label{le smaller ring}
 \mathcal{R}_\sigma^\tau\deq \{w\in\C\,:\, r_-+\tau\le |w|\le r_+-\tau\}\subset\mathcal{R}_\sigma\,,
\end{align}
where $r_\pm\equiv r_\pm(\mu_\sigma)$ are given in~\eqref{le regularity condition}. Then there exists a (small) constant $b_0>0$ and $N_0\in\N$, depending only on $\mu_\sigma$ and $S_+$, such that whenever the L\'evy distance $\mathrm{d}_{\mathrm{L}}(\mu_\Sigma,\mu_\sigma)$ satisfies
\begin{align}\label{le assumption on levy distances}
 \sup_{N\ge N_0}\mathrm{d}_{\mathrm{L}}(\mu_\Sigma,\mu_\sigma)\le b\,,
\end{align}
 for some $b\le b_0$, then the following holds. Choose any $w_0\in\mathcal{R}_\sigma^\tau$. Let $f\,:\,\C\rightarrow \R$ be a smooth function such that $\|f\|_\infty\leq C_0$ and $f(z)=0$ for all $|z|\geq C_0$, for some positive constant $C_0$. For $\alpha\in(0,1/2)$
set 
\begin{align}\label{le rescaled function}
f_{w_0}(w)\deq N^{2\alpha} f(N^\alpha(w-w_0))\,.
\end{align}
Then we have for any $\alpha\in (0,1/2)$ that the estimate
\begin{align}
\Big| \frac{1}{N}\sum_{i=1}^N f_{w_0} (\lambda_i(X))- \int_{\mathcal{R}_\sigma} f_{w_0}(w) \dd\rho_{\Sigma}( w) \Big|\prec N^{-1+2\alpha} \|\Delta f\|_{\mathrm{L}^1(\C)} \label{090602}
\end{align} 
holds uniformly in $f$ and in $w_0\in\mathcal{R}_\sigma^\tau$, for $N$ sufficiently large, depending on $\tau$, $S_+$, $\mu_\sigma$ and $C_0$.
\end{thm} 
\begin{rem}
Note that we can choose $\alpha$ in~\eqref{090602}, almost as large as $1/2$ in order to have an effective bound on the error term. Since the typical distance between the eigenvalues in the bulk of the ring $\mathcal{R}_\sigma$ is of order $N^{-1/2}$, our result is optimal, both in terms of range of the exponent $\alpha$ and the error term on the right side of~\eqref{090602}. In particular, this 
improves the recent local single ring theorem of Benaych-Georges in \cite{BG} from scale $(\log N)^{-1/4}$ to the optimal scale $N^{-1/2+\epsilon}$, for any small $\epsilon>0$.
\end{rem}

\begin{rem}
 Theorem~\ref{main theorem} holds with $U,V$ being Haar distributed on either $U(N)$ or on $O(N)$. 
\end{rem}

\begin{rem}\label{le remark stay away from edges}
 Note that $w_0$ in Theorem~\ref{main theorem} is chosen to be the (open) single ring $\mathcal{R}_\sigma$, in particular $w_0$ stays away from the boundary of $\mathcal{R}_\sigma$. In case $r_-=0$, $\mathcal{R}_\sigma$ is a punctuated disc. It has been proved in~\cite{GZ10,BG15} that there are no outliers at an order one distance from~$\mathcal{R}_\sigma$.

 Let $f\,:\,\C\rightarrow\R$ be smooth and supported on $\mathcal{R}^\tau_\sigma$, for some (small) $\tau>0$. Following the proof of Theorem~\ref{main theorem} it is straightforward to verify that~\eqref{090602} also holds with $\alpha=0$ and $f_{w_0}$ replaced with $f$, provided that the support of the function $f$ stays away from the spectral edges, \ie is contained in~$\mathcal{R}_\sigma^\tau$.

 \end{rem}

 The following corollary of Theorem~\ref{main theorem} expresses the speed of convergence in 
 the single ring theorem on the macroscopic scale. 
 
  \begin{cor}\label{le corollary xh} Under the conditions and with the notations  of Theorem~\ref{main theorem}, we have that
  \begin{align}\label{le xh}
   \bigg|\frac{1}{N}\sum_{i=1}^N f(\lambda_i(X))-  \int_{\mathcal{R}_\sigma} f(w) \dd\rho_{\sigma}(w) \bigg|\prec {\|\Delta f\|_{\mathrm{L}^1(\C)}}\Big(\frac{1}{N}+b\Big)\,,
  \end{align}
 uniformly for any function $f$ supported in $\mathcal{R}_\sigma^\tau$ with a bound $\|f\|_\infty\le C_0$, for~$N$ sufficiently large, depending on~$\tau$, $S_+$, $\mu_\sigma$ and $C_0$.
 
  \end{cor}
  \begin{rem}
 In~\eqref{le xh} the measure $\rho_\sigma$ is given by~\eqref{081401}. By Theorem~4.4 and Corollary~4.5 of~\cite{HL00}, the measure $\rho_{\sigma}$ is absolutely continuous on $\C\backslash\{0\}$ with respect to Lebesgue measure. Moreover, it satisfies $\rho_\sigma(\{0\})=\mu_\sigma(\{0\})$. (In case $\mu_\sigma(\{0\})>0$, we have $r_-=0$.) Note however that we have to exclude the point $w=0$ in our results since it is outside~$\mathcal{R}_\sigma$.
  \end{rem}

\begin{rem}\label{remove regularity}  Note that 
in Theorem~\ref{main theorem} and Corollary~\ref{le corollary xh}
we do not require any regularity assumption on the measure $\mu_{\sigma}$,  we even allow for atoms in $\mu_\sigma$.
 In particular, sending $b\rightarrow 0$, as $N\rightarrow \infty$, Corollary~\ref{le corollary xh} also implies that Assumption~\ref{ass one} and~\eqref{levy distance first time} together imply $\mathrm{d}_{\mathrm{L}}(\rho_\Sigma,\rho_\sigma)\rightarrow 0$, as $N\rightarrow\infty$, thus removing the regularity condition~\eqref{le regularity condition} in the bulk from the single ring theorem, this answers a question  in \cite[Remark 2]{GKZ11}.

\end{rem}

\subsection{Summary of previous results}

The first single ring theorem was established by Feinberg and Zee  for a class of unitary invariant ensemble in \cite{FZ97}, but without full 
rigor. The complete mathematical proof was given by Guionnet, Krishnapur and Zeitouni~\cite{GKZ11};
see also Remarks~\ref{remRV} and~\ref{remove regularity} for relaxing some conditions.  
 
 In  spirit of the Wigner ensemble for the Hermitian case, the Ginibre ensemble can also be naturally extended by 
 considering arbitrary i.i.d. entries; however,  the unitary invariance property is lost in this generalization.
 Starting from the work of Girko \cite{G84}, until the final result of Tao and Vu \cite{TV10} with the least moment assumption, 
 there have been many works devoted in proving circular law for general distribution. 
 We refer to the survey \cite{BC12} for more references in this direction.
 A prominent idea called {\it Hermitization} was introduced by Girko in \cite{G84}.
 This method translates spectral distribution problems of a non-Hermitian matrix
 to those of a Hermitian matrix (of double dimension),  whose spectral properties can be studied with more established techniques.

Similarly to Wigner's original semicircle law, the single ring theorem establishes weak convergence of the spectral distribution, \ie 
it captures the density of eigenvalues on the global scale.  Since the typical distance between nearby eigenvalues is very small, of order $N^{-1/2}$,
 it is natural to ask whether the empirical density can also  be approximated by the 
 deterministic limit density  on some local scale. Ideally, such {\it local law} should  hold on the smallest possible scale, \ie just 
 above  the scale  $N^{-1/2}$. In the Hermitian case, the local laws for Wigner and related ensembles have been extensively 
studied in the recent years, see {\it e.g.}~\cite{E15} for a survey and references therein;  the optimal local scale 
has been first achieved in \cite{ESY}.

With the aid of Girko's Hermitization, local laws for non-Hermitian matrices can be obtained via studying the local law for certain Hermitian matrices. With this strategy,
the local circular law on optimal scale  was established in  the series of works  Bourgade, Yau and Yin \cite{BYY, BYY2} and Yin \cite{Yin}.   The first local single ring theorem was obtained by Benaych-Georges in \cite{BG}, down to the scale $(\log N)^{-\frac{1}{4}}$, by proving the matrix subordination for Girko's Hermitization of $X$ in (\ref{091101}), \cf (\ref{def of Hw}). The strategy of matrix subordination
was originally introduced by Kargin in \cite{Kargin}  for proving a local law in the  additive  matrix model $A+UBU^*$, where~$A$ and~$B$ are deterministic Hermitian matrices and $U$ is a Haar unitary.
 This additive model shares certain similarities with the Hermitization of the model $X=U\Sigma V^*$, but
 the latter has a block structure and thus we call it {\it block additive model} (\cf (\ref{052902})).  Recently, in \cite{BES15b,BES15c,BES17}, we obtained the local law of the additive model $A+UBU^*$  on the optimal scale. The approach developed in
 these works opens up a path to treat  the optimal local law in the block additive model, 
  hence also sheds  light on  the optimal local single ring theorem. The key difference  is that in the block additive model
  the Haar unitary matrices provide only a randomized $U(N)\times U(N)$ symmetry instead of the full  $U(2N)$ symmetry.
  In particular, the coupling  between the blocks is deterministic, so the mixing mechanism is much weaker.
   A more detailed overview of the proof strategy and the difficulties will be given in 
Section~\ref{sec:outline}.

\subsection{Notational conventions}
We use the symbols $O(\,\cdot\,)$ and $o(\,\cdot\,)$ for the standard big-O and little-o notation. We use~$c$ and~$C$ to denote strictly positive constants that do not depend on~$N$. Their values may change from line to line.

We denote by $M_N(\C)$ the set of $N\times N$ matrices over $\C$. For $A\in M_N(\C)$, we denote by $\|A\|$ its operator norm and by $\|A\|_2$ its Hilbert-Schmidt norm. The matrix entries of $A$ are denoted by $A_{ij}$.

 Let $\mathbf{g}=(g_1,\ldots, g_N)$ be a real or complex Gaussian vector. We write $\mathbf{g}\sim \mathcal{N}_{\mathbb{R}}(0,\sigma^2I_N)$ if $g_1,\ldots, g_N$ are independent and identically distributed (i.i.d.) $N(0,\sigma^2)$ normal variables; and we write $\mathbf{g}\sim \mathcal{N}_{\mathbb{C}}(0,\sigma^2I_N)$ if $g_1,\ldots, g_N$ are i.i.d.\ $N_{\mathbb{C}}(0,\sigma^2)$ variables, where $g_i\sim N_{\mathbb{C}}(0,\sigma^2)$ means that $\Re g_i$ and $\Im g_i$ are independent $N(0,\frac{\sigma^2}{2})$ normal variables. 

We use double brackets to denote index sets, \ie for $n_1, n_2\in\R$, $\llbracket n_1,n_2\rrbracket\deq [n_1, n_2] \cap\Z$.

\medskip

{\emph{Acknowledgment:} Part of this work was accomplished when Z.-G.\ B. and K.\ S. were  working at
 IST Austria with the support of ERC Advanced Grant RANMAT No.\ 338804. Support and hospitality are gratefully acknowledged. We thank an anonymous referee for very useful comments and suggestions.

\section{Preliminaries and main technical task}

\subsection{Free additive convolution}\label{FAC}
We recall some basic notions and results for the free
additive convolution. We follow the notational conventions in our previous paper~\cite{BES15}.

Let $\mu$ be a Borel probability measure on $\R$ and recall its Stieltjes transform $m_{\mu}$ defined in~\eqref{le stieltjes transform}. Note that $m_{\mu}\,:\,\C^+\rightarrow \C^+$ is an analytic function such that
\begin{align}\label{le limit to be a probablity measure}
 \lim_{\eta\nearrow\infty} \ii \eta\, m_\mu(\ii\eta)=-1\,.
\end{align}
Conversely, if $m\,:\, \C^+\rightarrow \C^+$ is an analytic function such that $\lim_{\eta\nearrow\infty} \ii \eta\, m(\ii\eta)=-1$, then~$m$ is the Stieltjes transform of a probability measure $\mu$.

Given a Borel probability measure $\mu$ on $\R$, let $F_\mu$ be the {\it negative reciprocal Stieltjes transform} of $\mu$,
\begin{align}\label{le F definition}
 F_{\mu}(z)\deq -\frac{1}{m_{\mu}(z)}\,,\qquad \qquad z\in\C^+\,.
\end{align}
Observe that
 \begin{align}\label{le F behaviour at infinity}
\lim_{\eta\nearrow \infty}\frac{F_{\mu}(\ii\eta)}{\ii\eta}=1\,,
\end{align}
as follows from~\eqref{le limit to be a probablity measure}. Note that $F_\mu$ is analytic on~$\C^+$ with nonnegative imaginary part.

The {\it free additive convolution} is the symmetric binary operation on Borel probability measures on~$\R$ characterized by the following result.
\begin{thm}[Theorem 4.1 in~\cite{BB}, Theorem~2.1 in~\cite{CG}]\label{le prop 1}
Given two Borel probability measures, $\mu_1$ and $\mu_2$, on $\R$, there exist unique analytic functions, $\omega_1,\omega_2\,:\,\C^+\rightarrow \C^+$, such that,
 \begin{itemize}[noitemsep,topsep=0pt,partopsep=0pt,parsep=0pt]
  \item[$(i)$] for all $z\in \C^+$, $\im \omega_1(z),\,\im \omega_2(z)\ge \im z$, and
  \begin{align}\label{le limit of omega}
  \lim_{\eta\nearrow\infty}\frac{\omega_1(\ii\eta)}{\ii\eta}=\lim_{\eta\nearrow\infty}\frac{\omega_2(\ii\eta)}{\ii\eta}=1\,;
  \end{align}
  \item[$(ii)$] for all $z\in\C^+$, 
  \begin{align}\label{le definiting equations}
   F_{\mu_1}(\omega_{2}(z))=F_{\mu_2}(\omega_{1}(z))\,,\quad \omega_1(z)+\omega_2(z)-z=F_{\mu_1}(\omega_{2}(z))\,.
  \end{align}
 \end{itemize}
\end{thm}

It follows from~\eqref{le limit of omega} that the analytic function $F\,:\,\C^+\rightarrow \C^+$ defined by
\begin{align}\label{le kkv}
 F(z)\deq F_{\mu_1}(\omega_{2}(z))=F_{\mu_2}(\omega_{1}(z))\,,
\end{align}
satisfies the analogue of~\eqref{le F behaviour at infinity}. Thus $F$ is the negative reciprocal Stieltjes transform of a probability measure $\mu$, called the free additive convolution of $\mu_1$ and $\mu_2$, denoted by $\mu\equiv\mu_1\boxplus\mu_2$. The functions $\omega_1$ and $\omega_2$ are referred to as the {\it subordination functions} and $F$ is said to be subordinated to~$F_{\mu_1}$, respectively to $F_{\mu_2}$. The subordination phenomenon
was first noted by Voiculescu~\cite{Voi93} in a generic situation
and extended to full generality by Biane~\cite{Bia98}. To exclude trivial shifts of measures, we henceforth assume that both,~$\mu_1$ and~$\mu_2$, are supported at more than one point. Then the analytic functions $F$, $\omega_1$ and $\omega_2$ extend continuously to the real line; see Theorem 2.3~\cite{Bel1} or Theorem 3.3~\cite{Bel}. We use the same notation for their extensions~to~$\C^+\cup\R$.

\subsection{The limiting measure $\mu_{\sigma,r} $ }  

Recall the definitions $\mu_{\Sigma,r}\deq\mu_{\Sigma}^{\mathrm{sym}}\boxplus \delta_{r}^{\mathrm{sym}}$ and $\mu_{\sigma,r}\deq\mu_{\sigma}^{\mathrm{sym}}\boxplus  \delta_{r}^\mathrm{sym}$  from~\eqref{090210}. In this subsection, we will always assume that $\mu_\Sigma$ and $\mu_\sigma$ satisfy Assumption~\ref{ass one}. For sake of simplicity of notation, we abbreviate in this subsection 
\begin{align}\label{090215}
 \mu_1\equiv \mu_{\sigma}^{\mathrm{sym}}\,,\qquad \qquad\mu_2\equiv \delta_{r}^{\mathrm{sym}}\,.
\end{align}
The negative reciprocal Stieltjes transform of  $\mu_2=\delta_{r}^{\mathrm{sym}}$ is found to be
\begin{align}\label{le F of delta sym}
 F_{\mu_2}(z)=z-\frac{r^2}{z}\,,\qquad\qquad z\in\C^+\,.
\end{align}
Substituting~\eqref{le F of delta sym} into~\eqref{le definiting equations}, we obtain
\begin{align*}
F_{\mu_1}(\omega_2(z))=F_{\mu_2}(\omega_1(z))=F_{\mu_1}(\omega_2(z))-\omega_2(z)+z-\frac{r^2}{F_{\mu_1}(\omega_2(z))-\omega_2(z)+z}\,.
\end{align*}
Solving the above equation for $F_{\mu_1}(\omega_2(z))$ we conclude that the subordination function $\omega_2(z)$ is the unique solution to
\begin{align}
F_{\mu_1}(\omega_2(z))-\omega_2(z)=-z-\frac{r^2}{\omega_2(z)-z}\,,\qquad\qquad z\in\C^+\,, \label{090410}
\end{align}
subject to the condition $\im\omega_2(z)\ge \im z$. Comparing once more with~\eqref{le definiting equations} we immediately find that the other subordination function is given by
\begin{align}\label{le omega1 after knowing omega2}
 \omega_1(z)=-\frac{r^2}{\omega_2(z)-z}\,,\qquad \qquad z\in\C^+\,.
\end{align}

The analysis of the measure $\mu_{\sigma,r}=\mu_1\boxplus\mu_2$ thus reduces to the analysis of~\eqref{090410} for $\omega_2$. We first derive upper and lower bound on $\omega_2(z)$. For the purpose of proving Theorem~\ref{main theorem} it will suffice to consider $z\in\{ \ii\eta\,:\,\eta\ge 0\}$. Since $\mu_1$ and $\mu_2$ are symmetric, we  have $\omega_2(\ii\eta)=-\overline{\omega_2(\ii\eta)}$, \ie $\omega_2(\ii\eta)$
 and $\omega_1(\ii\eta)$ are both fully imaginary. This  simplifies our analysis; while detailed quantitative 
properties of the full measure $\mu_{\sigma,r}$ are still poorly understood, we now have a good control on it near zero, hence
on its Stieltjes transform along the imaginary axis.  The main result, formulated in Theorem~\ref{le corollary of 0 in the bulk}
below,  is that   the  subordination functions  are bounded from below and above on the imaginary axis without  any condition on $\mu_\sigma$.
This theorem is the key input that enables us to dispense with  the regularity  condition in the single ring theorem; see Remark~\ref{remove regularity}.

\begin{thm}[Bounds on subordination functions]\label{le corollary of 0 in the bulk}
We assume that the support of $\mu_\sigma$ contains more than one point, equivalently,  that $r_-<r_+$.
 Let $\mu_1=\mu_{\sigma}^{\mathrm{sym}}$ and $\mu_2=\delta_{r}^{\mathrm{sym}}$ for some $r>0$.
 Fix $\eta_{\rm M}<\infty$ and a (small) $\tau>0$. Set 
 $$
 J\deq [r_-+\tau,r_+-\tau]\,.
 $$
There exist constants $c\equiv c(\mu_1,\tau,\eta_{\rm M})>0$ and $C\equiv C(\mu_1,\tau,\eta_{\rm M})<\infty$ such that
 \begin{align}
  &\sup_{r\in J}\sup_{\eta \in [0,\eta_{\rm M}]} |\omega_1(\mathrm{i}\eta)|\leq C\,,\qquad\sup_{r\in J} \sup_{\eta \in [0,\eta_{\rm M}]} |\omega_2(\mathrm{i}\eta)|\leq C\,, \\
 &\inf_{r\in J}\inf_{\eta \in [0,\eta_{\rm M}]} \Im \omega_1(\mathrm{i}\eta)\geq c\,,\qquad \inf_{r\in J}\inf_{\eta \in [0,\eta_{\rm M}]}\Im \omega_2(\mathrm{i}\eta)\geq c\,,
 \end{align} 
 and
 \begin{align}\label{mbound}
   \inf_{r\in J}\inf_{\eta \in [0,\eta_{\rm M}]}|m_{\mu_1\boxplus\mu_2}(\ii\eta)|\ge c\,,\qquad\sup_{r\in J} \sup_{\eta \in [0,\eta_{\rm M}]} 
   |m_{\mu_1\boxplus\mu_2}(\ii\eta)|\le C\,.
 \end{align}
\end{thm}
\begin{rem}
 By~\eqref{mbound}, the measure $\mu_1\boxplus\mu_2$ has a positive and bounded density at $E=0$. In particular, $E=0$ is in the bulk of the measure $\mu_1\boxplus\mu_2$, as defined in Definition~\ref{def of bulk} below.
\end{rem}

The proof of Theorem~\ref{le corollary of 0 in the bulk} is quite technical and independent of the main line of the argument, so
we give it in Section~\ref{le section of 0 in the bulk}. In the subsequent sections, we will mainly rely on the following corollary of Theorem~\ref{le corollary of 0 in the bulk}. Let $m_{\Sigma,r}(z)$ be the Stieltjes transform of $\mu_{\Sigma, r}$; see~\eqref{090210}. 
\begin{cor}\label{cor:bounds}
Fix $\eta_{\rm M}<\infty$ and a (small) $\tau>0$. Then there are constants $C\equiv C(\mu_\sigma^{\mathrm{sym}},\tau,\eta_{\rm M})$, $c\equiv(\mu_\sigma^{\mathrm{sym}},\tau,\eta_{\rm M})$ and a threshold $N_0\equiv N_0(\mu_\sigma^{\mathrm{sym}},\tau,\eta_{\rm M})$ such that the conclusions in Theorem~\ref{le corollary of 0 in the bulk} 
hold with $\mu_1=\mu_\Sigma^{\mathrm{sym}}$ and $\mu_2=\delta_{r}^\mathrm{sym}$, for $N\ge N_0$.
\end{cor}
\begin{proof} This follows directly from the continuity of the subordination functions with respect to the L\'evy distance (see
Lemma 5.1 of~\cite{BES15}), from Theorem~\ref{le corollary of 0 in the bulk} and from~\eqref{le assumption on levy distances}.
\end{proof}

\subsection{Key technical inputs}\label{sec:key}

Following Girko's hermitization technique~\cite{G84}, we introduce for any $w\in \C$ the $2N\times 2N$ Hermitian matrix 
\begin{align}
H^w\deq\left(
\begin{array}{ccccc}
0 &X-w\\
X^*-w^* & 0
\end{array}
\right)\,. \label{def of Hw}
\end{align}
The main advantage of working with $H^w$ is that it is self-adjoint and we thus have a functional calculus at disposal. For any function $g\in C^2(\C)$, an application of Green's theorem reveals~that
\begin{align}
\frac{1}{N} \sum_{i=1}^N g(\lambda_i(X))=\frac{1}{2\pi }\int_{\C} (\Delta g)(w) \; \Big(\frac{1}{2N}\text{Tr}\log |H^w|\Big)\; \dd^2 w\,, \label{103101}
\end{align}
which is a manifestation of $\log|\cdot|$ being the Coulomb potential in two dimensions. The following identity, first used in this context by~\cite{TV}, allows us to efficiently deal with the right side of~\eqref{103101}. For any (large) $K>0$,
\begin{align}
\frac{1}{2N} \text{Tr} \log | H^w|=\frac{1}{2N} \text{Tr} \log |(H^w-\mathrm{i}K) |- \Im \int_0^K m^w(\mathrm{i}\eta) {\rm d}\eta\,, \label{081010bis}
\end{align}
with $|w|>0$, where $m^w(z)$, $z\in\C^+$, is the Stieltjes transform of the spectral distribution of $H^w$. For very large $K$ the first term on the right side of~\eqref{081010bis} is elementary to control, we hence focus on the second term. Due to the block structure of~$H^w$, the eigenvalues
come in pairs $\pm\lambda_i^w$, $i\in\llbracket 1,N\rrbracket$, where $0\leq \lambda_1^w\leq \ldots\leq \lambda_N^w$ are the non-negative eigenvalues. With these notations $m^w$ is given~by
\begin{align*}
m^w(z)\deq\frac{1}{2N} \sum_{i=1}^N\Big(\frac{1}{\lambda_i^w-z}+\frac{1}{-\lambda_i^w-z}\Big)=\frac{1}{N}\sum_{i=1}^N\frac{\lambda_i^w}{(\lambda_i^w)^2-z^2}\,,\qquad z\in\C^+\,.
\end{align*}

Recall the notation~$m_{\Sigma,|w|}$ for the Stieltjes transform of  $\mu_{\Sigma, |w|}$; \cf~\eqref{090210}. The following result is the main technical input for the proof of Theorem~\ref{main theorem}. Recall $\mathcal{R}_\sigma^\tau$ from~\eqref{le smaller ring}.

\begin{thm}[Local law for $H^{w}$] \label{lem.081002} Under the conditions and with the notations of Theorem~\ref{main theorem}, the estimate
\begin{align}
\sup_{w\in\mathcal{R}_\sigma^\tau}\big|m^w(\ii\eta)-m_{\Sigma,|w|}(\ii\eta)\big|\prec\frac{1}{N\eta}\,, \label{081020}
\end{align}
holds uniformly in $\eta>0$, for $N$ sufficiently large, depending on $\tau$, $S_+$ and~$\mu_\sigma$.
\end{thm}
This result controls $|m^w(\ii\eta)-m_{\Sigma,|w|}(\ii\eta)|$ along the positive imaginary axis. Note that the error estimate on the right side of~\eqref{081020} is effective when $\eta$ is chosen just above the local scale, \ie when $\eta> N^{-1+\gamma}$, for any small $\gamma>0$. For even smaller $\eta>0$,~\eqref{081020} yields the upper bound $|m^w(\ii\eta)|\prec (N\eta)^{-1}$ which improves the trivial deterministic bound~$|m^w(\ii\eta)|\le \eta^{-1}$ by a factor~$N^{-1}$. Theorem~\ref{lem.081002} is used to control the integrand in the second term on the right side of~\eqref{081010bis} for $\eta\gtrsim N^{-1}$. On very short scales, the behavior of $m^w(\ii\eta)$, $\eta\lesssim N^{-1}$, is essentially random  and determined by the smallest (in absolute value) eigenvalues of $H^w$. The following estimate on~$\lambda_1^w$,
 proved by Rudelson and Vershynin in \cite{RV14}, is then used to control the integrand of the second term on the right side of~\eqref{081010bis}
for very small $\eta\lesssim N^{-1}$.
\begin{thm}[Theorem~1.1 and Theorem~1.2 in~\cite{RV14} ] \label{lem.081003}  There exist positive numerical constants $c>0$ and $C<\infty$, such that
\begin{align}
\mathbb{P}\Big(\lambda_1^w\leq \frac{t}{|w|}  \Big)\leq \Big(\frac{t}{|w|}\Big)^cN^C\,, \label{081021}
\end{align}
uniformly in $t>0$, for all $N\in\N$.
\end{thm}
\begin{rem} In the orthogonal case,~\eqref{081021} holds, for $N$ sufficiently large, when the matrix~$\Sigma$ is away from the identity; see Theorem 1.2 in \cite{RV14}. In this case the constants $c$, $C$ and the threshold for $N$ in~\eqref{081021} depend on $S_+$ and $\mu_\sigma$. Indeed,~\eqref{le assumption on levy distances} and the assumption that the support of $\mu_\sigma$ contains more than one point imply that $\Sigma$ is separated away from the~identity.
\end{rem}

In Section~\ref{s. proof of main theorem}, we will choose $g$ in~\eqref{103101} to be the rescaled function $f_{w_0}(\cdot)=N^{2\alpha} f(N^\alpha(\cdot-w_0)$; see~\eqref{le rescaled function}. The local law in~\eqref{081020} together with~\eqref{081021} (with $t/|w|\ll N^{-1}$) will allow us to choose $\alpha\in(0,1/2)$ as is asserted in Theorem~\ref{main theorem}. The details of the proof of Theorem~\ref{main theorem}, assuming 
Theorem~\ref{lem.081002}, are carried out in Section~\ref{s. proof of main theorem}. Our main task then is to prove Theorem~\ref{lem.081002}. Actually, we will establish the local law in a more general setting; \cf Theorem~\ref{thm.081501}. This will be accomplished in Sections ~\ref{s.local law for block model}-\ref{s.strong law} and we will separately outline the main ideas 
of this proof in Section~\ref{sec:outline}.} We begin with the proof of Theorem~\ref{le corollary of 0 in the bulk} in the next section.

\section{Proof of Theorem~\ref{main theorem} and Corollary~\ref{le corollary xh}}\label{s. proof of main theorem}
In this section, we prove Theorem~\ref{main theorem} and Corollary~\ref{le corollary xh}, with the aid of Theorems~\ref{lem.081002} and~\ref{lem.081003}. The use of Girko's hermitized matrices to derive local laws is a standard argument, see \eg~\cite{BYY,TV} for related models. Following~\cite{TV}, we use the 
identity~\eqref{081010} below to link the log-determinant of~$H^w$ with the Stieltjes transform $m^{w}$.
 
\begin{proof}[Proof of Theorem~\ref{main theorem}]   For any $\zeta\in \C$, we  denote
\begin{align}
w\equiv w(\zeta)\deq w_0+N^{-\alpha}\zeta\,. \label{102940}
\end{align}
Given $f\,:\,\C\rightarrow \R$ satisfying the assumption of Theorem~\ref{main theorem}, we introduce the domain
\begin{align}
\mathcal{D}_{w_0}(\alpha)\equiv  \mathcal{D}_{w_0}(\alpha,f):= \Big\{ \tilde w: N^{\alpha}(\tilde w-w_0)\in \text{supp}(f)\Big\}. \label{110201}
\end{align}
According to~\eqref{102940}, $w\in \mathcal{D}_{w_0}(\alpha)$ is equivalent to $\zeta\in \text{supp}(f)$, in particular $|\zeta|\le C$ as $f$ is compactly supported. 
 Recall the notation $f_{w_0}(\cdot)$ from Theorem~\ref{main theorem}. Using~(\ref{103101}), we rewrite
\begin{align}
\frac{1}{N}\sum_i f_{w_0} (\lambda_i(X))= \frac{1}{2\pi }N^{2\alpha}\int_{\mathbb{C}} (\Delta f)(\zeta) \; \Big(\frac{1}{2N} \text{Tr}\log |H^w|\Big)\; {\rm d }^2\zeta\,. \label{090201}
\end{align}
Recalling the definitions in~(\ref{090210}) and ~(\ref{102910}), we  also have
\begin{align}
\int_{\mathbb{C}} f_{w_0}(w) \rho_{\Sigma}({\rm d}^2 w) &=\frac{1}{2\pi} \int_{\mathbb{C}} f_{w_0}(w) \Delta_w\Big( \int_{\mathbb{R}}\log|u| \mu_{\Sigma, |w|}({\rm d}u)\Big) {\rm d}^2 w\nonumber\\
&= \frac{1}{2\pi} N^{2\alpha}\int_{\mathbb{C}} (\Delta f)(\zeta) \Big( \int_{\mathbb{R}}\log|u| \mu_{\Sigma, |w|}({\rm d}u)\Big) {\rm d}^2 \zeta. \label{090202}
\end{align}
Hence, we can write
\begin{multline}
\frac{1}{N} \sum_{i} f_{w_0} (\lambda_i(X))- \int_{\mathbb{C}} f_{w_0}(w) \rho_\Sigma({\rm d}^2 w) \\
=\frac{1}{2\pi} N^{2\alpha} \int_{\mathbb{C}} (\Delta f) (\zeta) \bigg(\frac{1}{2N} \text{Tr} \log |H^w|-\int_{\mathbb{R}} \log |u| \; \mu_{\Sigma,|w|} ({\rm d} u)\bigg) {\rm d}^{2} \zeta. \label{103025}
\end{multline}
We next use the following observation due to~\cite{TV}, Section~8. For any (large) $K>0$ and $|w|>0$, we have
\begin{align}
\frac{1}{2N} \text{Tr} \log | H^w|=\frac{1}{2N} \text{Tr} \log |(H^w-\mathrm{i}K) |- \Im \int_0^K m^w(\mathrm{i}\eta) {\rm d}\eta\,. \label{081010}
\end{align}
Analogously, we can also write, with the same $K$,
\begin{align}
\int_{\mathbb{R}}\log|u| \; \mu_{\Sigma,|w|}({\rm d}u)=\int_{\mathbb{R}}\log|u-\mathrm{i}K| \; \mu_{\Sigma, |w|}({\rm d}u)-\Im \int_0^K m_{\Sigma,|w|}(\mathrm{i}\eta) {\rm d}\eta\,. \label{081011}
\end{align}

Choosing $K$ sufficiently large, say $K=N^L$ for some large constant~$L$, it is easy to see that
\begin{align}
\Big|\frac{1}{2N} \text{Tr} \log |(H^w-\mathrm{i}K) |-\int_{\mathbb{R}}\log|u-\mathrm{i}K| \mu_{\Sigma,|w|}({\rm d}u)\Big|\ll \frac{1}{N} \label{103026}
\end{align}
holds uniformly in $w\in \mathcal{D}_{w_0}(\alpha)$. Here we used the fact that $\|H^w\|\leq C$
 for some positive constant $C$, under \cf Assumption~\ref{ass one}.
The uniformity in $w$ can be guaranteed by the fact that $\mathcal{D}_{w_0}(\alpha)$ lies in a ball of finite (in fact  $CN^{-\alpha}$) radius
since  $f$ is compactly supported. 
Hence, it suffices to show 
\begin{align}
\Big| \int_{\mathbb{C}} (\Delta f)(\zeta) \; \Big(\Im \int_0^{N^L} \big(m^w(\mathrm{i}\eta)-m_{\Sigma,|w|}(\mathrm{i}\eta)\big) {\rm d}\eta\Big)  {\rm d}^{ 2} \zeta\Big| \prec \frac{\|\Delta f\|_{\mathrm{L}^1(\C)}}{N}.  \label{102920}
\end{align}
To show~(\ref{102920}), we decompose the integral with respect to $\eta$ into two parts:
\begin{align}\label{two parts}
\int_0^{N^L}= \int_0^{N^{-L_1}}+\int_{N^{-L_1}}^{N^{L}} \,,
\end{align}
for sufficiently large constants $L_1>1$ and $L>0$ to be chosen below.  To control the first part, we use~(\ref{081021}), while for the second part we use~(\ref{081020}).

First,  using the upper bound of  $m_{\Sigma,|w|}(\ii \eta)$ (\cf Corollary~\ref{cor:bounds}), we obtain 
\begin{align}
\Big|\int_0^{N^{-L_1}} \Im \; m_{\Sigma,|w|}(\mathrm{i}\eta) {\rm d}\eta\Big|\leq  \frac{1}{N}\,, \label{102931}
\end{align}
for  $L_1> 1$, 
uniformly in $w\in \mathcal{D}_{w_0}(\alpha)$. Hence, we have
\begin{align}
\Big| \int_{\mathbb{C}} (\Delta f)(\zeta) \; \Big( \int_0^{N^{-L_1}} \Im m_{\Sigma,|w|}(\mathrm{i}\eta) {\rm d}\eta\Big)  {\rm d}^{ 2} \zeta\Big| \le C\frac{\|\Delta f\|_{\mathrm{L}^1(\C)}}{N}. \label{103021}
\end{align} 
In addition, we observe that 
\begin{align}
\mathbb{P}\Big(\Big| \int_{\mathbb{C}} (\Delta f)(\zeta) \;& \Big( \int_0^{N^{-L_1}} \Im m^w(\mathrm{i}\eta) {\rm d}\eta\Big)  {\rm d}^2 \zeta\Big|  > \frac{\|\Delta f\|_{\mathrm{L}^1(\C)}}{N}\Big)\nonumber\\
&\leq \frac{ N }{\|\Delta f\|_{\mathrm{L}^1(\C)}} \mathbb{E}\Big| \int_{\mathbb{C}} (\Delta f)(\zeta) \; \Big( \int_0^{N^{-L_1}} \Im m^w(\mathrm{i}\eta) {\rm d}\eta\Big)  {\rm d}^2 \zeta\Big|\nonumber\\
& \leq \frac{ N }{\|\Delta f\|_{\mathrm{L}^1(\C)}}\int_{\mathbb{C}} \big|(\Delta f)(\zeta)\big|\;\mathbb{E}\Big( \int_0^{N^{-L_1}} \frac{\eta}{(\lambda_1^w)^2+\eta^2} {\rm d}\eta\Big)  {\rm d}^2 \zeta\,. \label{103005}
\end{align}
Note that 
\begin{align*}
 \mathbb{E}\Big( \int_0^{N^{-L_1}} &\frac{\eta}{(\lambda_1^w)^2+\eta^2} {\rm d}\eta\Big)=\frac{1}{2} \mathbb{E} \log \Big( 1+(N^{L_1}\lambda_1^w)^{-2}\Big)\nonumber\\
  &= \frac{1}{2} \int_0^\infty \mathbb{P}\Big(\log \Big( 1+(N^{L_1}\lambda_1^w)^{-2}\Big)\geq s \Big) {\rm d} s\nonumber\\
  &= \frac{1}{2} \int_0^\infty \mathbb{P}\Big(\lambda_1^w\leq N^{-L_1}({\rm e}^s-1)^{-\frac{1}{2}} \Big) {\rm d} s\nonumber\\
  &=\frac{1}{2} \Big(\int_0^{N^{-L_1}}+\int_{N^{-L_1}}^1+\int_{1}^\infty \Big)\mathbb{P}\Big(\lambda_1^w\leq N^{-L_1}({\rm e}^s-1)^{-\frac{1}{2}} \Big) {\rm d} s\,. 
\end{align*}
For the   first  integral,  we use the trivial bound $\mathbb{P}(\,\cdot\,)\leq 1$ to obtain
\begin{align}
\int_0^{N^{-L_1}} \mathbb{P}\Big(\lambda_1^w\leq N^{-L_1}({\rm e}^s-1)^{-\frac{1}{2}} \Big) {\rm d} s\leq  N^{-L_1}\,. \label{103001}
\end{align}
For the second part of the integral,  using the crude bound $({\rm e}^s-1)^{-\frac{1}{2}}\leq s^{-\frac{1}{2}}\leq N^{\frac{L_1}{2}}$, $s\in[N^{-L_1},1]$, and~(\ref{081021}), we estimate	
\begin{align*}
&\int_{N^{-L_1}}^1 \mathbb{P}\Big(\lambda_1^w\leq N^{-L_1}({\rm e}^s-1)^{-\frac{1}{2}} \Big) {\rm d} s\leq  \int_{N^{-L_1}}^1 \mathbb{P}\Big(\lambda_1^w\leq N^{-\frac{L_1}{2}}\Big) {\rm d} s  \le N^{-\frac{cL_1}{2}+C}\,,
\end{align*}
for some constants $c>0$ and $C<\infty$, for $N$ sufficiently large. For the third part, using ${\rm e}^s-1>  \frac{1}{2}  {\rm e}^s$, $s>1$, and~(\ref{081021}), we have  
\begin{align}
\int_{1}^\infty \mathbb{P}\Big(\lambda_1^w\leq N^{-L_1}({\rm e}^s-1)^{-\frac{1}{2}} \Big) {\rm d} s & \leq  
\int_{1}^\infty \mathbb{P}\Big(\lambda_1^w\leq  \sqrt{2} N^{-L_1}{\rm e}^{-\frac{s}{2}} \Big) {\rm d} s \nonumber\\
&  \le \frac{N^{-cL_1+C}}{2}\int_{1}^\infty {\rm e}^{-\frac{cs}{2}} {\rm d} s \le N^{-cL_1+C}\,, \label{103003}
\end{align}
for some constants $c>0$ and $C<\infty$. Combining~(\ref{103001})-(\ref{103003}), we obtain that there are positive constants $c'>0$ and $C'$, independent of $L_1$ such that
\begin{align}
\mathbb{E}\Big( \int_0^{N^{-L_1}} \frac{\eta}{(\lambda_1^w)^2+\eta^2} {\rm d}\eta\Big)\leq N^{-c'L_1+C'}\,, \label{103004}
\end{align}
for $N$ sufficiently large. In fact, the bound~\eqref{103004} is uniform in $w\in \mathcal{D}_{w_0}( \alpha)$ since the constants~$c$ and~$C$ in Theorem~\ref{lem.081003} are uniform in $t$ and $w$. Plugging~(\ref{103004}) into~(\ref{103005}), yields
\begin{align}
\mathbb{P}\bigg(\Big| \int_{\mathbb{C}} (\Delta f)(\zeta) \; \Big( \int_0^{N^{-L_1}} \Im m^w(\mathrm{i}\eta) {\rm d}\eta\Big)  {\rm d}^2 \zeta\Big| \geq \frac{\|\Delta f\|_{\mathrm{L}^1(\C)}}{N}\bigg)\leq N^{-c'L_1+C'+1}\,, \label{103022}
\end{align}
for $N$ sufficiently large (independent of $L_1$). Choosing $L_1$ large enough, the contribution of the first integral in \eqref{two parts} to~\eqref{102920} is within the claimed error. 
 
 To control the contributions from the second integral in~\eqref{two parts}, for any (large) constant $L_1$, we apply the local law for $m^w$ 
 in~(\ref{081020}), uniform in $w$, to find
 \begin{multline*}
\bigg| \int_{\mathbb{C}} (\Delta f)(\zeta) \; \Big(\Im \int_{N^{-L_1}}^{N^L} \big(m^w(\mathrm{i}\eta)-m_{\Sigma,|w|}(\mathrm{i}\eta)\big) {\rm d}\eta\Big)  {\rm d}^2 \zeta\bigg|\\
\prec  \int_{\mathbb{C}} |(\Delta f)(\zeta)| \; \Big(\int_{N^{-L_1}}^{N^L} \frac{1}{N\eta} {\rm d}\eta\Big)  {\rm d}^2  \zeta \prec 
 \frac{\|\Delta f\|_{\mathrm{L}^1(\C)}}{N}\,.
\end{multline*}
Combining~(\ref{103021}) and~(\ref{103022}), and choosing $L_1$  sufficiently  large, we get~(\ref{102920}), which together with~(\ref{103025})-(\ref{103026}) concludes the proof of  Theorem~\ref{main theorem}. 
\end{proof}

\begin{proof}[Proof of Corollary~\ref{le corollary xh}]
 Let $f\,:\,\C\rightarrow\R$ be smooth and supported on $\mathcal{R}^\tau_\sigma$; see~\eqref{le smaller ring}. It is straightforward following the proof of Theorem~\ref{main theorem} to verify that~\eqref{090602} also holds with $\alpha=0$ and $f_{w_0}$ replaced with $f$ provided that $\mathrm{supp}\,f\subset\mathcal{R}^\tau_\sigma$; \cf Remark~\ref{le remark stay away from edges}. Thus under the assumptions of Corollary~\ref{le corollary xh} it suffices to show that
 \begin{align*}
 \bigg|\int_{\mathcal{R}_\sigma^\tau} (\Delta f)(w) \int_\R\log |u|\big(\mu_{\Sigma,|w|}(\dd u)-\mu_{\sigma,|w|}(\dd u)\big){\rm d}^{ 2} w\bigg|\le C {\|\Delta f\|_{\mathrm{L}^1(\C)}}\dd_\mathrm{L}(\mu_\Sigma,\mu_\sigma)\,,
 \end{align*}
 for a constant $C$ (depending on $\tau$), to conclude its proof. From~\eqref{081011}, it is sufficient to prove~that
 \begin{align}\label{le first bom}
\Big|\int_{\R}\log|u-\ii | \;\big( \mu_{\Sigma, |w|}({\rm d}u)- \mu_{\sigma, |w|}({\rm d}u)\big)\Big|\le C  \mathrm{d}_{\mathrm{L}}(\mu_\Sigma,\mu_\sigma)
 \end{align}
and
\begin{align}\label{le second bom}
 \Big| \int_0^1 \big(m_{\Sigma,|w|}(\mathrm{i}\eta)-m_{\sigma,|w|}(\mathrm{i}\eta)\big)\, {\rm d}\eta\Big|\le C \mathrm{d}_{\mathrm{L}}(\mu_\Sigma,\mu_\sigma)\,,
\end{align}
uniformly for all $w\in\mathcal{R}_{\sigma}^\tau$, for $N$ sufficiently large. 	

Inequality~\eqref{le first bom} follows from the continuity  of the additive free convolution. More precisely, from Theorem~4.13 of~\cite{BeV93}, we know that $\mathrm{d}_\mathrm{L}(\mu_{\Sigma, |w|},\mu_{\sigma, |w|})\le\mathrm{d}_\mathrm{L}(\mu_\Sigma,\mu_\sigma)$. Since $\log|u-\ii|$ is a smooth function and $\mu_{\Sigma, |w|}$, $\mu_{\sigma, |w|}$ are compactly supported,~\eqref{le first bom} follows.

To establish~\eqref{le second bom}, we note that, for $N$ sufficiently large,
\begin{align*}
 \int_0^1 \big|m_{\Sigma,|w|}(\mathrm{i}\eta)-m_{\sigma,|w|}(\mathrm{i}\eta)\big|\, {\rm d}\eta&\le \max_{\eta\in(0,1]}\big|m_{\Sigma,|w|}(\mathrm{i}\eta)-m_{\sigma,|w|}(\mathrm{i}\eta)\big|\le C \mathrm{d}_\mathrm{L}(\mu_\Sigma,\mu_\sigma)\,,
\end{align*}
for all $w$ with $r_-+\tau\le |w|\le r_+-\tau$, with a constant depending on $\tau$. This follows directly from Theorem~2.7 of~\cite{BES15}. This shows~\eqref{le second bom}. 

So far we proved~\eqref{le xh} for smooth functions $f$. Since $\rho_\sigma$ is a Borel probability measure, see \eg Theorem~\ref{global single ring},~\eqref{le xh} extends to $f\in C^2(\C)$ supported in~$\mathcal{R}_\sigma^\tau$.  This completes the proof of Corollary~\ref{le corollary xh}.
\end{proof}

\section{Local law for block additive model} \label{s.local law for block model}

In this section, we derive a local law for block additive random matrices in a slightly generalized setting; see Theorem~\ref{thm.081501}
below. Theorem~\ref{lem.081002} is a direct consequence of this result. 

First, note that the matrix $H^w$ defined in~(\ref{def of Hw}) can be rewritten as
\begin{align}
H^w&=\left(
\begin{array}{ccccc}
U &0\\
0  & V
\end{array}
\right)\left(
\begin{array}{ccccc}
0 &\Sigma\\
\Sigma  & 0
\end{array}
\right)\left(
\begin{array}{ccccc}
U^* &0\\
0  & V^*
\end{array}
\right)+\left(
\begin{array}{ccccc}
0 & -w\\
-w^* & 0
\end{array}
\right)\,,\label{050390}
\end{align}
where $0$ is the $N\times N$ matrix filled with zeros. In the following we consider a slightly more general problem 
 by looking at  random matrices $H$ defined by
\begin{align}
H&\deq\left(
\begin{array}{ccccc}
U &0\\
0  & V
\end{array}
\right)\left(
\begin{array}{ccccc}
0 &\Sigma\\
\Sigma^*  & 0
\end{array}
\right)\left(
\begin{array}{ccccc}
U^* &0\\
0  & V^*
\end{array}
\right)+\left(
\begin{array}{ccccc}
0 & \Xi\\
\Xi^* & 0
\end{array}
\right)\,, \label{052902}
\end{align}
where
\begin{align}
 \Sigma\deq\text{diag} (\sigma_1,\ldots, \sigma_N)\,,\qquad \Xi\deq\text{diag}(\xi_1,\ldots, \xi_N)\,, \label{102002}
\end{align}
with  $\sigma_i, \xi_i\in \mathbb{C}$, $i\in \llbracket 1, N\rrbracket$. Here $\Sigma$ and $\Xi$ are deterministic diagonal matrices, while $U$ and $V$ are independent Haar unitary or Haar orthogonal matrices of degree $N$ as before. Note that we  allow in~\eqref{102002} for complex matrix elements in $\Sigma$ and $\Xi$. 
In the sequel, we always  assume that $\Sigma$ and $\Xi$ are bounded,
\begin{align}
\|\Sigma\|, \|\Xi\| \leq C\,, \label{102701}
\end{align}
for some constant $C$ independent of $N$. Denote the empirical density of their singular values~by
\begin{align}\label{le empirical measures of Sigma and Xi}
\mu_\Sigma\deq\frac{1}{N}\sum_{i=1}^N \delta_{|\sigma_i|}\,,\qquad \mu_\Xi\deq \frac{1}{N} \sum_{i=1}^N \delta_{|\xi_i|}\,.
\end{align}
Note that $\mu_\Sigma$ and $\mu_\Xi$ are probability measures on $[0,\infty)$. We assume that there are compactly supported probability measures $\mu_\sigma$ and $\mu_\xi$ such that
\begin{align}
\sup_{N\ge N_0}(\dL(\mu_\Sigma,\mu_\sigma)+\dL(\mu_\Xi, \mu_\xi)) \le 2b\,, \label{081410}
\end{align}
for a sufficiently small constant $b>0$ and sufficiently large~$N_0$.

The following general regularity result is of interest.

\begin{lem}[Theorem 4.1 in~\cite{Bel}]\label{le regularity lemma for free convolution}
 Let $\mu_1$ and $\mu_2$ be Borel probability measures on $\R$, neither of them a point mass. Then the singular continuous part of $\mu_1\boxplus\mu_2$ vanishes. A point $x\in\R$ is an atom of $\mu_1\boxplus\mu_2$ if and only if there are $x_1,x_2\in\R$ such that $x=x_1+x_2$ and $\mu_1(\{x_1\})+\mu_2(\{x_2\})>1$. Moreover, the absolutely continuous part of $\mu_1\boxplus\mu_2$ is always nonzero, and its density is
analytic wherever positive and finite.
 \end{lem}

\begin{defn}\label{def of bulk}
For two Borel probability measures $\mu_1$ on $\mu_2$ on $\R$ satisfying the assumptions of Lemma~\ref{le regularity lemma for free convolution}, we set 

\begin{align}
 \mathcal{B}_{\mu_1\boxplus\mu_2}\deq\{x\in\R: 0<f_{\mu_1\boxplus\mu_2}(x)<\infty\,, \mu_1\boxplus\mu_2(\{x\})=0 \}\,,
 \end{align}
 where~$f_{\mu_1\boxplus\mu_2}$ denotes the density function of $\mu_1\boxplus\mu_2$. We call~$\mathcal{B}_{\mu}$ the~bulk~of~$\mu$.
\end{defn}

Let $G\equiv G(z)\deq(H-z)^{-1}$ be the Green function of $H$ at parameter $z\in \C^+$,  and let  
\begin{align}
   m_H(z)\deq\ntr G(z) = \frac{1}{2N} \text{Tr} \,G(z)
   \end{align}
be the normalized trace of $G(z)$, which by the functional calculus agrees with the Stieltjes transform of the empirical eigenvalue distribution of $H$. 

Given an interval $\mathcal{I}\subset\R$ and $0\le a\le b$, we introduce the domain
\begin{align}
\mathcal{S}_\mathcal{I}(a,b)\deq\{z=E+\mathrm{i}\eta\in\C^+\,:\, E\in \mathcal{I}, a<\eta\leq b\}\,. \label{le definition of caS domain}
\end{align}

As before, we denote for a measure $\mu$ on $\R$ its symmetrization by $\mu^{\mathrm{sym}}$. The following is a key result of this paper.
  
\begin{thm}[Strong law for $H$] \label{thm.081501} Suppose that~\eqref{102701} holds. Let $\mu_\sigma$ and $\mu_\xi$ be two compactly supported probability measures on $[0,\infty)$ such that neither  $\mu_\sigma^{\mathrm{sym}}$  nor $\mu_\xi^{\mathrm{sym}}$ is a single point mass and at least of one of them is supported at more than two points. Fix some $L>0$ and let $\mathcal{I}$ be any compact interval of the bulk $\mathcal{B}_{\mu_\sigma^{\mathrm{sym}}\boxplus\mu_\xi^{\mathrm{sym}}}$.  Then there exists a (small) constant $b_0>0$ and $N_0\in\N$, depending only on $\mu_\sigma$, $\mu_\xi$, $\mathcal{I}$ and the constant $C$ in~\eqref{102701}, such that whenever
\begin{align}\label{le assumption on levy distances 2b}
\sup_{N\ge N_0}(\dL(\mu_\Sigma,\mu_\sigma)+\dL(\mu_\Xi, \mu_\xi)) \le 2b\,,
\end{align}
 for some $b\le b_0$, then 
\begin{align}
\big|m_H(z)-{ m_{\mu_\Sigma^{\mathrm{sym}}\boxplus\mu_\Xi^{\mathrm{sym}}}(z)}\big|\prec\frac{1}{N\eta (1+\eta)} \label{the local law for mH prec inequality}
\end{align}
holds uniformly on $\mathcal{S}_\mathcal{I}({ 0, N^L})$, for $N$ sufficiently large depending only on~$\mu_\sigma$, $\mu_\xi$, $\mathcal{I}$, $L$ and the constant $C$ in~\eqref{102701} . Moreover, there exists a constant $\eta_{\mathrm{M}}\ge 1$, independent of $N$, such that~\eqref{the local law for mH prec inequality} holds uniformly on $\mathcal{S}_{\mathcal{I}}(\eta_M,N^L)$, for any compact interval $\mathcal{I}\subset\R$,  for $N$ sufficiently large depending only on~$\mu_\sigma$, $\mu_\xi$, $L$ and the constant $C$ in~\eqref{102701} .
\end{thm}
Theorem~\ref{thm.081501} is proved in Sections~\ref{s.green function subordination}-\ref{s.strong law} and Section~\ref{s.large eta}. In fact in Section~\ref{s.large eta}, we  prove Theorem~\ref{thm.081501} for spectral parameters $z\in\C^+$ with large imaginary parts, $\eta$.  Here, large $\eta$ means  $\eta\ge \eta_\mathrm{M}$, for some $\eta_{\mathrm{M}}\ge 1$ independent of $N$ to be chosen below. The proof for large $\eta$ relies on the Gromov--Milman concentration inequality for the full Haar measure in conjunction with identities for expectations of Green functions originating in the global $U(N)$-symmetry. These arguments are independent of the main line followed here and are hence postponed to Section~\ref{s.large eta}. The results for large~$\eta$ serve as initial estimates in a boostrap argument carried out in Sections~\ref{s.green function subordination}-\ref{s.strong law} where we prove Theorem~\ref{thm.081501} in the complementary regime where $\eta< \eta_{\mathrm{M}}$.

\begin{proof}[Proof of~Theorem~\ref{lem.081002}] Theorem~\ref{lem.081002} follows from Theorem~\ref{thm.081501} by choosing $\mathcal{I}=\{0\}$. The conditions of Theorem~\ref{thm.081501} require that the density of~$\mu_\sigma^{\mathrm{sym}}\boxplus\mu_\xi^{\mathrm{sym}}$ is uniformly bounded from below on the compact interval $\mathcal{I}$. For $E=0$, this condition was verified in Theorem~\ref{le corollary of 0 in the bulk}. This yields~\eqref{081020} uniformly for $0<\eta\le N^L$, with $L>1$ as in Theorem~\ref{thm.081501} for fixed $w$ with $|w|\in[r_-+\tau,r_+-\tau]$.

Next, we show that~\eqref{081020} can be strengthened to a uniform bound in $w\in\mathcal{R}_{\sigma}^\tau\deq\{w\in\C\,:\, |w|\in[r_-+\tau,r_+-\tau]\}$. We introduce the lattice
\begin{align*}
\widehat{\mathcal{R}}_{\sigma}^\tau( L_1)\deq \mathcal{R}_{\sigma}^\tau \cap N^{-L_1}\big\{\mathbb{Z}\times \ii \mathbb{Z}\big\}\,,
\end{align*}
for some sufficiently large positive constant $L_1$ such that $L_1\geq 2 L$ (say).  Using the definition of stochastic domination in Definition~\ref{definition of stochastic domination} and~\eqref{081020} for fixed $w$, we obtain
\begin{align*}
\max_{w\in\widehat{\mathcal{R}}_{\sigma}^\tau( L_1)} \big|m^w(z)-m_{\Sigma,|w|}(z)\big|\prec\frac{1}{N\eta}\,, 
\end{align*} 
uniformly in $0<\eta\le N^L$. To extend this bound to all of $ \mathcal{R}_{\sigma}^\tau $, it suffices to show 
 Lipschitz continuity of these quantities in $w$. We need that, for any $w_1, w_2\in \mathcal{R}_{\sigma}^\tau $ with $|w_1-w_2|\leq N^{-L_1}$ for sufficiently large $L_1$, one has
\begin{align}
\big|m^{w_1}(z)-m^{w_2}(z)\big|\leq \frac{1}{N\eta}\,,\qquad    \big|m_{\Sigma, |w_1|}(z)-m_{\Sigma, |w_2|}(z)\big|\leq \frac{1}{N\eta}\,, \label{103010}
\end{align} 
uniformly in $0<\eta\le N^L$. To show the first deterministic bound in~(\ref{103010}), we use the bound
\begin{align*}
\big|m^{w_1}(z)-m^{w_2}(z)\big|&\leq \frac{|w_1-w_2|}{2N} \text{Tr} {|H^{w_1}-z|^{-1}|H^{w_2}-z|^{-1}}\nonumber\\ &\leq \frac{|w_1-w_2|}{2\eta^2}\leq \frac{1}{2\eta}N^{-L_1+ L}\leq \frac{1}{N\eta}\,,
\end{align*}
where $|A|\deq \sqrt{A^*A}$, for any square matrix $A$.

To show  the second bound in~(\ref{103010}), we use the stability of the Stieltjes transform of free additive convolution. Here it suffices to use the following bound  (\cf (2.20) in \cite{BES15} for instance) 
\begin{align*}
|m_{\Sigma,|w_1|}(z)-m_{\Sigma, |w_2|}(z)|\leq \frac{C}{\eta}\big(1+\frac{1}{\eta}\big) \dd_{\rm L} \Big(\delta_{|w_1|}^{\mathrm{sym}}, \delta_{|w_2|}^{\mathrm{sym}}\Big)\leq  \frac{C}{\eta}\big(1+\frac{1}{\eta}\big)|w_1-w_2|\,,
\end{align*} 
for all $z=E+\ii \eta\in \mathbb{C}^+$, where $C$ is a constant uniform in $z$.  Using the assumptions $|w_1-w_2|\leq N^{-L_1}$ and $L_1\geq 2L$, we get~(\ref{103010}), which in turn establishes the desired uniformity of~\eqref{081020} in $w\in\mathcal{R}_{\sigma}^\tau $.

To complete the proof of~\eqref{081020}, it remains to deal with the large $\eta$ regime, \ie when $\eta \ge N^{L}$. For that we use the elementary (deterministic) estimates
\begin{align}\label{le trivial unsinn}
 m_H(\ii\eta)=-\frac{1}{\ii\eta}+O\left(\frac{1}{|\eta|^3}\right)\,,\qquad m_{\Sigma,|w|}(\ii\eta)=-\frac{1}{\ii\eta}+O\left(\frac{1}{|\eta|^3}\right)\,,
\end{align}
as $\eta\nearrow\infty$, where we used a resolvent expansion of $G$ together with $\mathrm{tr} H=0$ and $\|H\|\le S_+$ (see~\eqref{le big S_+}), and the large $\eta$ expansion of the Stieltjes transform together with the fact that~$\mu_{\Sigma,|w|}$ is symmetric and compactly supported. Thus for $\eta\ge N^L$,~\eqref{081020} follows from~\eqref{le trivial unsinn}. Uniformity in $w\in \mathcal{R}_\sigma^\tau$ is immediate. \end{proof}

\subsection{Approximate subordination for block additive models}
In this subsection, we establish the matrix subordination for the Green function of $H$. To simplify notation, we introduce the block matrices
\begin{align}
 A  \deq\left(
\begin{array}{ccccc}
0 & \Xi\\
\Xi^* & 0
\end{array}
\right),\qquad  B :=\left(
\begin{array}{ccccc}
0 &\Sigma\\
\Sigma^*  & 0
\end{array}
\right),\qquad  \mathcal{U}:=\left(
\begin{array}{ccccc}
U &0\\
0  & V
\end{array}
\right). \label{050970}
\end{align}
Then we write~(\ref{052902}) as
\begin{align}
H=A+\wt{B}\,,\qquad\qquad \wt{ B }\deq \mathcal{U}  B\, \mathcal{U}^*\,. \label{091301}
\end{align}
As before, we let $G(z)\deq (H-z)^{-1}$ be the Green function of $H$ at spectral parameter~$z\in\C^+$. A simple consequence of the definition of $G$ are the identities
\begin{align}
\wt{B}G(z)=I_{2N}-(A-z)G(z)\,, \qquad G(z)\wt{B}=I_{2N}-G(z)(A-z)\,. \label{102702}
\end{align}
Inspired by~\cite{VP}, see also~\cite{BES15b,BG,Kargin}, we introduce the approximate subordination functions
\begin{align}
\omega_A^c(z)\deq z-\frac{\ntr  A  G}{\ntr G}\,,\qquad \omega_B^c(z)\deq z-\frac{\ntr \wt{ B }G}{\ntr G}\,.  \label{091107}
\end{align}
By these definitions and ~(\ref{102702}), we have
\begin{align}
\omega_A^c(z)+\omega_B^c(z)-z=-\frac{1}{m_H(z)}\,.  \label{091110}
\end{align}
  
Recall the measures $\mu_\Sigma$ and~$\mu_\Xi$ of~\eqref{le empirical measures of Sigma and Xi} as well as $\mu_\sigma$ and $\mu_\xi$ of~\eqref{081410}.
For their symmetrizations we introduce, hinting at~\eqref{050970},  the shorthands
\begin{align}
\mu_A\equiv \mu_\Xi^{\mathrm{sym}},\qquad \mu_B\equiv \mu_\Sigma^{\mathrm{sym}},\qquad \mu_\alpha\equiv \mu_\xi^{\mathrm{sym}},\qquad \mu_\beta\equiv \mu_\sigma^{\mathrm{sym}}.  \label{091105}
\end{align}
Note that $\mu_A$ and $\mu_B$ are the empirical spectral distributions of $ A  $ and $ B $. We denote by $\omega_A(z),\omega_B(z), \omega_\alpha(z), \omega_\beta(z)$ the subordination functions defined via~\eqref{le definiting equations} with the choices $(\mu_1,\mu_2)=(\mu_A,\mu_B)$ and $(\mu_\alpha, \mu_\beta)$, respectively. 
  
The next result shows that the approximate subordination functions $\omega_A^c$ and $\omega_B^c$ are indeed good approximations to the subordination functions $\omega_A$ and $\omega_B$. Moreover, it establishes the subordination for the diagonal Green function entries.  

 \begin{thm}\label{le proposition of Green function subordination} Under the conditions and with the notations of Theorem~\ref{thm.081501} the estimates
  \begin{align}
|\omega_A^c(z)-\omega_A(z)|   \prec\frac{1}{N\eta}\,,\qquad\qquad |\omega_B^c(z)-\omega_B(z)|\prec \frac{1}{N\eta}\,,
  \end{align}
hold uniformly on $\mathcal{S}_{ \mathcal{I} }(0,\eta_{\mathrm{M}})$, for $N$ sufficiently large depending only on~$\mu_\alpha$, $\mu_\beta$, $\mathcal{I}$, $L$ and the constant $C$ in~\eqref{102701}. Moreover, we have
  \begin{align}
& \bigg| G_{ii}(z)- \frac{\omega_B(z)}{|\xi_i|^2-(\omega_B(z))^2}\bigg|\prec\frac{1}{\sqrt{N\eta}}\,,\quad
 \bigg| G_{\hat{i}\hat{i}(z)}- \frac{\omega_B(z)}{|\xi_i|^2-(\omega_B(z))^2}\bigg|\prec\frac{1}{\sqrt{N\eta}}\,,\nonumber\\
& \bigg| G_{i\hat{i}}(z)- \frac{\xi_i}{|\xi_i|^2-(\omega_B(z))^2}\bigg|\prec\frac{1}{\sqrt{N\eta}}\,,\quad  \bigg| G_{\hat{i}i}(z)- \frac{\bar{\xi}_i}{|\xi_i|^2-(\omega_B(z))^2}\bigg|\prec\frac{1}{\sqrt{N\eta}}\,,\label{le diagonal Green function subordination}
 \end{align}
uniformly in $i\in\llbracket 1,N\rrbracket$ and in $z\in\mathcal{S}_{\mathcal{I}}(0,\eta_\mathrm{M})$, where $\hat i\deq i+N$, for $N$ sufficiently large depending only on~$\mu_\alpha$, $\mu_\beta$, $\mathcal{I}$, $L$ and the constant $C$ in~\eqref{102701}.
 \end{thm}
\begin{rem}
 Some crucial properties of the subordination functions $\omega_A$ and $\omega_B$ are collected in Lemma~\ref{lem.A.1}. Here, we mention that under the assumptions of Theorem~\ref{le proposition of Green function subordination}, for $N$ sufficiently large, the imaginary parts of the subordination functions, $\im\omega_A(z)$ and $\im\omega_B(z)$ are both bounded from below on $z\in\mathcal{S}_{\mathcal{I}}(0,\eta_\mathrm{M})$. This follows from Lemma~\ref{lem.A.1} and the assumption that $\mathcal{I}$ is a compact interval in the bulk of $\mu_\alpha\boxplus\mu_\beta$. It then follows from~\eqref{le diagonal Green function subordination} that $|G_{ii}(z)|\prec 1$ and $|G_{\hat i\hat i}(z)|\prec 1$ uniformly on $\mathcal{S}_{\mathcal{I}}(N^{-1+\gamma}, \eta_{\mathrm{M}})$, for any $\gamma>0$, and all $i\in\llbracket 1,N\rrbracket$. A direct consequence of this result is that the eigenvectors associated with eigenvalues in the bulk are fully delocalized. More precisely, letting $(\mathbf{u}_k)$ denote the $\ell^2$-normalized eigenvectors associated with the eigenvalues $(\lambda_k)$, $k\in\llbracket 1,2N\rrbracket$, we have
 \begin{align}\label{le delocalized eigenvectors}
  \max_{k\,:\, \lambda_k\in\mathcal{I}}\|\mathbf{u}_k\|_\infty\prec \frac{1}{\sqrt{N}}\,,
 \end{align}
for any compact interval $\mathcal{I}$ in the bulk of $\mu_\alpha\boxplus\mu_\beta$. For a proof of~\eqref{le delocalized eigenvectors} from Theorem~\ref{le proposition of Green function subordination}, we refer to the proof of Theorem~2.6 in~\cite{BES15b}.
\end{rem}
 
\subsection{Outline of the strategy of proof}\label{sec:outline}

The proof of the local law of Theorem~\ref{thm.081501} is carried out in three steps. In Step 1, we consider the large $\eta$ regime, \ie we establish~\eqref{the local law for mH prec inequality} on $\mathcal{S}_\mathcal{I}(\eta_{\mathrm{M}},N^L)$, for some sufficiently large, but $N$-independent, $\eta_{\mathrm{M}}$. In Step~2, we establish a weak local law for $m^w$ in the small $\eta$ regime, \ie we establish~\eqref{the local law for mH prec inequality} with a weaker error bound on $\mathcal{S}_{\mathcal{I}}(N^{-1+\gamma},\eta_{\mathrm{M}})$, for some small $\gamma>0$; see Theorem~\ref{thm.green function subordination} below for the statement of the weak law. The extension to $\mathcal{S}_{\mathcal{I}}(0,\eta_\mathrm{M})$ will follow directly from monotonicity of the Green function. This second step is based on a bootstrapping argument to reduce the spectral parameter $\im z$. Step~1 will provide the initial estimate to get the bootstrapping started. In Step~3, we use a fluctuation averaging argument together with the weak local law established in the second step  to get~\eqref{the local law for mH prec inequality} in its strong form.

Step~1 is carried out in Section~\ref{s.large eta}. It builds on the celebrated Gromov-Milman concentration inequality whose application to random matrix theory is fairly standard~\cite{AGZ}. For additive models of the form $X+UYU^*$, with deterministic $X,Y\in M_N(\C)$ and $U$ Haar distributed on $U(N)$ or on $O(N)$ it was used in~\cite{VP,Kargin,BES15}, and for the model block-additive model considered in this section in~\cite{BG}.

Step~2 is carried out in Section~\ref{s.green function subordination}, where we prove Theorem~\ref{thm.green function subordination}. This proof has three major ingredients. First, we use a partial randomness decomposition
of the Haar measure (see~\eqref{050202}) that enables us to take partial expectations of functions of the diagonal Green function entries $G_{ii}$, $G_{\hat i\hat i}$. Exploiting
concentration only for this partial randomness surpasses the more general but less flexible Gromov--Milman technique used in Step~1. Second, to compute the partial expectations of $G_{ii}$, we establish
a system of self-consistent equations involving only two auxiliary quantities $(S_{ij})$ and $(T_{ij})$; see~\eqref{053070}. In our previous work~\cite{BES15b}, we used a similar approach to derive the local law for $X+UYU^*$. For the model considered in this paper, we face with a new phenomenon causing several  substantial difficulties. The main point is that for block additive models, we have less randomness originating in the Haar measure on $U(N)\times U(N)$ than for the additive models with Haar measure on $U(2N)$. As a consequence, we have to control more quantities in the two blocks 
separately. Even more importantly, the coupling between the two blocks is provided solely by the  diagonal matrix $\Sigma$ without any randomness; see
\eqref{052902}. Our proof shows that the randomness in the diagonal blocks and the deterministic off-diagonal blocks effectively make
up for the lacking off-diagonal randomness.

To derive the aforementioned system of equations for $(S_{ij})$, $(T_{ij})$ and $(G_{ii})$, we use the partial decomposition of Haar measure in combination with recursive moment estimates; see \eg Lemma~\ref{lem.071420} for such a statement. Recursive moment estimates were used first in~\cite{LS16} to derive local laws for sparse Wigner matrices. They allow us to pass on cumbersome partial concentration estimates used in Section~5 of~\cite{BES15b}, and provide a conceptually clear approach to the weak local law for both models.
Third, to connect the diagonal Green function entries with the subordination functions from~Theorem~\ref{le prop 1}, we rely on the optimal stability result
for the subordination equations obtained in~\cite{BES15}.

Step~3 is carried out in Section~\ref{s.strong law}. In this section, we exploit the so-called fluctuation averaging mechanism to~improve the estimates of Step 2. While the fluctuation averaging mechanism is, thanks to the independence of the matrix entries, well understood for Wigner type matrices (see \eg ~\cite{EYYBer,EKY}), dependencies among the entries of the Haar matrices mask this mechanism and its current understanding for matrix ensembles involving Haar matrices is still rather poor. We gave a first result in~\cite{BES15c} for additive models. In the present paper, we approach the fluctuation mechanics for block-additive models by first deriving a set of so-called ``Ward identities'' which will enable us to finish the proof of Theorem~\ref{thm.081501}. Ward identities are relations among tracial quantities involving the Green function and the matrices~$A$ and~$B$. In expectation, these relations can be derived using the invariance of Haar measure (see \eg~\eqref{le pastur trick} for a first example), yet we will require optimal estimates that hold with high probability; see~\eg~\eqref{le Ward for Upsilon} and~\eqref{083110}. These estimates are obtained using recursive moment estimates for carefully chosen quantities; see~\eqref{071560}. Since we have less randomness coming from $U(N)\times U(N)$ in the setup of block-additive models, more quantities need to be simultaneously controlled than in the additive models, resulting in a more sophisticated analysis.

\subsection{Notations}  
We introduce some more notation used in the proof of Theorem~\ref{thm.081501}.
  
{\it Notation for matrices:} In our analysis we also use the matrices
 \begin{align}
 \mathcal{H}=B+\mathcal{U}^*A\,\mathcal{U}=:B+\wt{A}\,,\qquad \mathcal{G}(z)=(\mathcal{H}-z)^{-1}\,, \label{0913125}
 \end{align}
 which are the analogues of $H$ in~(\ref{091301}) and of its Green function $G(z)$, obtained by switching the r\^oles of $A$ and $B$, and also the r\^oles of
 $\mathcal{U}$ and $\mathcal{U}^*$. Note that by cyclicity $\mathrm{Tr} G(z)=\mathrm{Tr} \mathcal{G}(z)$.

{\it Vector space notation:}  For any index $i\in\llbracket1, N\rrbracket$, we let $\hat{i}\equiv i+N$. 
We make the convention hereafter that the index $i$ always runs from $1$ to $N$, unless said otherwise. Thus the index $\hat{i}$ runs from $N+1$ to $2N$. We denote by $\sum_{i}^{(k)}$  the sum over $i\in \llbracket 1, N\rrbracket\setminus\{ k\}$.  We denote by $\{\mathbf{e}_i\}$ the canonical basis of $\mathbb{C}^N$ while we denote by $\{\hat{\mathbf{e}}_i\}$ the canonical basis of $\mathbb{C}^{2N}$. We let~$\mathbf{0}$ denote the zero vector in either space. We use bold font for vectors and denote the components as $\mathbf{v}=(v_i)$.

The identity matrix in $M_N(\C)$, respectively $M_{2N}(\C)$, is denoted by
\begin{align}
I\equiv I_N\,,\qquad \hat{I}\equiv I_{2N}\,, \label{111301}
\end{align}
and we let
\begin{align}
\hat{I}_{1}\deq I\oplus 0\,,\qquad \hat{I}_{2}\deq 0\oplus I  \label{071540}
\end{align}
denote the block identities in $M_{N}(\C)\oplus M_N(\C)$, where $0$ represents the $N\times N$ zero matrix. 

For any  matrix $D\in M_n(\mathbb{C})$, $n\ge 1$, we let
\begin{align*}
\ntr D \deq\frac{1}{n} \text{Tr} D
\end{align*}
denote the normalized trace of $D$. For $D\in M_{2N}(\mathbb{C})$ we introduce the normalized partial traces
\begin{align}
\tau_1(D)\deq\frac{1}{N} \sum_{i=1}^N D_{ii}\,,\qquad \tau_2(D)\deq\frac{1}{N} \sum_{i=1}^N D_{\hat{i}\hat{i}}\,.  \label{111901}
\end{align}
Using the block structure of $H$, it is easy to check that the Green function $G(z)$ satisfies
\begin{align}
\tau_1(G(z))=\tau_2(G(z))\,,\qquad\qquad z\in\C^+\,. \label{neu083140}
\end{align}

{\it $\Phi$-system:} For our purposes it is convenient to recast~\eqref{le definiting equations} in a compact form: For generic probability measures $\mu_1,\mu_2$ on $\R$, let the function $\PP_{\mu_1,\mu_2}\,:\, (\C^+)^{3}\rightarrow \C^2$ be given by
\begin{align}\label{le H system defs}
\PP_{\mu_1,\mu_2}(\omega_1,\omega_2,z)\deq\left(\begin{array}{cc} F_{\mu_1}(\omega_2)-\omega_1-\omega_2+z \\ F_{\mu_2}(\omega_1)-\omega_1-\omega_2+z \end{array}\right)\,.
\end{align}
Considering $\mu_1,\mu_2$ as fixed, the equation
\begin{align}\label{le H system}
\PP_{\mu_1,\mu_2}(\omega_1,\omega_2,z)=0\,,
\end{align}
is equivalent to~\eqref{le definiting equations} and, by Theorem~\ref{le prop 1}, there are unique analytic functions $\omega_1,\omega_2\,:\, \C^+\rightarrow \C^+$, $z\mapsto \omega_1(z),\omega_2(z)$ satisfying~\eqref{le limit of omega} that solve~\eqref{le H system} in terms of $z$.

{\it Control parameters:} 
For $z\in\C^+$, we will use the following deterministic control parameter
\begin{align}
\Psi\equiv \Psi(z)\deq{\frac{1}{\sqrt{N\eta(1+\eta)}}}\,, \qquad  \qquad \eta=\im z \,. \label{071902}
\end{align}
We further introduce, for $z\in\C^+$ and $i\in\llbracket 1,N\rrbracket$, the random control parameters
\begin{align}
&\Lambda_{{\rm d};ii}(z)\deq \bigg| G_{ii}- \frac{\omega_B(z)}{|\xi_i|^2-(\omega_B(z))^2}\bigg|\,,\qquad
 \Lambda_{{\rm d}; \hat{i} \hat{i}}(z)\deq \bigg| G_{\hat{i}\hat{i}}- \frac{\omega_B(z)}{|\xi_i|^2-(\omega_B(z))^2}\bigg|\,,\nonumber\\
& \Lambda_{{\rm d};i\hat{i}}(z)\deq \bigg| G_{i\hat{i}}- \frac{\xi_i}{|\xi_i|^2-(\omega_B(z))^2}\bigg|\,,\qquad  \Lambda_{{\rm d};\hat{i}i}(z):= \bigg| G_{\hat{i}i}- \frac{\bar{\xi}_i}{|\xi_i|^2-(\omega_B(z))^2}\bigg|\,,\nonumber\\ \nonumber\\
&\hspace{20ex}\Lambda_{{\rm d}}(z)\deq \max_{i\in \llbracket 1, N\rrbracket}\max_{k,l=i \text{ or }\hat{i}}\Lambda_{{\rm d};kl}(z)\,. \label{071603}
\end{align}
We also define $ \Lambda_{{\rm d}}^c(z)$ analogously by replacing $\omega_B$ by $\omega_B^c$ (\cf~(\ref{091107})) in the definition of $\Lambda_{{\rm d}}(z)$. We will often omit the variable $z$ from the above notations when there is no confusion.

For notational simplicity, we do not follow the threshold $N$ for which the estimates apply. Following the dependence of this threshold on the other parameters along the proofs one may easily verify the dependences stated in Theorem~\ref{thm.081501} and Theorem~\ref{le proposition of Green function subordination}.

\section{Green function subordination for small $\eta$}  \label{s.green function subordination}

Let $\eta_{\mathrm{M}}>0$ be some sufficiently large constant, and for any given (small) $\gamma>0$, we set
\begin{align}\label{etamdef}
\eta_{\rm m} \equiv \eta_{\rm m} (\gamma)\deq N^{-1+\gamma}\,.
\end{align} 
In this section, we prove a Green function subordination property in the regime $\eta_{\mathrm{m}}\leq\eta \leq \eta_{\mathrm{M}}$. The formal statement is given in Theorem~\ref{thm.green function subordination} below.  For definiteness, we work with the unitary setup in this section. The necessary modifications for the orthogonal case are stated in Appendix~\ref{Appendix C}. We start with the partial randomness decomposition of the Haar measure on $U(N)\times U(N)$ announced in Subsection~\ref{sec:outline}.

\subsection{Partial randomness decomposition of the Haar measure}

Let $\mathbf{u}_i=(u_{i1},\ldots, u_{iN})'$ and $\mathbf{v}_i=(v_{i1},\ldots, v_{iN})'$ be the $i$th columns of $U$ and $V$, respectively. Let $\theta_{i}^u$ and $\theta_i^v$ be the arguments of $u_{ii}$ and $v_{ii}$, respectively, and let $\phi_i^a=\mathrm{e}^{\mathrm{i}\theta_i^a}$ for $a=u,v$. Our approach relies on the partial randomness decomposition of the Haar measure from \cite{DS87,Mezzadri}:
\begin{align}
U=-\phi_i^uR_i^u U^{\la i\ra}\,, \qquad V=-\phi_i^vR_i^v V^{\la i\ra}\,. \label{050202}
\end{align}
Here $U^{\la i\ra}$ and $V^{\la i\ra}$ are unitary matrices with $(i,i)$-th entry equal $1$, 
 and their $(i,i)$-minors are  independent,  Haar distributed on $\mathcal{U}(N-1)$. 
  In particular, $U^{\la i\ra} \mathbf{e}_i=V^{\la i\ra} \mathbf{e}_i=\mathbf{e}_i$ and $  \mathbf{e}_i^*U^{\la i\ra}  = \mathbf{e}_i^*V^{\la i\ra} = \mathbf{e}_i^*$,
 where $\mathbf{e}_i$ is the $i$-th coordinate vector. 
   In addition, $U^{\la i\ra}$ is independent of $\mathbf{u}_i$, and $V^{\la i\ra}$ is independent of $\mathbf{v}_i$. Here $R_i^u$ and $R_i^v$ are reflections, defined as
\begin{align}
R_i^a\deq I-\mathbf{r}_i^a(\mathbf{r}_i^a)^*\,,\qquad  a=u,v\,, \label{def of R}
\end{align}
where
\begin{align}
\mathbf{r}_i^u\deq\sqrt{2}\frac{\mathbf{e}_i+\bar{\phi}_i^u\mathbf{u}_i}{\|\mathbf{e}_i+\bar{\phi}_i^u\mathbf{u}_i\|_2}\,, \qquad \qquad  \mathbf{r}_i^v\deq\sqrt{2}\frac{\mathbf{e}_i+\bar{\phi}_i^v\mathbf{v}_i}{\|\mathbf{e}_i+\bar{\phi}_i^v\mathbf{v}_i\|_2}\,.  \label{110301}
\end{align}
Note that $R_i^u$ is independent of $U^{\la i\ra}$ and $R_i^v$ is independent of $V^{\la i\ra}$.

Set the  $(2N)\times (2N)$  matrices
\begin{align}
\Phi_i\deq(\phi_i^u I)\oplus (\phi_i^v I)\,,\qquad \qquad \mathcal{R}_i\deq R_i^u\oplus R_i^v\,,\qquad \qquad \mathcal{U}_i\deq U^{\la i\ra}\oplus V^{\la i\ra}\,. \label{052704}
\end{align}
With the above notations and the decompositions in~(\ref{050202}), we have
\begin{align}
 \mathcal{U}=-\mathcal{R}_i\mathcal{U}_i \Phi_i\,.   \label{101811}
\end{align}
Hence, for each $i\in \llbracket 1, N\rrbracket$, we can write
\begin{align}
H=A+\wt{B}=A+\mathcal{R}_i \mathcal{U}_i \Phi_iB\Phi_i^* \mathcal{U}_i^*  \mathcal{R}_i\deq A+\mathcal{R}_i \wt{B}^{\la i\ra} \mathcal{R}_i\,, \label{081101}
\end{align}
where we introduced the notation
\begin{align}
\wt{ B }^{\la i\ra}\deq \mathcal{U}_i \Phi_iB \Phi_i^* \mathcal{U}_i^* \,. \label{052703}
\end{align}
We further define the matrices
\begin{align}
H^{\la i\ra}\deq A+\wt{B}^{\la i\ra}\,, \qquad G^{\la i\ra}\deq\big(H^{\la i\ra}-z\big)^{-1}\,. \label{0916110}
\end{align}
Since $\mathbf{u}_i$ and $\mathbf{v}_i$ are  independent, uniformly distributed complex unit vectors, there exist
 independent normal vectors, $\wt{\mathbf{g}}_i^u,\wt{\mathbf{g}}_i^v\sim \mathcal{N}_{\mathbb{C}}(0, \frac{1}{N}I_N)$ such that
\begin{align*}
\mathbf{u}_i=\frac{\wt{\mathbf{g}}_i^u}{\|\wt{\mathbf{g}}_i^u\|_2}\,,\qquad \mathbf{v}_i=\frac{\wt{\mathbf{g}}_i^v}{\|\wt{\mathbf{g}}_i^v\|_2}\,.
\end{align*} 
We further define
\begin{align}
\mathbf{g}_i^u\deq\bar{\phi}_i^u\wt{\mathbf{g}}_i^u\,,\qquad \mathbf{h}_i^u\deq\frac{\mathbf{g}_i^u}{\|\mathbf{g}_i^u\|_2}=\bar{\phi}_i^u \mathbf{u}_i\,,\qquad \ell_i^u\deq\frac{\sqrt{2}}{\|\mathbf{e}_i+\mathbf{h}_i^u\|_2}\,,  \label{052802}
\end{align}
and define $\mathbf{g}_i^v$, $\mathbf{h}_i^v$ and $\ell_i^v$  analogously by replacing $\mathbf{u}_i$ by $\mathbf{v}_i$. 
Note that  for $a=u$ or $v$, $g_{ik}^a$'s  for $k\neq i$ are $N_\mathbb{C} (0,\frac{1}{N})$ variables and $g_{ii}^a$ is $\chi$-distributed with $\mathbb{E}[(g_{ii}^a)^2]=\frac{1}{N}$. In  addition, the components of $\mathbf{g}_i^a$ are independent, and  they  are all independent of $\phi_i^a$.   Hence, $\mathbf{g}_i^a$ and $\mathbf{h}_i^a$ are independent of $\wt{B}^{\la i\ra}$ (\cf~(\ref{052703})), for $a=u,v$. With these notations, we can write
\begin{align}
\mathbf{r}_i^a=\ell_i^a(\mathbf{e}_i+\mathbf{h}_i^a)\,,\qquad\qquad a=u,v\,, \label{050820}
\end{align}
where $\mathbf{r}_i^a$ is defined in~(\ref{110301}).
Using Lemma~\ref{lem.091720}, it is elementary to check that, for $a=u,v$,
\begin{align}
\|\mathbf{g}_i^a\|_2=1+\frac{1}{2}\big(\|\mathbf{g}_i^a\|_2^2-1\big)+O_\prec\big(\frac{1}{N}\big)\,,\quad (\ell_i^a)^2=\frac{1}{1+\mathbf{e}_i^*\mathbf{h}_i^a}=1-g_{ii}^a+O_\prec\big(\frac{1}{N}\big)\,,  \label{050901}
\end{align}
where in the first estimate we used the fact $\big|\|\mathbf{g}_i^a\|_2^2-1\big|\prec \frac{1}{\sqrt{N}}$. 
In addition, by definition, $R_i^a$ is a reflection sending $\mathbf{e}_i$ to $-\mathbf{h}_i^a$, \ie
\begin{align}
R_i^a\mathbf{e}_i=-\mathbf{h}_i^a,\qquad R_i^a\mathbf{h}_i^a=-\mathbf{e}_i\,,\qquad a=u,v\,.  \label{050805}
\end{align}
We also denote by $\mathring{\mathbf{g}}_i^a$ the vector obtained from $\mathbf{g}_i^a$ by replacing $g_{ii}^a$ by $0$, \ie
\begin{align*}
\mathring{\mathbf{g}}_i^a\deq\mathbf{g}_i^a-g_{ii}^a\mathbf{e}_i\,,\qquad a=u,v\,. 
\end{align*}
Correspondingly, we set
\begin{align}
\mathring{\mathbf{h}}_i^a\deq\frac{\mathring{\mathbf{g}}_i^a}{\|\mathbf{g}_i^a\|_2}\,,\qquad a=u,v\,.  \label{053065}
\end{align}
Recall the notation $\mathbf{0}$ for the $N\times 1$ null vector.  Finally, for brevity, we set
\begin{align}
\mathbf{k}_i^u\deq\left(\begin{array}{cc}
\mathbf{h}_i^u\\
\mathbf{0}
\end{array}\right),\quad \mathbf{k}_i^v\deq\left(\begin{array}{cc}
\mathbf{0}\\
\mathbf{h}_i^v
\end{array}\right),\quad  \mathring{\mathbf{k}}_i^u\deq\left(\begin{array}{cc}
\mathring{\mathbf{h}}_i^u\\
\mathbf{0}
\end{array}\right),\quad \mathring{\mathbf{k}}_i^v:=\left(\begin{array}{cc}
\mathbf{0}\\
\mathring{\mathbf{h}}_i^v
\end{array}\right)\,. \label{052801}
\end{align}
We move on to the formal statement of the Green function subordination.
\subsection{Green function subordination}

Recall the notation $\{\hat{\mathbf{e}}_i\}$ for the standard basis of  $\mathbb{C}^{2N}$,  
and also the notation $\hat{i}\equiv i+N$ 
 for any $i\in \llbracket 1, N\rrbracket$.  We introduce the following quantities for $j=i,\hat{i}$,   $i\in \llbracket 1, N\rrbracket$, 
\begin{align}
&S_{ij}:= (\mathbf{k}_i^u)^* \wt{ B }^{\la i\ra}  G\hat{\mathbf{e}}_j\,,\qquad \qquad  T_{ij}:= (\mathbf{k}_i^u)^* G\hat{\mathbf{e}}_j\,,\nonumber\\
&S_{\hat{i}j}:= (\mathbf{k}_i^v)^* \wt{ B }^{\la i\ra}  G\hat{\mathbf{e}}_j\,,\qquad \qquad  T_{\hat{i}j}:= (\mathbf{k}_i^v)^* G \hat{\mathbf{e}}_j\,, \label{053070}
\end{align}
and 
\begin{align}
\mathring{S}_{ii}:= (\mathring{\mathbf{k}}_i^u)^* \wt{ B }^{\la i\ra}  G\hat{\mathbf{e}}_i=S_{ii}-\tilde{\sigma}_i h_{ii}^u G_{\hat{i}i}\,,\qquad  \mathring{T}_{ii}:= (\mathring{\mathbf{k}}_i^u)^* G\hat{\mathbf{e}}_i=T_{ii}-h_{ii}^uG_{ii}\,, \label{052901}
\end{align}
where $\tilde{\sigma}_i=\phi_i^u\bar{\phi}_i^v \sigma_i$, and $\sigma_i$ is the $i$th diagonal entry of $\Sigma$, \cf~(\ref{102002}).  
Here in~(\ref{052901}) we used
\begin{align}
\hat{\mathbf{e}}_i^* \wt{ B }^{\la i\ra} =\tilde{\sigma}_i \hat{\mathbf{e}}_{\hat{i}}^*\,,\qquad  \wt{ B }^{\la i\ra}\hat{\mathbf{e}}_{\hat{i}} =\tilde{\sigma}_i \hat{\mathbf{e}}_i\,,\qquad\qquad i\in \llbracket 1, N\rrbracket\,, \label{0830200}
\end{align}
which is checked from  the definitions of $\wt{ B }^{\la i\ra}$ in~(\ref{052703}),  $\mathcal{U}_i$ and  $\Phi_i$ in ~(\ref{052704}),  and also $ B $ in~(\ref{050970}). 

Recall from \eqref{111901} the notations for normalized partial traces $\tau_1$ and $\tau_2$ on $M_{2N}(\mathbb{C})$. Moreover, recall from~\eqref{071603} the definition of the control parameters~$\Lambda_{{\rm d};ii}(z)$,~$\Lambda_{{\rm d}; \hat{i} \hat{i}}(z)$,~$\Lambda_{{\rm d};i\hat{i}}(z)$,~$ \Lambda_{{\rm d};\hat{i}i}(z)$ and~$\Lambda_{{\rm d}}(z)$. We further introduce $ \Lambda_{{\rm d}}^c(z)$ analogously by replacing $\omega_B$ by $\omega_B^c$ (\cf~(\ref{091107})) in the definition of $\Lambda_{{\rm d}}(z)$. We will often omit the variable $z$ from these notations.

In this section we will  show  that  $ \Lambda_{{\rm d}}(z)$, $ \Lambda_{{\rm d}}^c(z)$ and $\Lambda_{T}$ are
of order $\Psi$ with high probability; \ie matrix elements of the Green function
can be expressed in terms of the subordination functions, up to a small random fluctuations of order $\Psi$.
We will refer to these results as {\it Green function subordination}. The main
tool is a high moment calculation and Gaussian integration by parts. However, we cannot  directly estimate  the
high moments of $T_{kl}$ and  the formulas $|G_{ij} - [\ldots]|$ 
defining $\Lambda_{{\rm d};ij}(z)$. Instead, we introduce the following auxiliary quantities.
For each $i\in \llbracket 1, N\rrbracket $ and $j=i$ or $\hat{i}$, let
\begin{align}
&\mathcal{P}_{ij}\equiv\mathcal{P}_{ij}(z):= (\wt{ B }G)_{ij} \tau_1(G) -G_{ij} \tau_1(\wt{ B }G)+\big( G_{ij}+ T_{ij}\big) \Upsilon_{1}\,,\nonumber\\
& \mathcal{P}_{\hat{i}j}\equiv \mathcal{P}_{\hat{i}j}(z):= (\wt{ B }G)_{\hat{i}j} \tau_2(G) -G_{\hat{i}j} \tau_2(\wt{ B }G)+\big( G_{\hat{i}j}+ T_{\hat{i}j}\big) \Upsilon_{2}\,,\nonumber\\
& \mathcal{K}_{ij}\equiv \mathcal{K}_{ij}(z):=  T_{ij}+\tau_1(G)\big(\tilde{\sigma}_i T_{\hat{i}j} +(\wt{B}G)_{ij}\big)-\tau_1(G\wt{B})\big(G_{ij}+T_{ij}\big)\,,\nonumber\\
 &  \mathcal{K}_{\hat{i}j}\equiv \mathcal{K}_{\hat{i}j}(z):=  T_{\hat{i}j}+\tau_2(G)\big(\tilde{\sigma}_i^* T_{ij} +(\wt{B}G)_{\hat{i}j}\big)-\tau_2(G\wt{B})\big(G_{\hat{i}j}+T_{\hat{i}j}\big)\,,
\label{071560}
\end{align} 
where, with $a=1,2$,
\begin{align}
\Upsilon_{a}\equiv \Upsilon_{a}(z)\deq\tau_a(\wt{ B }G)+\tau_a(G)\; \tau_a(\wt{ B } G\wt{ B })-\tau_a(G\wt{ B })\tau_a(\wt{ B }G)\,.  \label{072130}
\end{align}

Using the invariance of the Haar measure, the following Ward identities
\begin{align}
 \E \Upsilon_a =0\,,\qquad\quad a=1,2\,,\label{le Ward for Upsilon}
\end{align}
can be checked. However, we will also need to know that $\Upsilon_a$ are small with high probability and not only in expectation in the following; see \eg~\eqref{071604} in Theorem~\ref{071601} below.

We will compute their high moments of these auxiliary quantities $\mathcal{P}$ and $\mathcal{K}$
and from them we will conclude the estimates on the $\Lambda$'s. 
The careful choice of these auxiliary 
quantities $\mathcal{P}$ and $\mathcal{K}$ is essential for the proof.  They have a built-in cancellation mechanism
that makes the  high moment calculation tractable, see~(\ref{071421})-\eqref{111751} later.

Moreover, we recall the following matrices introduced in~(\ref{0913125}) 
\begin{align*}
 \mathcal{H}=B+\mathcal{U}^*A\mathcal{U}=:B+\wt{A}\,,\quad\qquad \mathcal{G}(z)=(\mathcal{H}-z)^{-1}\,,\qquad z\in\C^+\,,
 \end{align*}
 which are the analogue of $H$ in~(\ref{091301}) and its Green function $G(z)$, obtained via  swapping  
 the r\^oles of $A$ and $B$, and also the r\^oles of
 $\mathcal{U}$ and $\mathcal{U}^*$.
  Note that the structure of $\mathcal{H}$ is exactly the same as $H$, so we can define the $\mathcal{H}$-counterparts
   of all  quantities we have introduced so far for~$H$. We will not repeat the heavy notations of the partial randomness
   decomposition for $\mathcal{H}$ as well, since we will not need all these details. We will only need to know that,
 accordingly,  we can define $\mathcal{G}_{ij}$, $\mathcal{S}_{ij}$ and $\mathcal{T}_{ij}$ by
 applying the same switching in the definitions of $G_{ij}$, $S_{ij}$~and~$T_{ij}$.

  Also note the following alternative definition of $\omega_A^c$ and $\omega_B^c$ in~(\ref{091107}):
 \begin{align}
\omega_A^c(z)\deq z-\frac{\ntr  \wt{A} \mathcal{ G}}{\ntr \mathcal{G}}\,,\qquad \omega_B^c(z)\deq z-\frac{\ntr B \mathcal{G}}{\ntr \mathcal{G}}\,, \label{091361}
 \end{align}
 and the trivial fact $\ntr G=\ntr \mathcal{G}$.
 
 In addition, replacing $\xi_i, \omega_B, G_{ij}$ by $\sigma_i, \omega_A, \mathcal{G}_{ij}$ respectively in~(\ref{071603}), we define 
  $\wt{\Lambda}_{\rm{d};ij}(z)$ and 
 $\wt{\Lambda}_{\rm{d}}(z)$ as the analogues of  ${\Lambda}_{\rm{d};ij}(z)$ and $\Lambda_{\rm {d}}(z)$.  For example
 \begin{equation}\label{tilde quantities}
 \wt{\Lambda}_{{\rm d};ii}(z)=\bigg| \mathcal{G}_{ii}- \frac{\omega_A(z)}{|\sigma_i|^2-(\omega_A(z))^2}\bigg|\,,
 \end{equation}
 and
 \begin{equation}
   \wt{\Lambda}_{\rm{d}}(z):= \max_{i\in \llbracket 1, N\rrbracket}\max_{k,l=i \text{ or }\hat{i}}\wt\Lambda_{{\rm d};kl}(z)\,.
\end{equation}
 Similarly, we can also define $\wt{\Lambda}_{\rm d}^c(z)$ and $\wt{\Lambda}_{T}(z)$ as the analogue of $\Lambda_{\rm d}^c(z)$ and $\Lambda_{T}(z)$, respectively. 
 The analysis of the operator $\mathcal{H}$ is very similar to that of $H$, but at some point it will be useful to
 work with them in tandem, so we will need to control both.

Our main aim in this section is to prove the following Green function subordination property.
 Recall the definition of the control parameter $\Psi(z)$ from \eqref{071902}. 
\begin{thm}\label{thm.green function subordination} Suppose that the assumptions in Theorem~\ref{thm.081501} hold. 
Then
\begin{align}
\Lambda_{{\rm d}}(z)\prec\Psi(z), \quad \wt{\Lambda}_{\rm d}(z)\prec \Psi (z),\quad \Lambda_{T}(z)\prec \Psi(z),\quad \wt{\Lambda}_{T}(z)\prec \Psi(z) \label{091601}
\end{align}
uniformly on $\mathcal{S}_{\mathcal{I}}(\eta_{\mathrm{m}}, \eta_{\mathrm{M}})$, for any  (large) constant $\eta_{\mathrm{M}}>0$ and (small) constant $\gamma>0$,  in the definition of  $\eta_{\mathrm{m}}$ (\cf~(\ref{etamdef})).  Moreover, the estimates
\begin{align}
&\big|\omega_A^c(z)-\omega_A(z)\big|\prec \Psi(z)\,,\qquad \big|\omega_B^c(z)-\omega_B(z)\big|\prec \Psi(z)\,,\nonumber\\
&\big|m_H(z)-m_{\mu_A\boxplus \mu_B}(z)\big|\prec \Psi(z) \label{101470}
\end{align}
also hold uniformly on $\mathcal{S}_{\mathcal{I}}(\eta_{\mathrm{m}},\eta_{\mathrm{M}})$.
\end{thm}
The estimates on the tracial quantities  and the subordination functions 
 in~(\ref{101470}) are weaker than the final result in Theorem~\ref{thm.081501} and Theorem~\ref{le proposition of Green function subordination}. Later in Section~\ref{s.strong law},  we will improve them. The estimates in~(\ref{091601}) are, however,  (believed to be) optimal. 

In what follows, we will mainly work with $\Lambda_{{\rm d}}(z)$. The discussion on $\wt{\Lambda}_{\rm d}(z)$ is the same. 
First, we show the analogous estimate for $\Lambda_{\rm d}^c$ by assuming an a priori bound on $\Lambda_{{\rm d}}$ and $\Lambda_T$, for a fixed $z\in \mathcal{S}_{\mathcal{I}}(\eta_{\mathrm{m}}, \eta_{\mathrm{M}})$. This is the content of Theorem~\ref{071601} below. A continuity argument in Subsection~\ref{sec: continuity argument} then allows us to conclude Theorem~\ref{thm.green function subordination} from Theorem~\ref{071601}.

\begin{thm}\label{071601}Suppose that the assumptions in Theorem~\ref{thm.081501} hold. 
Let  $\eta_{\mathrm{M}}>0$ be a (large) constant  and  $\gamma>0$ be a (small) constant  in \eqref{etamdef}.  Fix a $z\in \mathcal{S}_{\mathcal{I}}(\eta_{\mathrm{m}}, \eta_{\mathrm{M}})$. Assume that
\begin{align}
\Lambda_{{\rm d}}(z)\prec N^{-\frac{\gamma}{4}}\,,\qquad  \wt{\Lambda}_{{\rm d}}(z)\prec N^{-\frac{\gamma}{4}}\,, \qquad \Lambda_T(z)\prec 1\,,\qquad \wt{\Lambda}_T(z)\prec 1\,.  \label{071422}
\end{align}
Then we have  
\begin{align}
&|\mathcal{P}_{ij}(z)|\prec \Psi(z)\,, \qquad  |\mathcal{P}_{\hat{i}j}(z)|\prec \Psi(z)\,,\nonumber\\
& |\mathcal{K}_{ij}(z)|\prec \Psi(z)\,, \qquad  |\mathcal{K}_{\hat{i}j}(z)|\prec \Psi(z)\,, \label{071302}
\end{align}
for all $i\in \llbracket 1, N\rrbracket$ and $j=i$ or $\hat{i}$. In addition, under~(\ref{071422}) we also have
\begin{align}
|\Upsilon_{1}(z)|\prec \Psi(z)\,,\qquad |\Upsilon_{2}(z)|\prec \Psi(z)\,,  \label{071604}
\end{align}
and 
\begin{align}
\Lambda_{{\rm d}}^c(z)\prec\Psi(z)\,,\qquad \Lambda_T(z)\prec \Psi(z)\,. \label{071602}
\end{align}
The same statements  hold if we switch the r\^oles of $A$ and $B$, and also the r\^oles of $U$ and $U^*$, in all  the conclusions from~(\ref{071302}) to~(\ref{071602}).
\end{thm}
Note that, since $\eta_{\mathrm{m}}\leq \eta \leq \eta_{\mathrm{M}}$, we have $\Psi(z)\sim \frac{1}{\sqrt{N\eta}}$.

The proof of Theorem~\ref{071601} proceeds in two steps. In the first step, we establish in Subsection~\ref{sec: recursive moment estimates} recursive moment estimates for the quantities $\mathcal{P}_{ii}$ and $\mathcal{K}_{ii}$. In the second step, carried out in Subsection~\ref{sec: mini stability analysis}, we use a local stability analysis to conclude Theorem~\ref{071601} from the estimates established in Subsection~\ref{sec: recursive moment estimates}.

\subsection{Recursive moment estimates for $\mathcal{P}_{ii}$ and $\mathcal{K}_{ii}$}\label{sec: recursive moment estimates} In the proof of Theorem~\ref{071601}, assumption \eqref{071422} is used to conclude that various $G_{kl}$ and $T_{kl}$ with $k,l=i$ or $\hat{i}$ are
finite. More specifically, with the aid of  assumption~(\ref{071422})  and  with 
 the upper bound of $|\omega_B|$  and  the lower bound on $\Im \omega_B$ in~(\ref{110270})  that together
   imply  that $\omega_B^2$ is away from the positive real axis so the denominators
in  the definition of $\Lambda_{{\rm d},ij}$ do not vanish, we have
\begin{align}
\max_{i\in \llbracket 1, N\rrbracket}\max_{k,l=i \text{ or } \hat{i}} |G_{kl}|\prec 1\,,\qquad \max_{i\in \llbracket 1, N\rrbracket}\max_{k,l=i \text{ or } \hat{i}} |T_{kl}|\prec 1\,. \label{112401}
\end{align}
  In addition, using the  identities in~(\ref{102702}),
we can  further get the bound 
\begin{align}
\max_{i\in \llbracket 1, N\rrbracket}\max_{k,l=i \text{ or } \hat{i}}\big|(XGY)_{kl}\big|\prec 1\,,\qquad X, Y=\hat{I}\,, \text{ or } \wt{B}\,.
   \label{100401}
\end{align}
Observe that 
\begin{align}
\frac{1}{N}\sum_i  \frac{\omega_B(z)}{|\xi_i|^2-(\omega_B(z))^2}=m_{\mu_A}(\omega_B(z)) = m_{\mu_A\boxplus\mu_B}(z)\,, \label{111601}
\end{align} 
where the first step follows from the definition of $\mu_A$ in~(\ref{091105}), and the second step follows from~(\ref{le definiting equations}) with the choice $(\mu_1,\mu_2)=(\mu_A,\mu_B)$. Then,~(\ref{111601}) together with the first estimate in~(\ref{071422}),~(\ref{102702}), and the upper bound of $|\omega_B|$ and the lower bound of $\Im \omega_B$ in~(\ref{110270}) leads to 
 the following estimates for tracial quantities 
\begin{align}
\tau_a(G)&= m_{\mu_A\boxplus\mu_B}+O_\prec(N^{-\frac{\gamma}{4}})\,,\qquad\qquad a=1,2\,.\nonumber\\
\tau_a(\wt{B}G)&= (z-\omega_B) m_{\mu_A\boxplus\mu_B} +O_\prec(N^{-\frac{\gamma}{4}})\,,\nonumber\\
\tau_a(G\wt{B})&= (z-\omega_B) m_{\mu_A\boxplus\mu_B} +O_\prec(N^{-\frac{\gamma}{4}})\,,\nonumber\\
\tau_a(\wt{B}G\wt{B})&=(\omega_B-z)(1+(\omega_B-z)m_{\mu_A\boxplus\mu_B})+O_\prec(N^{-\frac{\gamma}{4}})\,. 
\label{102125}
\end{align}
Then, using the upper bound on $|\omega_B|$ and the lower bound on $\Im \omega_B$ in~(\ref{110270}), and the second identity in~(\ref{111601}), we see that all these tracial quantities are stochastically dominated by~$1$, under assumption~(\ref{071422}). Recalling $\Upsilon_a$ from~\eqref{072130}, we thus have under assumption~\eqref{071422}~that
\begin{align}
|\Upsilon_a(z)|\prec 1\,.
\end{align}

For~(\ref{071302}), we only handle the estimate of $\mathcal{P}_{ii}$ and  $\mathcal{K}_{ii}$ in detail.  The others are similar. 
It suffices to show the high order moment estimate: for any fixed integer $p\geq 1$, we have
\begin{align}
\mathbb{E}\big[| \mathcal{P}_{ii}|^{2p}\big]\prec \Psi^{2p}\,, \qquad  \mathbb{E} \big[ |\mathcal{K}_{ii}|^{2p}\big]\prec \Psi^{2p}\,. \label{071305}
\end{align}
Let us introduce the notation 
\begin{align}
\mathfrak{m}_i(k,l)\deq\mathcal{P}_{ii}^k\overline{\mathcal{P}_{ii}^l}\,, \qquad  \mathfrak{n}_i(k,l)\deq \mathcal{K}_{ii}^k\overline{\mathcal{K}_{ii}^l}\,. \label{102102}
\end{align}

We will use the following notational conventions in the statement of the recursive moment estimates. The notation $O_\prec(\Psi^k)$ for any given positive integer $k$, represents a generic (possibly) $z$-dependent random variable $X\equiv X(z)$ that satisfies 
\begin{align}
X\prec \Psi^k\,,\qquad \mathbb{E} [|X|^q]\prec \Psi^{qk}\,, \label{100402}
\end{align}
for any given positive integer $q$. In the sequel, we only check the first bound in~(\ref{100402}) for various $X$'s, 
 then the second bound is   valid as well.  Indeed, 
since the $X$'s we will encounter below are analogous to those in \cite{BES15c},  we refer to the paragraph below (6.2) of \cite{BES15c} for a general reasoning  why   the second bound in~(\ref{100402}) follows from the first one. Additionally, sometimes $X$ will be of the form  $1/|g|$
where $g$ is an $N$-dimensional Gaussian random variable (see \eg~\eqref{071435}-\eqref{111750}),
whose  $q$th moments are also integrable for any fixed~$q$ if $N$ is large enough.

The main technical task in the proof of~\eqref{071305} is the following recursive moment estimate.
\begin{lem}[Recursive moment estimate for $\mathcal{P}_{ii}$ and  $\mathcal{K}_{ii}$] \label{lem.071420} Suppose the assumptions of Theorem~\ref{071601} hold.  For any fixed integer $p\geq 1$,   and for any $i\in \llbracket 1, N\rrbracket$, we~have
\begin{align}
\mathbb{E}[\mathfrak{m}_i(p,p)]&=\mathbb{E}[O_\prec(\Psi) \mathfrak{m}_i(p-1,p)]+\mathbb{E}[O_\prec(\Psi^2) \mathfrak{m}_i(p-2,p)]\nonumber\\ &\qquad\qquad+\mathbb{E}[O_\prec(\Psi^2) \mathfrak{m}_i(p-1,p-1)]\,, \nonumber\\
\mathbb{E}[\mathfrak{n}_i(p,p)]&=\mathbb{E}[O_\prec(\Psi) \mathfrak{n}_i(p-1,p)]+\mathbb{E}[O_\prec(\Psi^2) \mathfrak{n}_i(p-2,p)]\nonumber\\ &\qquad\qquad+\mathbb{E}[O_\prec(\Psi^2) \mathfrak{n}_i(p-1,p-1)]\,, 
\label{071503}
\end{align}
where we made the convention $\mathfrak{m}_i(0,0)=\mathfrak{n}_i(0,0)=1$ and $\mathfrak{m}_i(-1,1)=\mathfrak{n}_i(-1,1)=0$ if $p=1$.
\end{lem}

\begin{proof}[Proof of Lemma~\ref{lem.071420}]  
According to the decomposition in~(\ref{081101}), for $i\in \llbracket 1,N\rrbracket$, we have
\begin{align}
(\wt{ B }G)_{ii}=\hat{\mathbf{e}}_i^* \mathcal{R}_i \wt{ B }^{\la i\ra} \mathcal{R}_i G\hat{\mathbf{e}}_i=-\big((\mathbf{h}_i^u)^*,\mathbf{0}^*\big) \wt{ B }^{\la i\ra} \mathcal{R}_i G\hat{\mathbf{e}}_i=-(\mathbf{k}_i^u)^* \wt{ B }^{\la i\ra} \mathcal{R}_i G\hat{\mathbf{e}}_i\,, \label{050811}
\end{align}
where in the second step we used~(\ref{050805}), and in the last step we used the notation in ~(\ref{052801}). 
Using~(\ref{050811}), the definition in~(\ref{def of R}) , and also the identity in~(\ref{050820}),   one can  check
\begin{align}
(\wt{ B }G)_{ii}&=-(\mathbf{k}_i^u)^* \wt{ B }^{\la i\ra} \big(\hat{I}-\mathbf{r}_i^u(\mathbf{r}_i^u)^*\oplus \mathbf{r}_i^v(\mathbf{r}_i^v)^*\big) G\hat{\mathbf{e}}_i \nonumber\\
&= -S_{ii}+(\mathbf{k}_i^u)^* \wt{ B }^{\la i\ra} \big(\mathbf{r}_i^u(\mathbf{r}_i^u)^*\oplus \mathbf{r}_i^v(\mathbf{r}_i^v)^*\big) G\hat{\mathbf{e}}_i\nonumber\\
&= -S_{ii}+(\mathbf{k}_i^u)^* \wt{ B }^{\la i\ra} \big(0\oplus \mathbf{r}_i^v(\mathbf{r}_i^v)^*\big) G\hat{\mathbf{e}}_i\nonumber\\
&= -S_{ii}+(\ell_i^v)^2(\mathbf{k}_i^u)^* \wt{ B }^{\la i\ra} (\hat{\mathbf{e}}_{\hat{i}}+\mathbf{k}_i^v)
(\hat{\mathbf{e}}_{\hat{i}}+\mathbf{k}_i^v)^* G\hat{\mathbf{e}}_i\nonumber\\
&=-S_{ii}+(\ell_i^v)^2\big(\tilde{\sigma}_i h_{ii}^u+(\mathbf{k}_i^u)^* \wt{ B }^{\la i\ra} \mathbf{k}_i^v\big)\big(G_{\hat{i}i}+T_{\hat{i}i}\big)\nonumber\\
&=:-\mathring{S}_{ii}+\varepsilon_{i1}\,,  \label{050930}
\end{align}
where  $0$ in the third line is the $N\times N$ zero matrix, and 
\begin{align}
\varepsilon_{i1}:=\Big(\big((\ell_i^v)^2-1\big)\tilde{\sigma}_i h_{ii}^u+(\ell_i^v)^2(\mathbf{k}_i^u)^* \wt{ B }^{\la i\ra} \mathbf{k}_i^v\Big)G_{\hat{i}i}+(\ell_i^v)^2\big(\tilde{\sigma}_i h_{ii}^u+(\mathbf{k}_i^u)^* \wt{ B }^{\la i\ra} \mathbf{k}_i^v\big)T_{\hat{i}i}\,.
\end{align}
 In the third step of~(\ref{050930}) we used the fact $(\mathbf{k}_i^u)^* \wt{ B }^{\la i\ra}(\mathbf{r}_i^u(\mathbf{r}_i^u)^*\oplus 0)=0$ which follows from the definition of $\mathbf{k}_i^u$ and $\wt{B}^{\la i\ra}$ in~(\ref{052801}) and 
(\ref{052703});  in the fifth step we used the second identity in \eqref{0830200}; and in the last step, we used~(\ref{052901}). We note that
\begin{align} \label{083011}
 |\varepsilon_{i1}|\prec\Psi\,,
\end{align}
where we used~\eqref{112401} and the large deviation bound~\eqref{091731} to show that  $(\mathbf{k}_i^u)^* \wt{ B }^{\la i\ra} \mathbf{k}_i^v \prec N^{-1/2}$.

Using integration by parts, we note that
\begin{align}
\int_{\mathbb{C}} \bar{g} f(g,\bar{g}) \mathrm{e}^{-\frac{|g|^2}{\sigma^2}} \,\dd^2g=\sigma^2 \int_\mathbb{C} \partial_g  f(g,\bar{g})  \mathrm{e}^{-\frac{|g|^2}{\sigma^2}} \dd^2g\,, \label{integration by parts formula}
\end{align}
for differentiable functions $f: \mathbb{C}^2 \to \mathbb{C}$  (recall that $\dd^2g$ is
 the Lebesgue measure on $\mathbb{C}$). 

According to the definitions in~(\ref{def of R}),~(\ref{052802}),  and the identity~(\ref{050820}), one can check for $k\neq i$, 
\begin{align}
\frac{\partial R_i^a}{\partial g_{ik}^a}=-\frac{(\ell_i^a)^2}{\|\mathbf{g}_i^a\|_2} \mathbf{e}_k\big(\mathbf{e}_i+\mathbf{h}_i^a\big)^*+\Delta_R^a(i,k)\,,\qquad a=u,v\,. \label{050903}
\end{align}
where
\begin{align}
\Delta_R^a(i,k):=&\frac{(\ell_i^a)^2}{2\|\mathbf{g}_i^a\|_2^2}\bar{g}_{ik}^a\big(\mathbf{e}_i(\mathbf{h}_i^a)^*+\mathbf{h}_i^a\mathbf{e}_i^*+2\mathbf{h}_i^a(\mathbf{h}_i^a)^*\big)\nonumber\\
&-\frac{(\ell_i^a)^4}{2\|\mathbf{g}_i^a\|_2^3} g_{ii}^a \bar{g}_{ik}^a\big(\mathbf{e}_i+\mathbf{h}_i^a\big)\big(\mathbf{e}_i+\mathbf{h}_i^a\big)^*\,,\qquad a=u,v\,. \label{0916100}
\end{align}
The $\Delta_R^a(i,k)$'s are irrelevant error terms. Their estimates will be  presented  separately  in Appendix~\ref{Appendix B}. For convenience, we set for $a=u,v$, 
\begin{align}
c_i^a\deq\frac{(\ell_i^a)^2}{\|\mathbf{g}_i^a\|_2}=\frac{1}{\|\mathbf{g}_i^a\|_2}-h_{ii}^a+O_\prec(\frac{1}{N})=\|\mathbf{g}_i^a\|_2-h_{ii}^a-\big(\|\mathbf{g}_i^a\|_2^2-1\big)+O_\prec(\frac{1}{N})\,, \label{053095}
\end{align} 
where the last step follows from~(\ref{050901}). Using~(\ref{081101}), we have  for $k\neq i$
\begin{align}
\frac{\partial G}{\partial g_{ik}^u}=-G\frac{\partial \wt{ B }}{\partial g_{ik}^u} G=-G\frac{\partial \mathcal{R}_i}{\partial g_{ik}^u} \wt{ B }^{\la i\ra} \mathcal{R}_i G-G\mathcal{R}_i \wt{ B }^{\la i\ra} \frac{\partial \mathcal{R}_i}{\partial g_{ik}^u} G\,. \label{050911}
\end{align}
According to~(\ref{050903}) and the fact $\mathcal{R}_i=R_i^u\oplus R_i^v$, we have
\begin{align}
\frac{\partial \mathcal{R}_i}{\partial g_{ik}^u}= -c_i^u \hat{\mathbf{e}}_k\big(\hat{\mathbf{e}}_i+\mathbf{k}_i^u\big)^*+\Delta_R^a(i,k)\oplus 0\,, \label{050910}
\end{align}
where  $0$  is the $N\times N$ zero matrix.
 We also used that $\partial  R_i^v/\partial g_{ik}^u =0$. 
Plugging~(\ref{050910}) into~(\ref{050911}), for $k\neq i$, we can write 
\begin{align}
\frac{\partial G}{\partial g_{ik}^u}=c_i^u  G\hat{\mathbf{e}}_k\big(\hat{\mathbf{e}}_i^*+(\mathbf{k}_i^u)^*\big) \wt{ B }^{\la i\ra} \mathcal{R}_i G+c_i^uG\mathcal{R}_i \wt{ B }^{\la i\ra} \hat{\mathbf{e}}_k\big(\hat{\mathbf{e}}_i^*+(\mathbf{k}_i^u)^*\big) G+\Delta_{G}^u(i,k)\,, \label{050915}
\end{align}
where we set 
\begin{align}
\Delta_{G}^u(i,k)\deq-G(\Delta_R^u(i,k)\oplus 0) \wt{ B }^{\la i\ra} \mathcal{R}_i G-G\mathcal{R}_i \wt{ B }^{\la i\ra} (\Delta_R^u(i,k)\oplus 0) G\,. \label{0916101}
\end{align}
With the above derivatives, we are ready to apply the integration by parts formula in~(\ref{integration by parts formula}).
We start with the following
\begin{align}
\mathbb{E}[\mathfrak{m}_i(p,p)]&=\mathbb{E}[\mathcal{P}_{ii} \mathfrak{m}_i(p-1,p)]=\mathbb{E}[(\wt{ B }G)_{ii} \tau_1(G)  \mathfrak{m}_i(p-1,p)]\nonumber\\
&\quad+\mathbb{E}\big[\big(-G_{ii} \tau_1(\wt{ B }G)+\big( G_{ii}+ T_{ii}\big) \Upsilon_{1}\big)\mathfrak{m}_i(p-1,p)\big]\,,\label{07142100}\\
\mathbb{E}[\mathfrak{n}_i(p,p)]&=\mathbb{E}[\mathcal{K}_{ii}\mathfrak{n}_i(p-1,p)]=\mathbb{E}[T_{ii}  \mathfrak{n}_i(p-1,p)] \nonumber\\
&\quad +\mathbb{E}\big[\big(\tau_1(G)\big(\tilde{\sigma}_i T_{\hat{i}j} +(\wt{B}G)_{ij}\big)-\tau_1(G\wt{B})\big(G_{ij}+T_{ij}\big)\big)\mathfrak{n}_i(p-1,p)\big]\,,\label{071421}
\end{align}
which follow from  the definitions in~(\ref{071560}) and~(\ref{102102}) directly. From~(\ref{050930}) and~(\ref{052901}), we have
\begin{align}
\mathbb{E}[(\wt{ B }G)_{ii} \tau_1(G) \mathfrak{m}_i(p-1,p)]&= -\mathbb{E}[\mathring{S}_{ii}\tau_1(G) \mathfrak{m}_i(p-1,p)]\nonumber\\ &\qquad\qquad+\mathbb{E}[\varepsilon_{i1}\tau_1(G)\mathfrak{m}_i(p-1,p)]\,, \label{071430}\\
\mathbb{E}[T_{ii}  \mathfrak{n}_i(p-1,p)] &= \mathbb{E}[\mathring{T}_{ii}  \mathfrak{n}_i(p-1,p)] +\mathbb{E}[O_\prec(\Psi)\mathfrak{n}_i(p-1,p)]\,, \label{111751}
\end{align}
where we used the fact $|h_{ii}|\prec N^{-\frac{1}{2}}$, and also~(\ref{112401}).

 Now we will carefully compute the first terms in the right hand side of \eqref{071430} and \eqref{111751} 
with the integration by parts formula since both $\mathring{S}_{ii}$ and $\mathring{T}_{ii}$ explicitly contain a
multiplicative Gaussian factor. We will then find that the leading term of the result of this calculation
will exactly cancel the last  quantities in the right side of  equations in~(\ref{07142100}) and \eqref{071421}. 
This cancellation is the key point of the following tedious calculation and this is the main reason for defining 
the key quantities $\mathcal{P}_{ii}$ and  $\mathcal{K}_{ii}$ in the form they are given in  \eqref{071560}.

For the first term on the right side of~(\ref{071430}), using the definition of $\mathring{S}_{ii}$ in~(\ref{052901}) and  the integration by parts formula in~(\ref{integration by parts formula}), we have
\begin{align}
\mathbb{E}[\mathring{S}_{ii}\tau_1(G) \mathfrak{m}_i(p-1,p)]&=\sum_{k}^{(i)} \mathbb{E} \Big[ \bar{g}^u_{ik}\frac{1}{\| \mathbf{g}^u_i\|_2} \hat{\mathbf{e}}_k^*\wt{ B }^{\la i\ra}  G\hat{\mathbf{e}}_i \tau_1(G) \mathfrak{m}_i(p-1,p)\Big]\nonumber\\
&=\frac{1}{N} \sum_{k}^{(i)} \mathbb{E} \Big[ \frac{1}{\|\mathbf{g}_{i}^u\|_2}\frac{  \partial (\hat{\mathbf{e}}_k^*\wt{ B }^{\la i\ra}  G\hat{\mathbf{e}}_i)}{\partial g^{u}_{ik}} \tau_1(G) \mathfrak{m}_i(p-1,p)\Big]\nonumber\\
&\quad +\frac{1}{N}\sum_{k}^{(i)} \mathbb{E} \Big[ \frac{ \partial \| \mathbf{g}^u_i\|_2^{-1}}{\partial g_{ik}^u} \hat{\mathbf{e}}_k^*\wt{ B }^{\la i\ra}  G\hat{\mathbf{e}}_i \tau_1(G) \mathfrak{m}_i(p-1,p)\Big]\nonumber\\
&\quad + \frac{1}{N}\sum_{k}^{(i)} \mathbb{E} \Big[  \frac{1}{\|\mathbf{g}_{i}^u\|_2} \hat{\mathbf{e}}_k^*\wt{ B }^{\la i\ra}  G\hat{\mathbf{e}}_i \frac{\partial \tau_1(G)}{\partial g^u_{ik}}  \mathfrak{m}_i(p-1,p)\Big] \nonumber\\
&\quad +\frac{p-1}{N} \sum_{k}^{(i)} \mathbb{E} \Big[ \frac{1}{\| \mathbf{g}^u_i\|_2} \hat{\mathbf{e}}_k^*\wt{ B }^{\la i\ra}  G\hat{\mathbf{e}}_i \tau_1(G) \frac{\partial \mathcal{P}_{ii}}{\partial g_{ik}^u}\mathfrak{m}_i(p-2,p)\Big]\nonumber\\
&\quad +  \frac{p}{N} \sum_{k}^{(i)} \mathbb{E} \Big[\frac{1}{\| \mathbf{g}^u_i\|_2} \hat{\mathbf{e}}_k^*\wt{ B }^{\la i\ra}  G\hat{\mathbf{e}}_i \tau_1(G) \frac{\partial \overline{\mathcal{P}_{ii}}}{\partial g_{ik}^u}\mathfrak{m}_i(p-1,p-1)\Big]\,. \label{071435}
\end{align}
Analogously, we have
\begin{align}
\mathbb{E}[\mathring{T}_{ii}  \mathfrak{n}_i(p-1,p)] 
&=\frac{1}{N} \sum_{k}^{(i)} \mathbb{E} \Big[ \frac{1}{\|\mathbf{g}_{i}^u\|_2}\frac{  \partial (\hat{\mathbf{e}}_k^* G\hat{\mathbf{e}}_i)}{\partial g^{u}_{ik}}  \mathfrak{n}_i(p-1,p)\Big]\nonumber\\
&\qquad +\frac{1}{N}\sum_{k}^{(i)} \mathbb{E} \Big[ \frac{ \partial \| \mathbf{g}^u_i\|_2^{-1}}{\partial g_{ik}^u} \hat{\mathbf{e}}_k^*  G\hat{\mathbf{e}}_i  \mathfrak{n}_i(p-1,p)\Big]\nonumber\\
&\qquad +\frac{p-1}{N} \sum_{k}^{(i)} \mathbb{E} \Big[ \frac{1}{\| \mathbf{g}^u_i\|_2} \hat{\mathbf{e}}_k^* G\hat{\mathbf{e}}_i  \frac{\partial \mathcal{K}_{ii}}{\partial g_{ik}^u}\mathfrak{n}_i(p-2,p)\Big]\nonumber\\
&\qquad +  \frac{p}{N} \sum_{k}^{(i)} \mathbb{E} \Big[\frac{1}{\| \mathbf{g}^u_i\|_2} \hat{\mathbf{e}}_k^*G\hat{\mathbf{e}}_i  \frac{\partial \overline{\mathcal{K}_{ii}}}{\partial g_{ik}^u}\mathfrak{n}_i(p-1,p-1)\Big]\,. \label{111750}
\end{align}

We start from the first term on the right side of~(\ref{071435}).  Using~(\ref{050915}), we have
\begin{align}
\frac{1}{N}\sum_{k}^{(i)} \frac{  \partial (\hat{\mathbf{e}}_k^*\wt{ B }^{\la i\ra}  G\hat{\mathbf{e}}_i)}{\partial g^{u}_{ik}}=&c_i^u\frac{1}{N}  \sum_{k}^{(i)} \hat{\mathbf{e}}_k^*\wt{ B }^{\la i\ra}  G\hat{\mathbf{e}}_k\big(\hat{\mathbf{e}}_i+\mathbf{k}_i^u\big)^* \wt{ B }^{\la i\ra} \mathcal{R}_i G\hat{\mathbf{e}}_i\nonumber\\
&\hspace{-15ex}+c_i^u\frac{1}{N}  \sum_{k}^{(i)} \hat{\mathbf{e}}_k^*\wt{ B }^{\la i\ra}  G\mathcal{R}_i \wt{ B }^{\la i\ra} \hat{\mathbf{e}}_k\big(\hat{\mathbf{e}}_i+\mathbf{k}_i^u\big)^* G\hat{\mathbf{e}}_i+\frac{1}{N}  \sum_{k}^{(i)} \hat{\mathbf{e}}_k^*\wt{ B }^{\la i\ra}\Delta_{G}^u(i,k)\hat{\mathbf{e}}_i\,. \label{083015}
\end{align}
Let
\begin{align}
\varepsilon_{i2}:=\frac{1}{N}  \sum_{k}^{(i)} \hat{\mathbf{e}}_k^*\wt{ B }^{\la i\ra}\Delta_{G}^u(i,k)\hat{\mathbf{e}}_i\,. \label{112470}
\end{align}
Note that 
\begin{align}
\frac{1}{N}\sum_{k}^{(i)}\hat{\mathbf{e}}_k^*\wt{ B }^{\la i\ra}  G\hat{\mathbf{e}}_k=\tau_1(\wt{B}^{\la i\ra}G)-\frac{1}{N}(\wt{B}^{\la i\ra}G)_{ii}=  \tau_1(\wt{B}G)+O_\prec(\Psi^2) \,, 
\label{083020}
\end{align}
where in the last step we used the second estimate in Corollary~\ref{cor.finite rank} with the choice $Q=\hat{I}_1$ (\cf~(\ref{071540})), $(\wt{B}^{\la i\ra}G)_{ii}=\tilde{\sigma}_iG_{\hat{i}i}$ (\cf ~(\ref{0830200})), and the bound in~(\ref{112401}). Analogously, one shows 
\begin{align}
\frac{1}{N}\sum_{k}^{(i)}\hat{\mathbf{e}}_k^*\wt{ B }^{\la i\ra}  G\mathcal{R}_i \wt{ B }^{\la i\ra} \hat{\mathbf{e}}_k= \tau_1(\wt{B}G\wt{B})+O_\prec(\Psi^2)\,. \label{083021}
\end{align}
Moreover, using~(\ref{0830200}),~(\ref{050805}) and the fact $\mathcal{R}_i^2=\hat{I}$,  we also have the following observations
\begin{align}
&\hat{\mathbf{e}}_i^*\wt{ B }^{\la i\ra} \mathcal{R}_i G\hat{\mathbf{e}}_i=\tilde{\sigma}_i\hat{\mathbf{e}}_{\hat{i}}^*\mathcal{R}_i G\hat{\mathbf{e}}_i=-\tilde{\sigma}_i(\mathbf{k}_i^v)^*G\hat{\mathbf{e}}_i=-\tilde{\sigma}_iT_{\hat{i}i}\,,\nonumber\\
& (\mathbf{k}_i^u)^* \wt{ B }^{\la i\ra} \mathcal{R}_i G\hat{\mathbf{e}}_i= (\mathbf{k}_i^u)^*\mathcal{R}_i \wt{ B } G\hat{\mathbf{e}}_i=-\hat{\mathbf{e}}_i^*\wt{ B } G\hat{\mathbf{e}}_i=-(\wt{ B }G)_{ii}\,. \label{071553}
\end{align}
Plugging~(\ref{083020}),~(\ref{083021}) and~(\ref{071553}) into~(\ref{083015}), we obtain
\begin{multline}
\frac{1}{N}\sum_{k}^{(i)} \frac{  \partial (\hat{\mathbf{e}}_k^*\wt{ B }^{\la i\ra}  G\hat{\mathbf{e}}_i)}{\partial g^{u}_{ik}}= - c_i^u\tau_1(\wt{B}G)\big(\tilde{\sigma}_iT_{\hat{i}i}+(\wt{ B }G)_{ii}\big)\\+c_i^u \tau_1(\wt{B}G\wt{B})\big( G_{ii}+T_{ii}\big)+\varepsilon_{i2}+O_\prec(\Psi^2)\,.\label{071440}
\end{multline}
Analogously to~(\ref{071440}), we also have 
\begin{multline}
\frac{1}{N}\sum_{k}^{(i)} \frac{  \partial (\hat{\mathbf{e}}_k^*  G\hat{\mathbf{e}}_i)}{\partial g^{u}_{ik}}= -c_i^u \tau_1(G)\big(\tilde{\sigma}_iT_{\hat{i}i}+(\wt{ B }G)_{ii}\big)\\
+ c_i^u\tau_1(G\wt{ B })\big( G_{ii}+T_{ii}\big)+\varepsilon_{i3}+O_\prec(\Psi^2)\,, \label{071441}
\end{multline}
where
\begin{align*}
\varepsilon_{i3}:=\frac{1}{N}  \sum_{k}^{(i)} \hat{\mathbf{e}}_k^*\Delta_{G}^u(i,k)\hat{\mathbf{e}}_i\,.
\end{align*}
The following estimates on $\varepsilon_{i2}$ and $\varepsilon_{i3}$ will be proved in Lemma~\ref{lem.100655} in Appendix~\ref{Appendix B}.
\begin{align}
|\varepsilon_{i2}|\prec \Psi^2,\qquad |\varepsilon_{i3}|\prec \Psi^2\,. \label{083050}
\end{align}
Combining~(\ref{071440}), ~(\ref{071441}) with an appropriate linear combination  and  using ~(\ref{083050}), we get
\begin{multline}
\frac{1}{N}\sum_{k}^{(i)} \frac{  \partial (\hat{\mathbf{e}}_k^*\wt{ B }^{\la i\ra}  G\hat{\mathbf{e}}_i)}{\partial g^{u}_{ik}} \tau_1(G)- \frac{1}{N}\sum_{k}^{(i)} \frac{  \partial (\hat{\mathbf{e}}_k^*  G\hat{\mathbf{e}}_i)}{\partial g^{u}_{ik}} \tau_1(\wt{ B }G)\\
=-c_i^u\big( G_{ii}+T_{ii}\big) \Big( \tau_1(\wt{ B }G)-\Upsilon_{1}\Big)+O_\prec(\Psi^2)\,. \label{071450}
\end{multline}
Here we also used that the tracial quantities $\tau_1(G)$, $ \tau_1(\wt{ B }G)$ and  $\Upsilon_{1}$   are stochastically dominated by $1$, in light of~(\ref{102125}). 
Applying ~(\ref{053095}), the fact $\mathring{T}_{ii}=T_{ii}-h_{ii}^uG_{ii}$ from~(\ref{052901}),  we can write
\begin{align}
&\frac{1}{N}\sum_{k}^{(i)} \frac{  \partial (\hat{\mathbf{e}}_k^*\wt{ B }^{\la i\ra}  G\hat{\mathbf{e}}_i)}{\partial g^{u}_{ik}} \tau_1(G)=-c_i^u(G_{ii}+T_{ii})\Big(\tau_1(\wt{ B }G)-\Upsilon_{1}\Big)\nonumber\\ &\qquad\qquad\qquad\qquad+\frac{1}{N}\sum_{k}^{(i)}\frac{\partial (\hat{\mathbf{e}}_k^*G\hat{\mathbf{e}}_i)}{\partial g_{ik}^u}\tau_1(\wt{B}G)+O_\prec(\Psi^2)\nonumber\\
&=-c_i^u(G_{ii}+T_{ii})\Big(\tau_1(\wt{ B }G)-\Upsilon_{1}\Big)+\mathring{T}_{ii} \tau_1(\wt{ B }G)\nonumber\\ &\qquad\qquad\qquad\qquad+\Big(\frac{1}{N}\sum_{k}^{(i)}\frac{\partial (\hat{\mathbf{e}}_k^*G\hat{\mathbf{e}}_i)}{\partial g_{ik}^u}-\mathring{T}_{ii} \Big)\tau_1(\wt{ B }G)+O_\prec(\Psi^2)\nonumber\\
& =-\|\mathbf{g}_i^u\|_2\Big( G_{ii}\tau_1(\wt{ B }G)-(G_{ii}+T_{ii})\Upsilon_{1}\Big)+\Big(\frac{1}{N}\sum_{k}^{(i)}\frac{\partial (\hat{\mathbf{e}}_k^*G\hat{\mathbf{e}}_i)}{\partial g_{ik}^u}-\mathring{T}_{ii} \Big)\tau_1(\wt{ B }G)\nonumber\\
&\qquad\quad+ \varepsilon_{i4}+\varepsilon_{i5}+O_\prec(\Psi^2)\,, \label{083030}
\end{align}
where 
\begin{align}
\varepsilon_{i4}&\deq \big((1-\|\mathbf{g}_i^u\|_2^2)\tau_1\big( \wt{B} G\big)+(1-\|\mathbf{g}_i^u\|_2^2-h_{ii})(\tau_1\big( \wt{B}G\big)-\Upsilon_1)\big)T_{ii}\nonumber  \label{112001}\\
&\qquad\qquad+(1-\|\mathbf{g}_i^u\|_2^2-h_{ii})  G_{ii} \Upsilon_1\,,\\
\varepsilon_{i5}&\deq\big( \|\mathbf{g}_i^u\|_2^2-1\big)G_{ii}\tau_1(\wt{ B }G)\,. \label{083190}
\end{align}
Using $\|\mathbf{g}_i^u\|_2=1+O_\prec(\frac{1}{\sqrt{N}})$, the estimates~(\ref{112401}), ~(\ref{100401}) and~(\ref{102125}), and Corollary~\ref{cor.finite rank}, we~get
\begin{align}
|\varepsilon_{i4}|\prec \frac{1}{\sqrt{N}}\,,\qquad |\varepsilon_{i5}|\prec \frac{1}{\sqrt{N}}\,.  \label{1023100}
\end{align}

Notice that the first term in the right side of \eqref{083030} will exactly cancel  the explicit last term in
the right side of~(\ref{07142100}). This cancellation is one of the main reasons behind the choice of
the auxiliary quantity $\mathcal{P}$.
Combining the first equation of~(\ref{071421}),~(\ref{071430}),~(\ref{071435}) with ~(\ref{083030}), we~get
\begin{align}
\mathbb{E}[\mathfrak{m}_i(p,p)]= &\mathbb{E} \Big[ \frac{1}{\|\mathbf{g}_{i}^u\|_2}\Big(\mathring{T}_{ii}- \frac{1}{N} \sum_{k}^{(i)} \frac{  \partial (\hat{\mathbf{e}}_k^*  G\hat{\mathbf{e}}_i)}{\partial g^{u}_{ik}} \Big)\tau_1(\wt{ B }G)  \mathfrak{m}_i(p-1,p)\Big]\nonumber\\
& -\frac{1}{N}\sum_{k}^{(i)} \mathbb{E} \Big[ \frac{ \partial \| \mathbf{g}^u_i\|_2^{-1}}{\partial g_{ik}^u} \hat{\mathbf{e}}_k^*\wt{ B }^{\la i\ra}  G\hat{\mathbf{e}}_i \tau_1(G) \mathfrak{m}_i(p-1,p)\Big]\nonumber\\
& - \frac{1}{N}\sum_{k}^{(i)} \mathbb{E} \Big[  \frac{1}{\|\mathbf{g}_{i}^u\|_2} \hat{\mathbf{e}}_k^*\wt{ B }^{\la i\ra}  G\hat{\mathbf{e}}_i \frac{\partial \tau_1(G)}{\partial g^u_{ik}}  \mathfrak{m}_i(p-1,p)\Big] \nonumber\\
& -\frac{p-1}{N} \sum_{k}^{(i)} \mathbb{E} \Big[ \frac{1}{\| \mathbf{g}^u_i\|_2} \hat{\mathbf{e}}_k^*\wt{ B }^{\la i\ra}  G\hat{\mathbf{e}}_i \tau_1(G)  \frac{\partial \mathcal{P}_{ii}}{\partial g_{ik}^u}\mathfrak{m}_i(p-2,p)\Big]\nonumber\\
& -  \frac{p}{N} \sum_{k}^{(i)} \mathbb{E} \Big[\frac{1}{\| \mathbf{g}^u_i\|_2} \hat{\mathbf{e}}_k^*\wt{ B }^{\la i\ra}  G\hat{\mathbf{e}}_i \tau_1(G)  \frac{\partial \overline{\mathcal{P}_{ii}}}{\partial g_{ik}^u}\mathfrak{m}_i(p-1,p-1)\Big]\nonumber\\
&+\mathbb{E}\Big[\Big(\varepsilon_{i1} \tau_1(G) - \frac{\varepsilon_{i4}+\varepsilon_{i5}}{\|\mathbf{g}_i^u\|_2}\Big) \mathfrak{m}_i(p-1,p)\Big]+\mathbb{E}\big[O_\prec(\Psi^2) \mathfrak{m}_i(p-1,p)\big]\,. \label{071501}
\end{align}
Note that the sixth term on the right side can be estimated  by $\mathbb{E}\big[O_\prec(\Psi) \mathfrak{m}_i(p-1,p)\big]$, according to~(\ref{083011}) and~(\ref{1023100}). This estimate is sufficient for the proof of Lemma~\ref{lem.071420}. But here we keep the $\varepsilon$-terms explicit for further use.

 In order to estimate the first term in the right side, 
  similarly to~(\ref{111750}), we can apply the integration by parts formula~(\ref{integration by parts formula}) to obtain
\begin{align}
&\mathbb{E} \Big[ \frac{1}{\|\mathbf{g}_{i}^u\|_2} \Big(\mathring{T}_{ii}- \frac{1}{N} \sum_{k}^{(i)} \frac{  \partial (\hat{\mathbf{e}}_k^*  G\hat{\mathbf{e}}_i)}{\partial g^{u}_{ik}} \Big)\tau_1(\wt{ B }G)  \mathfrak{m}_i(p-1,p)\Big]\nonumber\\
&\qquad= \frac{1}{N} \sum_{k}^{(i)}\mathbb{E} \Big[ \frac{\partial \|\mathbf{g}_{i}^u\|_2^{-2}}{\partial g_{ik}^u}  \hat{\mathbf{e}}_k^*  G\hat{\mathbf{e}}_i\tau_1(\wt{ B }G)  \mathfrak{m}_i(p-1,p)\Big]\nonumber\\
&\qquad\qquad +\frac{p-1}{N} \sum_{k}^{(i)}\mathbb{E} \Big[ \frac{1}{\|\mathbf{g}_{i}^u\|_2^2}\hat{\mathbf{e}}_k^*  G\hat{\mathbf{e}}_i  \frac{\partial \tau_1(\wt{ B }G)}{\partial g_{ik}^u}\mathfrak{m}_i(p-1,p)\Big] \nonumber\\
&\qquad\qquad +\frac{p-1}{N} \sum_{k}^{(i)}\mathbb{E} \Big[  \frac{1}{\|\mathbf{g}_{i}^u\|_2^2}\hat{\mathbf{e}}_k^*  G\hat{\mathbf{e}}_i \tau_1(\wt{ B }G) \frac{\partial \mathcal{P}_{ii}}{\partial g_{ik}^u}\mathfrak{m}_i(p-2,p)\Big]\nonumber\\
&\qquad\qquad +\frac{p}{N} \sum_{k}^{(i)}\mathbb{E} \Big[  \frac{1}{\|\mathbf{g}_{i}^u\|_2^2} \hat{\mathbf{e}}_k^*  G\hat{\mathbf{e}}_i \tau_1(\wt{ B }G) \frac{\partial \overline{\mathcal{P}_{ii}}}{\partial g_{ik}^u}\mathfrak{m}_i(p-1,p-1)\Big]\,.\label{071502}
\end{align}
 Notice the cancellation  between the two terms in the bracket in the first line.

Next we consider the estimate  of $\mathfrak{n}_i(p,p)$; especially we control the first term in the right side of  \eqref{111750}. 
In addition, using~(\ref{071441}),~(\ref{083050}), and the facts $\|\mathbf{g}_i^u\|_2=1+O_\prec(\frac{1}{\sqrt{N}})$ and  $c_i^u=1+O_\prec(\frac{1}{\sqrt{N}})$, we have 
\begin{align}
&\frac{1}{N}\frac{1}{\|\mathbf{g}_i^u\|_2}\sum_{k}^{(i)} \frac{  \partial (\hat{\mathbf{e}}_k^*  G\hat{\mathbf{e}}_i)}{\partial g^{u}_{ik}}= - \tau_1(G)\big(\tilde{\sigma}_iT_{\hat{i}i}+(\wt{ B }G)_{ii}\big)+ \tau_1(G\wt{ B })\big( G_{ii}+T_{ii}\big)+O_\prec(\Psi)\,. \label{111752}
\end{align}
 Note that the result of this calculation exactly cancels the second term in the right side of~\eqref{071421}.  
Hence, analogously to~(\ref{071501}), combining ~(\ref{111750}),~(\ref{083050}), (\ref{111751}), (\ref{071421}) and~(\ref{111752}), we~get
\begin{align}
\mathbb{E}[\mathfrak{n}_i(p,p)]= &\frac{1}{N}\sum_{k}^{(i)} \mathbb{E} \Big[ \frac{ \partial \| \mathbf{g}^u_i\|_2^{-1}}{\partial g_{ik}^u} \hat{\mathbf{e}}_k^*  G\hat{\mathbf{e}}_i  \mathfrak{n}_i(p-1,p)\Big]\nonumber\\
&+\frac{p-1}{N} \sum_{k}^{(i)} \mathbb{E} \Big[ \frac{1}{\| \mathbf{g}^u_i\|_2} \hat{\mathbf{e}}_k^* G\hat{\mathbf{e}}_i  \frac{\partial \mathcal{K}_{ii}}{\partial g_{ik}^u}\mathfrak{n}_i(p-2,p)\Big]\nonumber\\
&+  \frac{p}{N} \sum_{k}^{(i)} \mathbb{E} \Big[\frac{1}{\| \mathbf{g}^u_i\|_2} \hat{\mathbf{e}}_k^*G\hat{\mathbf{e}}_i  \frac{\partial \overline{\mathcal{K}_{ii}}}{\partial g_{ik}^u}\mathfrak{n}_i(p-1,p-1)\Big]\nonumber\\ &+ \mathbb{E} \Big[ O_\prec(\Psi) \mathfrak{n}_i(p-1,p)\Big]. \label{112402}
\end{align}

Hence, to prove the second equation of~(\ref{071503}), it suffices to estimate the first three terms on the  right side of~(\ref{112402}). For the first equation of~(\ref{071503}), with~(\ref{083011}) and~(\ref{1023100}),  it suffices  to estimate the second to the fifth terms on the right side of~(\ref{071501}), and the terms on the right side of~(\ref{071502}). 
All these estimates can be derived from the following lemma.

\begin{lem} \label{lem.071510}  Suppose that the assumptions in Theorem~\ref{071601} hold. 
Set $X_i=\hat{I}$ or $\wt{ B }^{\la i\ra}$. Let $Q$ be any deterministic diagonal matrix satisfying $\|Q\|\leq C$ and $X=\hat{I}$ or~$A$. We have the following estimates 
\begin{align}
\frac{1}{N}\sum_{k}^{(i)}\frac{ \partial \| \mathbf{g}^u_i\|_2^{-1}}{\partial g_{ik}^u} \hat{\mathbf{e}}_k^*X_i G\hat{\mathbf{e}}_i&=
 O_\prec(\frac{1}{N}),\quad  &\frac{1}{N} \sum_{k}^{(i)} \hat{\mathbf{e}}_i^*X\frac{\partial G}{\partial g_{ik}^u} \hat{\mathbf{e}}_i \hat{\mathbf{e}}_k^*  X_i G\hat{\mathbf{e}}_i & =O_\prec(\Psi^2), \nonumber\\
 \frac{1}{N}\sum_{k}^{(i)}\frac{ \partial T_{ji}}{\partial g_{ik}^u} \hat{\mathbf{e}}_k^*X_i G\hat{\mathbf{e}}_i&=
 O_\prec(\Psi^2), \quad & 
\frac{1}{N} \sum_{k}^{(i)}  \frac{\partial \ntr QXG}{\partial g_{ik}^u} \hat{\mathbf{e}}_k^*  X_i G\hat{\mathbf{e}}_i&=O_\prec(\Psi^4), \label{071520}
\end{align}
where $j=i$ or $\hat{i}$ in the third equation. 
\end{lem}

Assuming the validity of Lemma~\ref{lem.071510}, we continue with the proof of Lemma \ref{lem.071420}. Recall that our task is to bound the terms on the right 
sides  of (\ref{071501}), (\ref{071502}), (\ref{112402}). The second term in~(\ref{071501}), the first term in~(\ref{071502}) and the first term in 
~(\ref{112402}) can all be estimated with the aid of first bound in~(\ref{071520}). 
The estimates for the  third term in~(\ref{071501}) and  the second term  in~(\ref{071502}) follow from the last bound in~(\ref{071520}). 
Finally, the fourth term in~(\ref{071501}), the third term  in~(\ref{071502}) and the second term in~(\ref{112402}) 
 together with  their complex conjugate analogues  can be  estimated in a similar way, so 
we only present the details for the fourth term on the right side of~(\ref{071501}) in the sequel.

Recall the definition of $\mathcal{P}_{ii}$ from~(\ref{071560})
\begin{align*}
\mathcal{P}_{ii}= (\wt{ B }G)_{ii} \tau_1(G) -G_{ii} \tau_1(\wt{ B }G)+\big( G_{ii}+ T_{ii}\big) \Upsilon_{1}\,.
\end{align*}
Using ~(\ref{102702}), and recalling the definition of $\Upsilon_{1}$ in~(\ref{072130}), we can see that $\mathcal{P}_{ii}$ is a combination of the terms of the following forms: $T_{ii}$, $(XG)_{ii}$ and $\ntr (QX G)$, for $X=\hat{I}$ or $A$, and  $Q$ is certain deterministic diagonal  matrix with $\|Q\|\leq C$ for some positive constant $C$. For example: $(\wt{B}G)_{ii}= 1+zG_{ii}-(AG)_{ii}$, and
\begin{align*}
\tau_1(G\wt{B})&= \tau_1(\hat{I}-G(A-z))=1+z\tau_1(G)-\tau_1(GA)\nonumber\\
&= 1+2z\ntr (\hat{I}_1G)-2\ntr (A\hat{I}_1G)=1+2z\ntr (\hat{I}_1G)-2\ntr (\hat{I}_2 AG)\,.
\end{align*}
Then, by the product rule for derivative, and the boundedness of all the partial traces  (\cf~(\ref{102125})) and entries  (\cf~(\ref{112401}),~(\ref{100401})), we can apply  the last three bounds in~(\ref{071520}) to conclude that the fourth term on the right side of~(\ref{071501}) is $\mathbb{E}[O_\prec(\Psi^2)\mathfrak{m}_i(p-2,p)]$.  

This completes the proof of Lemma~\ref{lem.071420}, up to Lemma~\ref{lem.071510}.
\end{proof}

\begin{proof}[Proof of Lemma~\ref{lem.071510}]  Since the sums in~\eqref{071520} are over $k\neq i$, it will be convenient to work in this proof with the following notations 
\begin{align}
I^{\la i\ra}:=I-\mathbf{e}_i\mathbf{e}_i^*\,,\qquad \hat{I}_{1}^{\la i\ra}:= I^{\la i\ra}\oplus 0\,, \label{081201}
\end{align} 
where $0$ is the $N\times N$ zero matrix.
 We check the estimates in~(\ref{071520}) one by one. For the first estimate, we have
\begin{align*}
\frac{1}{N}\sum_{k}^{(i)}\frac{ \partial \| \mathbf{g}^u_i\|_2^{-1}}{\partial g_{ik}^u} \hat{\mathbf{e}}_k^*X_i G\hat{\mathbf{e}}_i&=-\frac{1}{2N}  \frac{1}{\|\mathbf{g}_i^u\|_2^3} \sum_{k}^{(i)}  \bar{g}_{ik}^u  \hat{\mathbf{e}}_k^*X_i G\hat{\mathbf{e}}_i\nonumber\\
&= -\frac{1}{2N} 
 \frac{1}{\|\mathbf{g}_i^u\|_2^2}   (\mathring{\mathbf{k}}_i^{u})^*X_i G\hat{\mathbf{e}}_i=O_\prec(\frac{1}{N})\,,
\end{align*}
where in the last step we used  that 
\begin{align}
(\mathring{\mathbf{k}}_i^{u})^*X_i G\hat{\mathbf{e}}_i\prec 1\,, \label{XG}
\end{align}
which  would follow once we show 
$|\mathring{S}_{ii}|\prec 1$ and $|\mathring{T}_{ii}| = |T_{ii} - h_{ii}^u G_{ii}|\prec 1$ by 
~(\ref{052901}). Since $\mathring{S}_{ii} = -(\widetilde BG)_{ii} + O_\prec(\Psi)$  by
 ~(\ref{050930}),~(\ref{083011}) and  $|(\widetilde BG)_{ii}|\prec 1$ from~(\ref{100401}), we get  $|\mathring{S}_{ii}|\prec 1$.
  The estimate $|\mathring{T}_{ii}|\prec 1$ follows from~(\ref{112401}) and the fact $|h_{ii}^u|\prec\frac{1}{\sqrt{N}}$.

Next, we show  the second estimate in~(\ref{071520}). Using~(\ref{050915}),  we have
\begin{align}
&\frac{1}{N} \sum_{k}^{(i)}   \hat{\mathbf{e}}_i^* X\frac{\partial G}{\partial g_{ik}^u} \hat{\mathbf{e}}_i  \hat{\mathbf{e}}_k^*  X_i G\hat{\mathbf{e}}_i = c_i^u\frac{1}{N} \sum_{k}^{(i)}  \hat{\mathbf{e}}_i^*X G\hat{\mathbf{e}}_k\big(\hat{\mathbf{e}}_i+\mathbf{k}_i^u\big)^* \wt{ B }^{\la i\ra} \mathcal{R}_i G \hat{\mathbf{e}}_i  \hat{\mathbf{e}}_k^*  X_i G\hat{\mathbf{e}}_i\nonumber\\
&\qquad\qquad+c_i^u\frac{1}{N} \sum_{k}^{(i)}  \hat{\mathbf{e}}_i^* XG\mathcal{R}_i \wt{ B }^{\la i\ra} \hat{\mathbf{e}}_k\big(\hat{\mathbf{e}}_i^*+(\mathbf{k}_i^u)^*\big) G\hat{\mathbf{e}}_i \hat{\mathbf{e}}_k^*  X_i G\hat{\mathbf{e}}_i\nonumber\\ &\qquad\qquad+\frac{1}{N} \sum_{k}^{(i)}  \hat{\mathbf{e}}_i^*X\Delta_{G}^u(i,k)\hat{\mathbf{e}}_i \hat{\mathbf{e}}_k^*  X_i G\hat{\mathbf{e}}_i\nonumber\\
 &= c_i^u\frac{1}{N} \hat{\mathbf{e}}_i^* XG \hat{I}^{\la i\ra}_1 X_i G\hat{\mathbf{e}}_i \big(\hat{\mathbf{e}}_i+\mathbf{k}_i^u\big)^* \wt{ B }^{\la i\ra} \mathcal{R}_i G \hat{\mathbf{e}}_i\nonumber\\ &\qquad\qquad +c_i^u\frac{1}{N} \hat{\mathbf{e}}_i^* XG\mathcal{R}_i \wt{ B }^{\la i\ra}  
 \hat{I}^{\la i\ra}_1 X_i G\hat{\mathbf{e}}_i \big(\hat{\mathbf{e}}_i+\mathbf{k}_i^u\big)^* G\hat{\mathbf{e}}_i \nonumber\\
&\qquad\qquad+\frac{1}{N} \sum_{k}^{(i)}  \hat{\mathbf{e}}_i^*X\Delta_{G}^u(i,k)\hat{\mathbf{e}}_i \hat{\mathbf{e}}_k^*  X_i G\hat{\mathbf{e}}_i\,, \label{100501}
\end{align}
where we have used the notation introduced in~(\ref{081201}).

From  Lemma~\ref{lem.100655} in Appendix~\ref{Appendix B}, we see that the last term on the right side of~(\ref{100501}) is of order $O_\prec(\Psi^2)$.  For the first two terms,  we first claim that
\begin{align}
|\hat{\mathbf{e}}_i^* XG \hat{I}_1^{\la i\ra}X_i G\hat{\mathbf{e}}_i |\prec \frac{1}{\eta}\,,\qquad |\hat{\mathbf{e}}_i^* XG\mathcal{R}_i \wt{ B }^{\la i\ra} \hat{I}_1^{\la i\ra}X_i G\hat{\mathbf{e}}_i |\prec \frac{1}{\eta}\,.  \label{1023101}
\end{align} 
We prove the first estimate~(\ref{1023101}) as follows. Note that
\begin{align}
\hat{\mathbf{e}}_i^* XG \hat{I}_i^{\la i\ra}X_i G\hat{\mathbf{e}}_i &\le \hat{\mathbf{e}}_i^* X|G|^2X \hat{\mathbf{e}}_i+ \hat{\mathbf{e}}_i^* G^*X_i^* \hat{I}^{\la i\ra}_1X_iG\hat{\mathbf{e}}_i \nonumber\\ &\leq  \frac{1}{\eta} \Im (X G X)_{ii} + \| X_i\|^2  \frac{1}{\eta} \Im G_{ii}\,. \label{112410}
\end{align}
Recall  $X=\hat{I}$ or $A$, and the fact $(AGA)_{ii}=|\sigma_i|^2G_{\hat{i}\hat{i}}$.  This together with~(\ref{112401}) and the fact $\|X_i\|\leq C$ since $X_i=\hat{I}$ or $\wt{B}^{\la i\ra}$ implies the first estimate in~(\ref{1023101}). The second estimate can be derived in a similar way.
  
Then, we recall  from ~(\ref{071553}) that 
$\big(\hat{\mathbf{e}}_i+\mathbf{k}_i^u\big)^* \wt{ B }^{\la i\ra} \mathcal{R}_i G \hat{\mathbf{e}}_i = 
-\widetilde \sigma_i T_{\hat ii} - (\widetilde B G)_{ii}$, and from the definition of $T_{ij}$ in~(\ref{053070}) that
$\big(\hat{\mathbf{e}}_i+\mathbf{k}_i^u\big)^*  G \hat{\mathbf{e}}_i = G_{ii}
+T_{ ii} $, 
 which together with~(\ref{112401}),~(\ref{100401}) and ~(\ref{1023101})  imply   that the first two terms on the right side of~(\ref{100501}) are also of order $O_\prec(\Psi^2)$. This completes the second estimate in~(\ref{071520}).

For the third estimate in~(\ref{071520}), we present the details for $j=i$ in the sequel. The  case of $j=\hat{i}$ is 
similar but simpler and we omit it.  According to the definition of $T_{ii}$ in~(\ref{053070}), it suffices to show 
\begin{align} \frac{1}{N} \sum_{k}^{(i)}  \frac{\partial (\mathbf{k}_i^u)^*}{\partial g_{ik}^u}G\hat{\mathbf{e}}_i \hat{\mathbf{e}}_k^*  X_i G\hat{\mathbf{e}}_i= O_\prec(\frac{1}{N} ),\quad  \frac{1}{N} \sum_{k}^{(i)}  (\mathbf{k}_i^u)^* \frac{\partial G}{\partial g_{ik}^u} \hat{\mathbf{e}}_i \hat{\mathbf{e}}_k^*  X_i G\hat{\mathbf{e}}_i=O_\prec(\Psi^2). \label{071580}
\end{align}
For the first estimate in~(\ref{071580}), we have 
\begin{align*}
\frac{1}{N} \sum_{k}^{(i)}  \frac{\partial (\mathbf{k}_i^u)^*}{\partial g_{ik}^u}G\hat{\mathbf{e}}_i \hat{\mathbf{e}}_k^*  X_i G\hat{\mathbf{e}}_i&=-\frac{1}{2\|\mathbf{g}_i^u\|_2^2}\frac{1}{N} \sum_{k}^{(i)} \bar{h}_{ik}^u\hat{\mathbf{e}}_k^*  X_i G\hat{\mathbf{e}}_i (\mathbf{k}_i^u)^*G\hat{\mathbf{e}}_i \nonumber\\
&= -\frac{1}{2\|\mathbf{g}_i^u\|_2^2}\frac{1}{N}(\mathring{\mathbf{k}}_i^u)^* X_i G\hat{\mathbf{e}}_i (\mathbf{k}_i^u)^*G\hat{\mathbf{e}}_i =O_\prec(\frac{1}{N})\,,
\end{align*}
where in the last step we used ~(\ref{112401}) and~\eqref{XG}. 
The proof of the second estimate in~(\ref{071580}) is similar to that for the second estimate in~(\ref{071520}). It suffices to go through the discussion from~(\ref{100501}) to~(\ref{112410}) again, with the vector $\hat{\mathbf{e}}_i^* X$ replaced by $(\mathbf{k}_i^u)^*$. The main differences are: instead of the last term of~(\ref{100501}), we have
\begin{align}
\frac{1}{N} \sum_{k}^{(i)}  (\mathbf{k}_i^u)^*\Delta_{G}^u(i,k)\hat{\mathbf{e}}_i \hat{\mathbf{e}}_k^*  X_i G\hat{\mathbf{e}}_i\,, \label{112450}
\end{align}
and instead of the first term on the right side of~(\ref{112410}), we have
\begin{align}
\frac{1}{\eta} \Im (\mathbf{k}_i^u)^*G \mathbf{k}_i^u\,.  \label{112451}
\end{align}
The bound on~(\ref{112450}) is stated in~(\ref{120503}). For~(\ref{112451}), we recall the identity ~(\ref{050805}) which implies $\mathbf{k}_i^u=-\mathcal{R}_i\hat{\mathbf{e}}_i$, the fact $G=\mathcal{U}\mathcal{G}\mathcal{U}^*$, together with ~(\ref{101811}) and the fact $\mathcal{R}_i^2=\hat{I}$. Then we have 
\begin{align}
(\mathbf{k}_i^u)^*G\mathbf{k}_i^u=
\hat{\mathbf{e}}_i^*\mathcal{R}_i\mathcal{U}\mathcal{G}\mathcal{U}^*\mathcal{R}_i\hat{\mathbf{e}}_i= \hat{\mathbf{e}}_i^*\mathcal{U}_i\Phi_i\mathcal{G}\Phi_i^*\mathcal{U}_i^*\hat{\mathbf{e}}_i=\mathcal{G}_{ii}\,.  \label{112230}
\end{align} 
Similarly to~(\ref{112401}), with the second bound in  assumption~(\ref{071422}), we can also show that
\begin{align}
\max_{k,l}|\mathcal{G}_{kl}|\prec 1\,. \label{112460}
\end{align} 
With these bounds for~(\ref{112450}) and ~(\ref{112451}), we can show the second estimate of~(\ref{071580}), which together with the first estimate in~(\ref{071580}) implies the third bound in ~(\ref{071520}).

At the end, we show the last bound in~(\ref{071520}). Applying~(\ref{050915}),  we have 
\begin{align}
\frac{\partial \ntr QXG}{\partial g_{ik}^u}= &\frac{1}{N} c_i^u \big(\hat{\mathbf{e}}_i+\mathbf{k}_i^u\big)^* \wt{ B }^{\la i\ra} \mathcal{R}_i GQXG\hat{\mathbf{e}}_k\nonumber\\
&+\frac{1}{N}c_i^u \big(\hat{\mathbf{e}}_i+\mathbf{k}_i^u\big)^* GQXG\mathcal{R}_i \wt{ B }^{\la i\ra} \hat{\mathbf{e}}_k+\ntr QX\Delta_{G}^u(i,k)\,. \label{101820}
\end{align}
Summing over $k$ and using the notation in~(\ref{081201}), we can write
\begin{align}
&\frac{1}{N} \sum_{k}^{(i)}  \frac{\partial \ntr QXG}{\partial g_{ik}^u} \hat{\mathbf{e}}_k^*  X_i G\hat{\mathbf{e}}_i= \frac{c_i^u }{N^2} \big(\hat{\mathbf{e}}_i+\mathbf{k}_i^u\big)^* \wt{ B }^{\la i\ra} \mathcal{R}_i GQXG\hat{I}_1^{\la i\ra}  X_i G\hat{\mathbf{e}}_i \nonumber\\
&\qquad  +\frac{c_i^u }{N^2} \big(\hat{\mathbf{e}}_i+\mathbf{k}_i^u\big)^* GQXG\mathcal{R}_i \wt{ B }^{\la i\ra} \hat{I}_1^{\la i\ra}  X_i G\hat{\mathbf{e}}_i+ \frac{1}{N} \sum_{k}^{(i)} \ntr QX\Delta_{G}^u(i,k) \hat{\mathbf{e}}_k^*  X_i G\hat{\mathbf{e}}_i\,. \label{112455}
\end{align}
The bound for the last term of the right side of~(\ref{112455}) can be found in~(\ref{100657}).  

In the sequel, we bound the first two terms on the right side of ~(\ref{112455}). We only present the details for the first one; the second is estimated analogously. First, similarly to~(\ref{071553}), we~have 
\begin{align*}
 \big(\hat{\mathbf{e}}_i+\mathbf{k}_i^u\big)^* \wt{ B }^{\la i\ra} \mathcal{R}_i=-(\tilde{\sigma}_i (\mathbf{k}_i^v)^*+\hat{\mathbf{e}}_i^* \wt{B})\,.
\end{align*}
Then we can write
\begin{multline}
 \frac{c_i^u }{N^2} \big(\hat{\mathbf{e}}_i+\mathbf{k}_i^u\big)^* \wt{ B }^{\la i\ra} \mathcal{R}_i GQXG\hat{I}_1^{\la i\ra}  X_i G\hat{\mathbf{e}}_i =  - \frac{c_i^u }{N^2} (\tilde{\sigma}_i (\mathbf{k}_i^v)^*+\hat{\mathbf{e}}_i^* \wt{B})GQXG\hat{I}_1 X_i G\hat{\mathbf{e}}_i\\+  \frac{c_i^u }{N^2} (\tilde{\sigma}_i (\mathbf{k}_i^v)^*+\hat{\mathbf{e}}_i^* \wt{B})GQXG\hat{\mathbf{e}}_i\hat{\mathbf{e}}_i^* X_i G\hat{\mathbf{e}}_i\,. \label{112201}
\end{multline}
For the second term on the right side of~(\ref{112201}), we use the bounds
\begin{align}
\big|(\tilde{\sigma}_i (\mathbf{k}_i^v)^*+\hat{\mathbf{e}}_i^* \wt{B})GQXG\hat{\mathbf{e}}_i\big|\prec\eta^{-2}\,,\qquad \big|\hat{\mathbf{e}}_i^* X_i G\hat{\mathbf{e}}_i\big|\prec 1\,, \label{112202}
\end{align}
where  in the first inequality we used the trivial bound $\| G\| \le \eta^{-1}$, while in the second inequality
we used the fact that $X_i=\hat{I}$ or $\wt{B}^{\la i\ra}$, together with~(\ref{0830200}), and the first bound in~(\ref{112401}). Using the bounds in~(\ref{112202}), we see that the second term on the right side of~(\ref{112201}) is of order $O_\prec(\Psi^4)$.

Now, we turn to the first term on the right side of~(\ref{112201}).  Note that 
\begin{align}
\Big|\frac{1}{N^2}  (\tilde{\sigma}_i (\mathbf{k}_i^v)^*+\hat{\mathbf{e}}_i^* \wt{B})GQXG\hat{I}_1 &X_i G\hat{\mathbf{e}}_i\Big| \leq \frac{C}{N^2\eta}  \big(\|(\mathbf{k}_i^v)^*G\|_2+\|\hat{\mathbf{e}}_i^* \wt{B}G\|_2\big) \|G\hat{\mathbf{e}}_i\|_2\nonumber\\
&\hspace{-10ex}\leq  \frac{C}{N^2\eta}  \big(\|(\mathbf{k}_i^v)^*G\|_2^2+\|\hat{\mathbf{e}}_i^* \wt{B}G\|_2^2+ \|G\hat{\mathbf{e}}_i\|_2^2\big)\nonumber\\
& \hspace{-10ex}\leq  \frac{C}{N^2\eta^2}  \big(\Im (\mathbf{k}_i^v)^*G \mathbf{k}_i^v+\Im \hat{\mathbf{e}}_i^* \wt{B}G \wt{B}\hat{\mathbf{e}}_i+ \Im \hat{\mathbf{e}}_i^*G\hat{\mathbf{e}}_i\big)\,. \label{112250}
\end{align}
Similarly to \eqref{112230}, we have
\begin{align}
(\mathbf{k}_i^v)^*G\mathbf{k}_i^v=\hat{\mathbf{e}}_{\hat{i}}^*\mathcal{R}_i\mathcal{U}\mathcal{G}\mathcal{U}^*\mathcal{R}_i\hat{\mathbf{e}}_{\hat{i}}= \hat{\mathbf{e}}_{\hat{i}}^*\mathcal{U}_i\Phi_i\mathcal{G}\Phi_i^*\mathcal{U}_i^*\hat{\mathbf{e}}_{\hat{i}}=\mathcal{G}_{\hat{i}\hat{i}}\,. \label{112251}
\end{align}
Combining~(\ref{112250}) and~(\ref{112251}), we obtain 
\begin{multline*}
\Big|\frac{1}{N^2}  (\tilde{\sigma}_i (\mathbf{k}_i^v)^*+\hat{\mathbf{e}}_i^* \wt{B})GQXG\hat{I}_1 X_i G\hat{\mathbf{e}}_i\Big|\leq  \frac{C}{N^2\eta^2}  \big(\Im \mathcal{G}_{\hat{i}\hat{i}}+\Im  (\wt{B}G \wt{B})_{ii}+ \Im G_{ii}\big)\\=O_\prec(\Psi^4)\,, 
\end{multline*}
where we also used~(\ref{100401}) and~(\ref{112460}). Hence the first term on the right side of~(\ref{112455}) is $O_\prec(\Psi^4)$. The second term on the right side of~(\ref{112455}) is bounded  similarly. These bounds together with~(\ref{100657}) yield the other estimates in~(\ref{071520}). This completes the proof of~Lemma~\ref{lem.071510}. 
\end{proof}

\subsection{Local stability analysis: proof of Theorem~\ref{071601}}\label{sec: mini stability analysis}
Having established Lemma~\ref{lem.071420}, we move on to the local stability analysis in order to conclude the proof of Theorem~\ref{071601}.
 
 \begin{proof}[Proof of Theorem~\ref{071601}]  Applying Young's inequality, we obtain from~(\ref{071503}) that for any given (small) $\varepsilon>0$, 
 \begin{align*}
 \mathbb{E}[\mathfrak{m}_i(p,p)]\leq 3\frac{1}{2p} \mathbb{E}\big[N^{2p\varepsilon}\Psi^{2p}\big]+ 3\frac{2p-1}{2p} N^{-\frac{2p\varepsilon}{2p-1}} \mathbb{E}[\mathfrak{m}_i(p,p)]\,, 
 \end{align*}
 which implies $ \mathbb{E}[\mathfrak{m}_i(p,p)]\prec \Psi^{2p}$. Hence, we conclude the proof of the first estimate of~(\ref{071305}). 
 
 The second estimate of~\eqref{071305} can be proved in the same way, with the aid of the second equation in~(\ref{071503}). Then, applying Markov's inequality  we get the first  and the third estimates of ~(\ref{071302}) with $j=i$.  The others in~(\ref{071302}) are proved in an analogous way. We omit the details.

 Next, we show that~(\ref{071302}) together with the assumption~(\ref{071422})  imply~(\ref{071604}). To this end, we first show the following crude bound
 \begin{align}
 \Lambda_T(z)\prec N^{-\frac{\gamma}{4}} \label{101401}
 \end{align}
 under the assumption ~(\ref{071422}).  We need the following  equations for $j=i,\hat{i}$,
 \begin{align}
&T_{ij}= - \tau_1(G)\big(\tilde{\sigma}_iT_{\hat{i}j}+(\wt{ B }G)_{ij}\big)
+ \tau_1(G\wt{ B })\big( G_{ij}+T_{ij}\big)+O_\prec(\Psi)\,,\nonumber\\
&T_{\hat{i}j}= - \tau_2(G)\big(\tilde{\sigma}_i^*T_{ij}+(\wt{ B }G)_{\hat{i}j}\big)
+ \tau_2(G\wt{ B })\big( G_{\hat{i}j}+T_{\hat{i}j}\big)+O_\prec(\Psi)\,, \label{110801}
 \end{align}
which is just a rewriting of the second line of~(\ref{071302}), according to the definition in~(\ref{071560}). 

Using the first identity in~(\ref{102702}) and the definition of $A$ in~(\ref{050970}), we have
\begin{align}
&(\wt{ B }G)_{ii}=1+zG_{ii}-\xi_i G_{\hat{i}i},\qquad (\wt{B}G)_{i\hat{i}}=-\xi_iG_{\hat{i}\hat{i}}+z G_{i\hat{i}}, \nonumber\\
&(\wt{ B }G)_{\hat{i}i}=-\bar{\xi}_i G_{ii}+zG_{\hat{i}i},\qquad (\wt{ B }G)_{\hat{i}\hat{i}}=1+zG_{\hat{i}\hat{i}}-\bar{\xi}_i G_{i\hat{i}}. \label{0913101}
\end{align}
Applying the assumption on $\Lambda_{\dd}$ in~(\ref{071422}),  and also the lower bound of $\Im \omega_B$ and the upper bound on $|\omega_B|$ in ~(\ref{110270}), we can get  from~(\ref{0913101}) that
 \begin{align}
 &(\wt{ B }G)_{ii}=\frac{(z-\omega_B)\omega_B}{|\xi_i|^2-\omega_B^2}+O_\prec(N^{-\frac{\gamma}{4}})\,,\qquad (\wt{B}G)_{i\hat{i}}=\frac{(z-\omega_B)\xi_i}{|\xi_i|^2-\omega_B^2}+O_\prec(N^{-\frac{\gamma}{4}})\,, \nonumber\\
&(\wt{ B }G)_{\hat{i}i}=\frac{(z-\omega_B)\bar{\xi}_i}{|\xi_i|^2-\omega_B^2}+O_\prec(N^{-\frac{\gamma}{4}})\,,\qquad (\wt{ B }G)_{\hat{i}\hat{i}}=\frac{(z-\omega_B)\omega_B}{|\xi_i|^2-\omega_B^2}+O_\prec(N^{-\frac{\gamma}{4}})\,.  \label{111910}
 \end{align}
This together with  ~(\ref{102125}),  
leads to the following estimates for $j=i,\hat{i}$,
\begin{align*}
&- \tau_1(G)(\wt{ B }G)_{ij}
+ \tau_1(G\wt{ B })G_{ij}=O_\prec(N^{-\frac{\gamma}{4}})\,,\\ & - \tau_2(G)(\wt{ B }G)_{\hat{i}j}
+ \tau_2(G\wt{ B }) G_{\hat{i}j}= O_\prec(N^{-\frac{\gamma}{4}})\,,
\end{align*}
which together with~(\ref{110801}) implies
\begin{align}
&\big(1- \tau_1(G\wt{ B })\big)T_{ij}+ \tau_1(G)\tilde{\sigma}_iT_{\hat{i}j}
= O_\prec(N^{-\frac{\gamma}{4}})\,,\nonumber\\
&\big(1- \tau_2(G\wt{ B })\big)T_{\hat{i}j}+ \tau_2(G)\tilde{\sigma}_i^*T_{ij}=
O_\prec(N^{-\frac{\gamma}{4}})\,,\qquad\quad j=i,\hat{i}\,. 
\label{071716}
\end{align}
Solving $T_{ij}$ from the  equations in~(\ref{071716}), we get
\begin{align}
\big(\big(1- \tau_1(G\wt{ B })\big)\big(1- \tau_2(G\wt{ B })\big)-|\sigma_i|^2 \tau_1(G) \tau_2(G)\big)T_{ij}= O_\prec(N^{-\frac{\gamma}{4}})\,.  \label{071718}
\end{align}
Using the  assumption on $\Lambda_T$ in~(\ref{071422}),  and also~(\ref{102125}), we obtain from~(\ref{071718}) that
\begin{align}
\big(\big(1+(\omega_B-z) m_{\mu_A\boxplus\mu_B}\big)^2-|\sigma_i|^2 m_{\mu_A\boxplus\mu_B}^2\big) T_{ij}=O_\prec(N^{-\frac{\gamma}{4}})\,. \label{083101}
\end{align}

Further, observe that
\begin{align}
\big(1+(\omega_B-z) m_{\mu_A\boxplus\mu_B}\big)^2-|\sigma_i|^2 m_{\mu_A\boxplus\mu_B}^2= m_{\mu_A\boxplus \mu_B}^2\big(\omega_A-|\sigma_i|\big)\big(\omega_A+|\sigma_i|\big)\,, \label{101415}
\end{align}
which follows from the second equation in~(\ref{le definiting equations}) with $(\mu_1,\mu_2)=(\mu_A, \mu_B)$.
Then by~(\ref{110270}) and the fact $m_{\mu_A\boxplus\mu_B}=m_{\mu_A}(\omega_B)$,  we see that 
 $|T_{ij}|\prec N^{-\frac{\gamma}{4}}$ for $j=i,\hat{i}$. Analogously, one can show $|T_{\hat{i}j}|\prec N^{-\frac{\gamma}{4}}$. This completes the proof of the crude bound~(\ref{101401}).

With~(\ref{101401}), we can now proceed to the proof of~(\ref{071604}).  We
 consider the average of $\mathcal{P}_{ii}$ over $i\in \llbracket 1, N\rrbracket $,
  and use~(\ref{071302}) to obtain
 \begin{align}
\Upsilon_{1}\cdot \frac{1}{N}\sum_{i=1}^N\big( G_{ii}+ T_{ii}\big)
=    \frac{1}{N}  \sum_{i=1}^N \mathcal{P}_{ii} = O_\prec(\Psi)\,. \label{071607}
 \end{align}
By  the first estimate in~(\ref{102125}), the fact $m_{\mu_A\boxplus\mu_B}=m_{\mu_A}(\omega_B)$, the lower bound on $\Im \omega_B$ in~(\ref{110270}),  and also the crude bound~(\ref{101401}), we can see that 
\begin{align}
\bigg|\frac{1}{\frac{1}{N}\sum_{i=1}^N\big( G_{ii}+ T_{ii}\big)}\bigg|=\bigg|\frac{1}{ m_{\mu_A}(\omega_B)+O_\prec(N^{-\frac{\gamma}{4}})}\bigg|\prec 1\,. \label{071608}
\end{align} 
Then the first estimate in~(\ref{071604}) follows from~(\ref{071607}) and~(\ref{071608}) immediately. The second one can be verified similarly. 

Finally, using~(\ref{071302}) and ~(\ref{071604}), we can prove~(\ref{071602}) as follows.  
Recall the definition in~(\ref{071560}). Applying~(\ref{071422})-(\ref{071604}), we obtain, for $j=i,\hat{i}$,
\begin{align}
(\wt{ B }G)_{ij}  =G_{ij} \frac{\tau_1(\wt{ B }G)}{\tau_1(G)}+O_\prec(\Psi),\qquad (\wt{ B }G)_{\hat{i}j}  =G_{\hat{i}j} \frac{\tau_2(\wt{ B }G)}{\tau_2(G)}+O_\prec(\Psi)\,.   \label{1023110}
\end{align} 
Using~(\ref{0913101}) and ~(\ref{1023110}) we get the following system of equations,
\begin{align}
&1-\xi_i G_{\hat{i}i}+\omega_{B,1}^c G_{ii}=O_\prec(\Psi)\,,\qquad -\xi_iG_{\hat{i}\hat{i}}+\omega_{B,1}^cG_{i\hat{i}}= O_\prec(\Psi)\,,\nonumber\\
&-\bar{\xi}_i G_{ii}+\omega_{B,2}^cG_{\hat{i}i}= O_\prec(\Psi)\,,  \qquad  1-\bar{\xi}_i G_{i\hat{i}}+\omega_{B,2}^cG_{\hat{i}\hat{i}}=O_\prec(\Psi)\,, \label{112901}
\end{align}
where we used the notation introduced in~(\ref{102401}). Solving~(\ref{112901}) we find
\begin{align}
&G_{ii}=\frac{\omega_{B,2}^c}{|\xi_i|^2-\omega_{B,1}^c\omega_{B,2}^c}+O_\prec(\Psi)\,,\qquad G_{i\hat{i}}= \frac{\xi_i}{|\xi_i|^2-\omega_{B,1}^c\omega_{B,2}^c}+O_\prec(\Psi)\,,\nonumber\\
& G_{\hat{i}i}=\frac{\bar{\xi}_i}{|\xi_i|^2-\omega_{B,1}^c\omega_{B,2}^c}+O_\prec(\Psi)\,,\qquad G_{\hat{i}\hat{i}}= \frac{\omega_{B,1}^c}{|\xi_i|^2-\omega_{B,1}^c\omega_{B,2}^c}+O_\prec(\Psi)\,. \label{071711}
\end{align}
From~(\ref{102125}), we see that 
\begin{align}
\omega_{B,a}^c=\omega_B+O_\prec(N^{-\frac{\gamma}{4}})\,,\qquad \qquad a=1,2\,. \label{110281}
\end{align}
 The first estimate of~(\ref{071602}) could be verified  from \eqref{071711}, if we could show  
\begin{align}
\omega_{B,a}^c=\omega_B^c+O_\prec(\Psi)\,,\quad \qquad a=1,2\,.  \label{071710}
\end{align}
To this end, we use $\tau_1(G(z))=\tau_2(G(z))$; \cf~(\ref{neu083140}). From~(\ref{102401}) and~(\ref{neu083140}), we also have 
\begin{align}
\omega_{B,1}^c+\omega_{B,2}^c=2\omega_B^c\,. \label{113003}
\end{align}

Then,  averaging the first and the fourth equations of~(\ref{071711}) over $i\in \llbracket 1, N\rrbracket $, we get
\begin{align}
\omega_{B,2}^c\frac{1}{N}\sum_{i=1}^N \frac{1}{|\xi_i|^2-\omega_{B,1}^c\omega_{B,2}^c}=\omega_{B,1}^c\frac{1}{N}\sum_{i=1}^N \frac{1}{|\xi_i|^2-\omega_{B,1}^c\omega_{B,2}^c}+O_\prec(\Psi)\,, \label{113001}
\end{align}
where we also used~(\ref{neu083140}). We further claim that 
\begin{align}
   \bigg( \frac{1}{N}\sum_{i=1}^N \frac{1}{|\xi_i|^2-\omega_{B,1}^c\omega_{B,2}^c}\bigg)^{-1}\prec 1\,,  \label{113004}
\end{align}
which together with~(\ref{113001}) implies that 
\begin{align}
\omega_{B,2}^c=\omega_{B,1}^c+O_\prec(\Psi)\,.  \label{113002}
\end{align}
Combining~(\ref{113002}) with~(\ref{113003}), we get~(\ref{071710}). Hence, it suffices to show~(\ref{113004}).  To this end, we use 
(\ref{110281}). Then we have
\begin{align*}
&\frac{1}{N}\sum_{i=1}^N \frac{1}{|\xi_i|^2-\omega_{B,1}^c\omega_{B,2}^c}=\frac{1}{N}\sum_{i=1}^N \frac{1}{|\xi_i|^2-\omega_{B}^2+O_\prec(N^{-\frac{\gamma}{4}})}\nonumber\\
&= \frac{1}{N}\sum_{i=1}^N \frac{1}{|\xi_i|^2-\omega_{B}^2}+O_\prec(N^{-\frac{\gamma}{4}})=\omega_B^{-1} m_{\mu_A}(\omega_{B})+O_\prec(N^{-\frac{\gamma}{4}})\,,
\end{align*}
where in the first step above, we used the upper bound of $|\omega_B|$ in~(\ref{110270}); in the second step, we used again the fact that $|\xi_i|^2-\omega_{B}^2$ is away from $0$ due to the lower bound of $\Im \omega_B$ in~(\ref{110270}); and the last step follows from~(\ref{111601}).  Then the fact $\|A\|\leq C$ (\cf~(\ref{102701})), the lower bound of $\Im \omega_B$ and the upper bound on $|\omega_B|$ in~(\ref{110270}), we can get \eqref{113004}. Hence, we conclude the proof of the first estimate of~(\ref{071602}).

For the second estimate in~(\ref{071602}), we need to go through the proof of~(\ref{101401})  again, but this time  with the a priori input~(\ref{071422}) replaced by the first estimate of~(\ref{071602}). Therefore, with~(\ref{071602}), we can get 
\begin{align}
\Big(\big(1+(\omega_B^c-z) m_A(\omega_B^c)\big)^2-|\sigma_i|^2 (m_A(\omega_B^c))^2\Big) T_{ij}=O_\prec(\Psi)\,, \label{101410}
\end{align}
which is the analogue of~(\ref{083101}). Then, by the estimates in~(\ref{102125}) and the definition in~(\ref{091361}), it is not difficult to check that 
the coefficient of $T_{ij}$ above can be approximated by~(\ref{101415}), up to an error $O_\prec(N^{-\frac{\gamma}{4}})$. Hence, 
 we can improve the estimate to $|T_{ij}|\prec \Psi$ for $j=i,\hat{i}$. Similarly, we can prove the same bound for $T_{\hat{i}j}$.  This completes the second estimate of~(\ref{071602}).  Hence, we conclude the proof of Theorem~\ref{071601}.
 \end{proof}

 \subsection{Continuity argument: Proof of Theorem~\ref{thm.green function subordination}}\label{sec: continuity argument} Having derived Theorem~\ref{071601}, we prove Theorem~\ref{thm.green function subordination} using a continuity argument similar to~\cite{EYY}.

\begin{proof}[Proof of Theorem~\ref{thm.green function subordination}] 
 
First,  we show that $\Lambda_{\mathrm{d}}^c(z)$ in~(\ref{071602}) can be replaced by $\Lambda_{\mathrm{d}}(z)$.
 This means, we have to control the difference between $(\omega_A, \omega_B)$ and $(\omega_A^c, \omega_B^c)$
as described in \eqref{101470}; this
estimate will follow   from  the stability of the system $\Phi_{\mu_A,\mu_B} \big(\omega_A,\omega_B, z\big)=0$, (\cf~(\ref{le H system}) with $(\mu_1,\mu_2)=(\mu_A,\mu_B)$). 
We will use the dual pair of subordination equations, \ie when we analyze  $\mathcal{H}$ instead of $H$. 
Recall the notations introduced in~(\ref{0913125}), and also $\wt{\Lambda}_{\rm d}$ and $\wt{\Lambda}_T$ as the analogue of $\Lambda_{\rm d}$  and $\Lambda_T$, respectively,  see the explanation around \eqref{tilde quantities}. 
 For any $\delta\in [0,1]$ and $z\in \mathcal{S}_{\mathcal{I}}(\eta_{\mathrm{m}},\eta_{\mathrm{M}})$, we introduce the following event
 \begin{align}
 \Theta(z, \delta):=\big \{\Lambda_{\rm d}(z)\leq \delta\,, \quad \wt{\Lambda}_{\rm d}(z)\leq \delta\,, \quad \Lambda_T(z)\leq 1\,, \quad \wt{\Lambda}_T(z)\leq 1  \big\}\,.  \label{101465}
 \end{align} 
 
 With the above notation, we have the following lemma.

 \begin{lem} \label{pro.091379}  Suppose that the assumptions in Theorem~\ref{thm.081501} hold. 
 Let $\eta_{\mathrm{M}}>0$ be a sufficiently large constant and  $\gamma>0$ be a small constant
  in the definition \eqref{etamdef}.  For any $\varepsilon$ with $0< \varepsilon\leq \frac{\gamma}{8}$ 
 and for any $D>0$, there exists a positive integer $N_2(D,\varepsilon)$ such that the following holds: For any fixed  $z\in \mathcal{S}_{\mathcal{I}}(\eta_{\mathrm{m}}, \eta_{\mathrm{M}})$ there exists an event $\Omega(z)\equiv \Omega(z, D,\varepsilon)$ with 
 \begin{align}
 \mathbb{P}(\Omega(z))\geq 1-N^{-D},\qquad \forall N\geq N_2(D,\varepsilon)\,, \label{101460}
 \end{align}
 such that if the estimate 
 \begin{align}
 \mathbb{P}(\Theta (z, N^{-\frac{\gamma}{4}}))\geq 1-N^{-D}(1+N^5(\eta_{\mathrm{M}}-\eta)) \,,\qquad  \eta=\im z\,,  \label{101450}
 \end{align}
 holds for all $D>0$ and $N\geq N_1(D,\gamma,\varepsilon)$, for some threshold $N_1(D,\gamma,\varepsilon)$, then we also have 
 \begin{align}
 \Theta (z, N^{-\frac{\gamma}{4}})\cap \Omega(z)\subset \Theta\big(z, \frac{N^{\varepsilon}}{\sqrt{N\eta}}\big)\,, \label{112511}
 \end{align}
 for all $N\geq N_3(D,\gamma,\varepsilon):=\max\big\{N_1(D,\gamma,\varepsilon), N_2(D,\varepsilon)\big\}$.
 \end{lem}
 
 \begin{proof}[Proof of Lemma~\ref{pro.091379}]   In this proof, we fix $z\in \mathcal{S}_{\mathcal{I}}(\eta_{\mathrm{m}}, \eta_{\mathrm{M}})$.  According to the definition of $\prec$ in Definition~\ref{definition of stochastic domination}, we see from the assumption~(\ref{101450}) that 
\begin{align}
\Lambda_{{\rm d}}(z)\prec N^{-\frac{\gamma}{4}}\,, \qquad \wt{\Lambda}_{{\rm d}}(z)\prec N^{-\frac{\gamma}{4}}\,, \qquad  \Lambda_T(z)
 \prec  1\,, \qquad \wt{\Lambda}_T(z) \prec  1\,. \label{0913128}
\end{align}
We apply  Theorem~\ref{071601};  by the estimates on $\Lambda_{\rm d}^c$ and on $\Lambda_T$ in~(\ref{071602}) and  their analogues for~$\wt{\Lambda}_{\rm d}^c$ and~$\wt{\Lambda}_T$,  we have
\begin{align}
\Lambda_{\rm d}^c(z)\prec \Psi\,,\qquad \wt{\Lambda}_{\rm d}^c(z)\prec \Psi\,, \qquad \Lambda_T(z)\prec \Psi\,,\qquad \wt{\Lambda}_T(z)\prec \Psi\,.  \label{101451}
\end{align}

Now, we state the conclusions in~(\ref{101451}) in a more explicit quantitative form, with the quantitative assumption~(\ref{101450}). To this end, we need a more quantitative version of Lemma \ref{lem.071420}. Let  $\varphi:\R\to \R$ be a smooth cutoff   function s.t 
\begin{align}
\varphi(x)=1\text{  if  } |x|\leq K, \quad \varphi(x)=0 \text{  if  } |x|\geq 2K, \quad \sup_{x\in \mathbb{R}}|\varphi'(x)|\leq CK^{-1}  \label{19022301} 
\end{align}
for some sufficiently large constant $K>0$.  Let 
\begin{align}
\Gamma_i\equiv \Gamma_i(z):=&\sum_{a,b=i,\hat{i}} \big(|G_{ab}|^2+|\mathcal{G}_{ab}|^2+|T_{ab}|^2+|\mathcal{T}_{ab}|^2\big)\nonumber\\
&+\sum_{a=1,2} \big(|\tau_a (G)|^2+ |\tau_a(\wt{B}G)|^2+|\tau_a(G\wt{B})|^2+|\tau_a(\wt{B}G\wt{B})|^2\big). \label{19021801}
\end{align}
Note that for a given $i$, all the a priori bounds we needed in the proof of Lemma \ref{lem.071420} are the  $O_\prec(1)$ bound for $G_{ab}$, $\mathcal{G}_{ab}$, $T_{ab}$, $\mathcal{T}_{ab}$ with $a,b=i,\hat{i}$ and the tracial quantities in (\ref{19021801}). The $O_\prec(1)$ bound for $(XGY)_{ab}$ with $X,Y=\hat{I}$ or $\wt{B}$  were also used (see $(\wt{B}X\wt{B})_{ii}$ in (\ref{112250}) for instance), but  they can be derived from the bound of $G_{ab}$'s by using  (\ref{102702}).  Recall the definitions of $\mathfrak{m}_i$ and $\mathfrak{n}_i$ in~\eqref{102102}. We now introduce modifications of $\mathfrak{m}_i$ and $\mathfrak{n}_i$ by setting
\begin{align*}
\wt{\mathfrak{m}}_i(p,q)\deq \mathfrak{m}_i(p,q) (\varphi(\Gamma_i))^{p+q}\,, \qquad \wt{\mathfrak{n}}_i(p, q)\deq \mathfrak{n}_i(p,q)(\varphi(\Gamma_i))^{p+q}\,. 
\end{align*}
 In addition,  for any $\varepsilon'>0$,  let $\widehat{\Omega}(z) =\widehat{\Omega}(z,\varepsilon')$  be the event that all the concentration estimates of the components or quadratic forms of $\mathbf{h}_i^u$ and $\mathbf{h}_i^v$ in the proof of Lemma \ref{lem.071420} hold  with precision $N^{\varepsilon'}$. For instance, we used the large deviation bound~\eqref{091731} to bound  $(\mathbf{k}_i^u)^* \wt{ B }^{\la i\ra} \mathbf{k}_i^v$ in (\ref{050930}) by $O_\prec(N^{-\frac12})$, in the proof of Lemma \ref{lem.071420}. Now we can bound it more quantitatively by $\frac{N^{\varepsilon'}}{\sqrt{N}}$ on $\widehat{\Omega}(z)$. 
Now we claim that 
\begin{align}
\mathbb{E}[\wt{\mathfrak{m}}_i(p,p)]&=\mathbb{E}[\mathfrak{c}_1\wt{\mathfrak{m}}_i(p-1,p)]+\mathbb{E}[\mathfrak{c}_2 \wt{\mathfrak{m}}_i(p-2,p)]+\mathbb{E}[\mathfrak{c}_3 \wt{\mathfrak{m}}_i(p-1,p-1)]\, \label{19021701}
\end{align}
 with some random variables $\mathfrak{c}_1$, $\mathfrak{c}_2$, $\mathfrak{c}_3$, satisfying 
\begin{align}
|\mathfrak{c}_1|\leq C\frac{N^{\varepsilon'}}{\sqrt{N\eta}}, \qquad |\mathfrak{c}_2|\leq C\frac{N^{2{\varepsilon'}}}{N\eta}, \qquad |\mathfrak{c}_3|\leq C\frac{N^{2\varepsilon'}}{N\eta}, \qquad \text{ on   } \widehat{\Omega}(z),  \label{19021702}
\end{align}
for some positive constant $C$ which may depend on $K$ in (\ref{19022301}).
In addition,  the $\mathfrak{c}_i$'s also admit trivial deterministic  bounds of order $\eta^{-k}$, for some constant $k>0$. Moreover, for any $D'>0$, there exists $N(D', \varepsilon')$, such that if $N\geq N(D',\varepsilon')$
\begin{align*}
\mathbb{P}(\widehat{\Omega}(z))\geq 1-N^{-D'}. 
\end{align*}
Observe that (\ref{19021701})  is just a more explicit version of  (\ref{071503}),  considering that $\widehat{\Omega}(z)$ holds with high probability. The proof of the more quantitative estimate (\ref{19021701}) with (\ref{19021702}) is basically the same as the proof of the non-quantitative one in (\ref{071503}). 
 
  The price for introducing $\varphi(\Gamma_i)$ into $\wt{\mathfrak{m}}_i$ is that it creates  additional terms in the integration by parts. However, they are absorbed  into  the first term  in the right side of (\ref{19021701}).  For instance, in the analogue of the step (\ref{071435}), except for replacing $\mathfrak{m}_i$ by $\wt{\mathfrak{m}}_i$, we will have an additional term
\begin{align*}
\frac{1}{N} \sum_{k}^{(i)} \mathbb{E} \Big[ \frac{1}{\|\mathbf{g}_{i}^u\|_2} (\hat{\mathbf{e}}_k^*\wt{ B }^{\la i\ra}  G\hat{\mathbf{e}}_i)  \frac{  \partial \varphi(\Gamma_i) }{\partial g^{u}_{ik}} \tau_1(G) \wt{\mathfrak{m}}_i(p-1,p)\Big].
\end{align*}
For example,  one term of $\frac{  \partial \varphi(\Gamma_i) }{\partial g^{u}_{ik}}$ is 
\begin{align*}
 \varphi'(\Gamma_i) \frac{\partial |G_{ii}|^2}{\partial g^u_{ik}}=\varphi'(\Gamma_i) \frac{\partial G_{ii}}{\partial g^u_{ik}} \overline{G}_{ii}+\varphi'(\Gamma_i) \frac{\partial \overline{G_{ii}}}{\partial g^u_{ik}} G_{ii}.
\end{align*}
Using the second estimate in (\ref{071520}),   
\begin{align*}
\frac{1}{N} \sum_{k}^{(i)} \hat{\mathbf{e}}_k^*\wt{ B }^{\la i\ra}  G\hat{\mathbf{e}}_i\frac{\partial |G_{ii}|^2}{\partial g^u_{ik}} =O(\frac{N^{\varepsilon'}}{\sqrt{N\eta}}), \quad \text{
 on} \; \{\varphi'(\Gamma_i)\neq 0\}\cap \widehat{\Omega}(z).
\end{align*}

It is also easy to check that the other terms in $\frac{  \partial \varphi(\Gamma_i) }{\partial g^{u}_{ik}}$ give the same bound.  Therefore, we have (\ref{19021701}).

 Using Young's  inequality to  (\ref{19021701}), we can get 
\begin{align*}
\mathbb{E}[\wt{\mathfrak{m}}_i(p,p)]\leq &C_p N^{2p\varepsilon'}\Big(\mathbb{E}[|\mathfrak{c}_1|^{2p}+\mathbb{E}[|\mathfrak{c}_2|^p]+\mathbb{E}[|\mathfrak{c}_3|^p ]\Big) \nonumber\\
&\leq C_p N^{2p\varepsilon'} \Big( (\frac{N^{\varepsilon'}}{\sqrt{N\eta}})^{2p}+N^{-D'}\eta^{-2kp}\Big),
\end{align*}
which implies  by Markov's inequality  that 
\begin{align}
\mathbb{P}\Big(|\mathcal{P}_{ii}\varphi(\Gamma_i)|\geq \frac{N^{\frac{\varepsilon}{4}}}{\sqrt{N\eta}}\Big)\leq C_p \Big(\frac{N^{\frac{\varepsilon}{4}}}{\sqrt{N\eta}}\Big)^{-2p}N^{2p\varepsilon'} \Big( (\frac{N^{\varepsilon'}}{\sqrt{N\eta}})^{2p}+N^{-D'}\eta^{-2kp}\Big). \label{19021802}
\end{align}
 For the given $\varepsilon>0$ in Lemma~\ref{pro.091379}, by 
 first choosing $\varepsilon'=\varepsilon'(\varepsilon)$ to be smaller   than $\frac{\varepsilon}{8}$, and then choosing  $p=p(\varepsilon, D)$ to be sufficiently large, we get
\begin{align}
C_p \Big(\frac{N^{\frac{\varepsilon}{4}}}{\sqrt{N\eta}}\Big)^{-2p}N^{2p\varepsilon'}  \Big(\frac{N^{\varepsilon'}}{\sqrt{N\eta}}\Big)^{2p}\leq \frac{1}{2}N^{-D}. \label{19021803}
\end{align}
Then, by further choosing $D'=D'(\varepsilon, D)$ sufficiently large, we can guarantee 
\begin{align}
C_p \Big(\frac{N^{\frac{\varepsilon}{4}}}{\sqrt{N\eta}}\Big)^{-2p}N^{2p\varepsilon'}N^{-D'}\eta^{-2kp}\leq \frac{1}{2}N^{-D}.  \label{19021804}
\end{align} 
 With  these choices of $\varepsilon'$ and $D'$, we now set $N_2(D,\varepsilon)\deq N(D',\varepsilon')$.

Further, by (\ref{19021802})-(\ref{19021804}), there exists an event $\Omega(z)$, such that 
\begin{align*}
\mathbb{P}(\Omega(z))\geq 1-N^{-D}, \qquad N\geq N_2(D,\varepsilon)
\end{align*}
and 
\begin{align*}
 |\mathcal{P}_{ii}\varphi(\Gamma_i)|\leq  \frac{N^{\frac{\varepsilon}{4}}}{\sqrt{N\eta}}\,, \qquad \text{on}\quad \Omega(z)\,.
\end{align*}
This now implies that 
$|\mathcal{P}_{ii}|\leq  \frac{N^{\frac{\varepsilon}{4}}}{\sqrt{N\eta}}$ on $\Theta(z, N^{-\frac{\gamma}{4}})\cap \Omega(z)$.  Similarly, by working on $\wt{\mathfrak{n}}_i$, we can get $
 |\mathcal{K}_{ii}|\leq  \frac{N^{\frac{\varepsilon}{4}}}{\sqrt{N\eta}}$ on $\Theta(z, N^{-\frac{\gamma}{4}})\cap \Omega(z)$. 
 
The same bound can be obtained for  $\mathcal{P}_{ij}, \mathcal{P}_{\hat{i}j}$, $\mathcal{K}_{ij}$ and $\mathcal{K}_{\hat{i}j}$ for $j=i,\hat{i}$. 
The remaining argument is the same as the proof of (\ref{071602}) in Theorem~\ref{071601}. The only change is, instead of the notation $\prec$, we use the deterministic $\leq$, but restricting onto    the event $\Theta(z, N^{-\frac{\gamma}{4}})\cap \Omega(z)$. 

More specifically,  the quantitative proof of \eqref{101451} yields that 
\begin{align}
\Lambda_{\rm d}^c(z)\leq \frac{N^{\frac{\varepsilon}{2}}}{\sqrt{N\eta}},\quad \wt{\Lambda}_{\rm d}^c(z)\leq \frac{N^{\frac{\varepsilon}{2}}}{\sqrt{N\eta}},\quad \Lambda_T(z)\leq \frac{N^{\frac{\varepsilon}{2}}}{\sqrt{N\eta}},\quad \wt{\Lambda}_T(z)\leq \frac{N^{\frac{\varepsilon}{2}}}{\sqrt{N\eta}}  \label{101567}
\end{align}
hold on the event $\Theta (z, N^{-\frac{\gamma}{4}})\cap \Omega(z)$,  for all $N\geq N_3(D,\gamma,\varepsilon)$.

Therefore, by the definitions of $\Lambda_{\rm d}^c$ and $\wt{\Lambda}_{\rm d}^c$, we have
\begin{align}
&\Big|G_{ii}-\frac{\omega_B^c}{|\xi_i|^2-(\omega_B^c)^2}\Big| \leq\frac{N^{\frac{\varepsilon}{2}}}{\sqrt{N\eta}}\,,\qquad \Big|\mathcal{G}_{ii}-\frac{\omega_A^c}{|\sigma_i|^2-(\omega_A^c)^2}\Big|\leq \frac{N^{\frac{\varepsilon}{2}}}{\sqrt{N\eta}}\,,\nonumber\\
& \Big|G_{\hat{i}\hat{i}}-\frac{\omega_B^c}{|\xi_i|^2-(\omega_B^c)^2}\Big| \leq\frac{N^{\frac{\varepsilon}{2}}}{\sqrt{N\eta}}\,,\qquad \Big|\mathcal{G}_{\hat{i}\hat{i}}-\frac{\omega_A^c}{|\sigma_i|^2-(\omega_A^c)^2}\Big|\leq \frac{N^{\frac{\varepsilon}{2}}}{\sqrt{N\eta}}\,,\label{091377}
\end{align}
for all $i\in\llbracket 1,N\rrbracket$, on the event  $\Theta (z, N^{-\frac{\gamma}{4}})\cap \Omega(z)$ for all $N\geq N_3(D,\gamma,\varepsilon)$. Averaging the above estimates over $i$, we obtain the system of equations
\begin{align}
m_H(z)=m_A(\omega_B^c(z))+r_A(z)\,,\nonumber\\
m_H(z)=m_B(\omega_A^c(z))+r_B(z)\,,\nonumber\\
\omega_A^c(z)+\omega_B^c(z)=z-\frac{1}{m_H(z)}\,, \label{091360}
\end{align}
where the error terms $r_A(z)$ and $r_B(z)$ satisfy
$
|r_A(z)|, |r_B(z)|\leq \frac{C N^{\frac{\varepsilon}{2}}}{\sqrt{N\eta}}$ on the event  $\Theta (z, N^{-\frac{\gamma}{4}})\cap \Omega(z)$ for all $N\geq N_3(D,\gamma,\varepsilon)$. 
Here the last equation in~(\ref{091360}) follows from the definition~(\ref{091107}) or~(\ref{091361}). From the definition of $\Theta(z,\delta)$ in~(\ref{101465}),~(\ref{091107}) or~(\ref{091361}), and the equations in~(\ref{le definiting equations}) with $(\mu_1,\mu_2)=(\mu_A,\mu_B)$,  it is not difficult to check  that 
\begin{align*}
|\omega_A^c-\omega_A|\leq CN^{-\frac{\gamma}{4}}\,,\qquad |\omega_B^c-\omega_B|\leq  CN^{-\frac{\gamma}{4}} 
\end{align*}
hold on $\Theta (z, N^{-\frac{\gamma}{4}})$.  In particular, with the help of~(\ref{110270}), this guarantees 
that the imaginary parts of $\omega_A^c$ and $\omega_B^c$ are separated away from zero,
hence so are $m_A(\omega_B^c)$ and $m_B(\omega_B^c)$. This allows us to rewrite~(\ref{091360}) as 
\begin{align}\label{Phierror}
\big\|\Phi_{\mu_A,\mu_B} (\omega_A^c,\omega_B^c,z)\big\|=\wt{r}(z)\,,
\end{align}
where $\wt{r}(z)=(\wt{r}_A(z),\wt{r}_B(z))'$ satisfy $|\wt{r}_A(z)|, |\wt{r}_B(z)|\leq \frac{C N^{\frac{\varepsilon}{2}}}{\sqrt{N\eta}}$  on the event $\Theta (z, N^{-\frac{\gamma}{4}})\cap \Omega(z)$ for all $N\geq N_3(D,\gamma,\varepsilon)$. 
Applying the stability of the system $\Phi_{\mu_A,\mu_B} (\omega_A,\omega_B,z)=0$ (see Theorem 4.1 of \cite{BES15}), we obtain 
\begin{align}
\big|\omega_A^c-\omega_A\big|\leq \frac{C N^{\frac{\varepsilon}{2}}}{\sqrt{N\eta}}\,,\qquad \big|\omega_B^c-\omega_B\big|\leq \frac{C N^{\frac{\varepsilon}{2}}}{\sqrt{N\eta}}\,, \label{091375}
\end{align}
on the event $\Theta (z, N^{-\frac{\gamma}{4}})\cap \Omega(z)$ for all $N\geq N_3(D,\gamma,\varepsilon)$. Substituting~(\ref{091375}) into
 the definition of  $\Lambda_{\rm d}^c$ and $\widetilde \Lambda_{\rm d}^c$,  we see that the first two inequalities in
 ~(\ref{101567}) imply similar bounds for~$\Lambda_{\rm d}$ and~$\widetilde \Lambda_{\rm d}$.
  This  completes the proof of Lemma~\ref{pro.091379}. 
\end{proof}

With   Lemma~\ref{pro.091379},  the remaining proof of Theorem~\ref{thm.green function subordination}   
 closely follows that for Theorem 2.5 in \cite{BES15b}, so we will only sketch the argument. 
 We start with the result with large $\eta=\eta_M$ for some large but fixed positive constant $\eta_{\rm M}$. More specifically, from Lemma~\ref{lem.081970}, we see that 
 \begin{align}
\Lambda_{\rm d}(E+\ii \eta_{\rm M})\prec\frac{1}{\sqrt{N\eta_{\rm M}^4}}\,,\qquad  \wt{\Lambda}_{\rm d}(E+\ii \eta_{\rm M})\prec\frac{1}{\sqrt{N\eta_{\rm M}^4}}\,,  \label{112501}
 \end{align}
 for any fixed $E\in\R$. The second estimate in~(\ref{112501}) can be obtained from Lemma~\ref{lem.081970} since one can apply this lemma to $\mathcal{H}$ as well.  In addition, using the trivial bound $\|G\|\leq \frac{1}{\eta}$ and inequality $|\mathbf{x}^* G\mathbf{y}|\leq \|G\|\|\mathbf{x}\|_2\|\mathbf{y}\|_2$, we also have
 \begin{align}
 \Lambda_T(E+\ii \eta_{\rm M})\leq  \frac{1}{\eta_{\rm M}}\,,\qquad \wt{\Lambda}_{T}(E+\ii \eta_{\rm M})\leq  \frac{1}{\eta_{\rm M}}\,,\label{112502}
 \end{align}
 for any fixed $E\in \mathcal{B}_{\mu_\alpha\boxplus\mu_\beta}$.  According to the definition of $\Theta(z, \delta)$ in~(\ref{101465}), (\ref{112501}) and~(\ref{112502}), we see that for any fixed $E\in \mathcal{B}_{\mu_\alpha\boxplus\mu_\beta}$ and $D>0$, 
 \begin{align}
 \mathbb{P} \big( \Theta (E+\ii \eta_{\rm M}\,, N^{-\frac{3\gamma}{8}})\big)\geq 1-N^{-D}\,, \label{112503}
 \end{align}
 holds for all $N\geq N_0(D,\gamma)$ for some positive integer $N_0(D,\gamma)$.
 
 Starting with~(\ref{112503}), we conduct a standard continuity argument, whose setup is best suited to our problem 
  in the form   presented  in \cite{BES15b}. Specifically, we do a bootstrap by reducing~$\eta$ in very small steps, $N^{-5}$ (say),  starting from $\eta_{\rm M}$  and successively control the probability of the ``good" events $\Theta$. 
    Recall the event $\Omega(z)$ in Lemma~\ref{pro.091379}. The main task is to show for any fixed $E\in \mathcal{I}$ and any $\eta\in [\eta_{\rm m}, \eta_{\rm M}]$, 
 \begin{align}
 \Theta (E+\ii \eta, N^{-\frac{3\gamma}{8}})\cap \Omega (E+\ii(\eta-N^{-5}))\subset \Theta (E+\ii (\eta-N^{-5}), N^{-\frac{3\gamma}{8}})\,,  \label{112512}
 \end{align} 
which is the analogue of~(7.20) of~\cite{BES15b}. 
To see this  inclusion, one first uses the Lipschitz continuity of the Green function, $\|G(z)-G(z')\|\leq N^2|z-z'|$, and
of the subordination functions, \cf~(\ref{110270}), to obtain 
\begin{align}
\Theta (E+\ii \eta, N^{-\frac{3\gamma}{8}})\subset  \Theta (E+\ii (\eta-N^{-5}), N^{-\frac{\gamma}{4}})\,. \label{112510}
\end{align} 
Then~(\ref{112510}) together with~(\ref{112511}) implies~(\ref{112512}). Using~(\ref{112512}) recursively, one goes from~$\eta_{\rm M}$ down to $\eta_{\rm m}$, step by step.  The remaining proof of~(\ref{091601}),  based on~(\ref{112512}) and Lemma~\ref{pro.091379}, is the same as the counterpart in \cite{BES15b} (\cf (7.20)-(7.25) therein). We omit the details. 
 
With~(\ref{091601}), we can prove~(\ref{101470}) in the sequel. The first two inequalities in~(\ref{101470}) have already been proved in ~(\ref{091375}) with a fixed $\eta$, under~(\ref{0913128}). The uniformity then follows from~(\ref{091601}) which holds uniformly on $\mathcal{S}_{\mathcal{I}}(\eta_{\mathrm{m}}, \eta_{\mathrm{M}})$.   Then the last inequality in~(\ref{101470}) follows from the first two, together with the last equation in~(\ref{091360}) and  the second equation in~(\ref{le definiting equations}) with $(\mu_1,\mu_2)=(\mu_A,\mu_B)$.
This completes the proof of Theorem~\ref{thm.green function subordination}.
\end{proof}

\section{Strong law for small $\eta$} \label{s.strong law}

In this section, we prove the strong law, \ie Theorem~\ref{thm.081501},  for $z\in\mathcal{S}_{\mathcal{I}}(0, \eta_{\mathrm{M}})$. It suffices to work on the regime $z\in\mathcal{S}_{\mathcal{I}}(\eta_{\rm m}, \eta_{\mathrm{M}})$ at first. The extension to $z\in\mathcal{S}_{\mathcal{I}}(0, \eta_{\mathrm{M}})$ will be easy. Our main task is to establish the fluctuation averaging for the quantities $\mathcal{P}_{ij}$ defined in \eqref{071560}.   

\begin{lem} [Fluctuation averaging]  \label{lem.071920} Suppose that the assumptions in Theorem~\ref{thm.081501} hold. 
Let $\eta_{\mathrm{M}}>0$ be any  (large) constant  and $\gamma>0$ be any (small) constant  in the definition of $\eta_{\rm m}$ (\cf~(\ref{etamdef})).   For any fixed integer $p\geq 1$, and deterministic numbers $d_1,\ldots, d_N\in \mathbb{C}$ satisfying $\max_{i\in \llbracket 1, N\rrbracket} |d_i|\leq 1$, we have
\begin{align}
&\Big|\frac{1}{N} \sum_{i=1}^N d_i \mathcal{P}_{ii}\Big|\prec \Psi^2,\qquad  \Big|\frac{1}{N} \sum_{i=1}^N d_i \mathcal{P}_{\hat{i}i}\Big|\prec \Psi^{2},\nonumber\\
&\Big|\frac{1}{N} \sum_{i=1}^N d_i \mathcal{P}_{i\hat{i}}\Big|\prec\Psi^2,\qquad  \Big|\frac{1}{N} \sum_{i=1}^N d_i \mathcal{P}_{\hat{i}\hat{i}}\Big|^2\prec \Psi^2 \label{072102}
\end{align}
uniformly on $\mathcal{S}_{\mathcal{I}}(\eta_{\mathrm{m}}, \eta_{\mathrm{M}})$. 
\end{lem}

We will often use the following improvement of ~(\ref{102125}),
\begin{align}
\tau_a(G)&= m_{\mu_A\boxplus\mu_B}+O_\prec(\Psi)\,,\nonumber\\
\tau_a(\wt{B}G)&= (z-\omega_B) m_{\mu_A\boxplus\mu_B} +O_\prec(\Psi)\,,\nonumber\\
\tau_a(G\wt{B})&= (z-\omega_B) m_{\mu_A\boxplus\mu_B} +O_\prec(\Psi)\,,\nonumber\\
\tau_a(\wt{B}G\wt{B})&=(\omega_B-z)(1+(\omega_B-z)m_{\mu_A\boxplus\mu_B})+O_\prec(\Psi)\,,\qquad a=1,2\,, 
\label{112133}
\end{align}
which can be proved in the same way as~(\ref{102125}), but with the first inequality in~(\ref{071422}) replaced by the first inequality in~(\ref{091601}), as the input of the proof.

 In the next Section~\ref{thm53}
 we will show how to prove Theorem~\ref{thm.081501} on $\mathcal{S}_{\mathcal{I}}(0, \eta_{\mathrm{M}})$
with the aid of Lemma~\ref{lem.071920}. Then, in Section~\ref{lm81} we will prove Lemma~\ref{lem.071920}. 

\subsection{Proof of Theorem~\ref{thm.081501} on $\mathcal{S}_{\mathcal{I}}(0, \eta_{\mathrm{M}})$}
\label{thm53}
To prove the strong law from Lemma~\ref{lem.071920}, first of all, we need to derive that the estimates
\begin{align}
|\Upsilon_{1}|\prec \Psi^2, \qquad |\Upsilon_{2}|\prec \Psi^2\label{083110}
\end{align}
hold uniformly on  $\mathcal{S}_{\mathcal{I}}(\eta_{\rm m}, \eta_{\mathrm{M}})$. These are the strongest high probability bounds 
related to the Ward identities in~(\ref{071604}). To see~(\ref{083110}), we choose $d_i=1$ for all $i\in \llbracket 1, N \rrbracket$ in~(\ref{072102}). From the definition of $\mathcal{P}_{ii}$ in~(\ref{071560}), we get
\begin{align}
\frac{1}{N}\sum_{i=1}^N \mathcal{P}_{ii}=\frac{\Upsilon_{1}}{N}\sum_{i=1}^N \big( G_{ii}+ T_{ii}\big)&= \Upsilon_{1}\Big( \tau_1(G)+ \frac{1}{N}\sum_{i=1}^N T_{ii}\Big)\nonumber\\ &=\Upsilon_{1}\Big( m_{\mu_A\boxplus\mu_B}+O_\prec(\Psi)\Big)\,, \label{112140}
\end{align}
 where in the last step we used~(\ref{112133}) and the third inequality in~(\ref{091601}). Then, using the  lower bound of  $\Im m_{\mu_A\boxplus\mu_B}=\Im m_{\mu_A}(\omega_B)$ inherited from the lower bound of $\Im \omega_B$ in~(\ref{110270}), and also  the first bound in~(\ref{072102}),   we can easily see $|\Upsilon_{1}|\prec \Psi^2$ from~(\ref{112140}). Similarly, we can also show $|\Upsilon_{2}|\prec \Psi^2$. 
 Notice that {\it a posteriori} we could have defined $\mathcal{P}_{ij}$ in \eqref{071560} without the last
term involving $\Upsilon_a$ with $a=1,2$, since we are interested only up to $O_\prec(\Psi^2)$ precision. We do not, however,
know how to prove directly that $\Upsilon_a = O_\prec(\Psi^2)$ without first proving a fluctuation averaging 
result  \eqref{072102} involving the quantity $\mathcal{P}_{ij}$ with $\Upsilon_a$. The correct choice
of $\mathcal{P}_{ij}$ is the essential idea of the entire proof.

Plugging~(\ref{083110}) back to the definition of $\mathcal{P}_{ii}$, $\mathcal{P}_{\hat{i}i}$, $\mathcal{P}_{i\hat{i}}$ and $\mathcal{P}_{\hat{i}\hat{i}}$ in~(\ref{071560}), we obtain from~(\ref{072102}),
\begin{align}
&\bigg|\frac{1}{N}\sum_{i=1}^N d_i \Big( G_{ij} \tau_1(\wt{ B }G)-\big(\wt{B}G\big)_{ij}\tau_1(G)\Big)\bigg|\prec \Psi^2,\nonumber\\
&\bigg|\frac{1}{N}\sum_{i=1}^N d_i \Big( G_{\hat{i}j} \tau_2(\wt{ B }G)-\big(\wt{B}G\big)_{\hat{i}j}\tau_2(G)\Big)\bigg|\prec \Psi^2\,,\qquad\qquad j=i,\hat{i}\,, \label{0913100}
\end{align}
for any deterministic numbers $d_1,\ldots, d_N\in \mathbb{C}$ satisfying $|d_i|\lesssim 1$, 
which is a shorthand notation for $|d_i|\le C$ with some constant $C$. 
 While  Lemma~\ref{lem.071920}
was formulated for $|d_i|\le 1$, it clearly holds as long as  $|d_i|\lesssim 1$. 
Recall the notation introduced in~(\ref{102401}). We claim that  the following estimates can be derived from~(\ref{0913101}) and ~(\ref{0913100}):
\begin{align}
&\Big|\frac{1}{N} \sum_{i=1}^N d_i \Big(G_{ii}-\frac{\omega_{B,2}^c}{|\xi_i|^2- \omega_{B,1}^c \omega_{B,2}^c}\Big)\Big|\prec \Psi^2\,,\nonumber\\ &\Big|\frac{1}{N} \sum_{i=1}^N d_i \Big(G_{\hat{i}i}-\frac{\bar{\xi}_i}{ |\xi_i|^2-\omega_{B,1}^c \omega_{B,2}^c}\Big)\Big|\prec \Psi^2,\nonumber\\
&\Big|\frac{1}{N} \sum_{i=1}^N d_i \Big(G_{\hat{i}\hat{i}}-\frac{\omega_{B,1}^c}{ |\xi_i|^2-\omega_{B,1}^c\omega_{B,2}^c}\Big)\Big|\prec \Psi^2\,,\nonumber\\ & \Big|\frac{1}{N} \sum_{i=1}^N d_i \Big(G_{i\hat{i}}-\frac{\xi_i}{ |\xi_i|^2-\omega_{B,1}^c\omega_{B,2}^c}\Big)\Big|\prec \Psi^2. \label{080150}
\end{align}
We derive the first estimate in~(\ref{080150}), 
the others are proven similarly. We~write 
\begin{multline*}
\frac{1}{N} \sum_{i=1}^N d_i \Big(G_{ii}-\frac{\omega_{B,2}^c}{|\xi_i|^2- \omega_{B,1}^c \omega_{B,2}^c}\Big)\\=\frac{1}{N}\sum_{i=1}^N\frac{d_i}{|\xi_i|^2- \omega_{B,1}^c \omega_{B,2}^c}\Big(G_{ii}\big(|\xi_i|^2- \omega_{B,1}^c \omega_{B,2}^c\big)-\omega_{B,2}^c\Big).
\end{multline*}
Applying Theorem~\ref{thm.green function subordination},  and \eqref{071710} along its proof,  it is easy to check that
\begin{align}\label{omegaom}
\omega_{B,a}^c = \omega_B + O_\prec(\Psi)\,,\quad  \qquad a=1,2\,,
\end{align}
hence 
\begin{align}
\omega_{B,1}^c \omega_{B,2}^c=\omega_B^2+O_\prec(\Psi)\,,\qquad G_{ii}\big(|\xi_i|^2- \omega_{B,1}^c \omega_{B,2}^c\big)-\omega_{B,2}^c=O_\prec(\Psi)\,. \label{110290}
\end{align}
Moreover, from   the lower bound on $\im \omega_B$ from  ~(\ref{110270}) and the first estimate of~(\ref{110290}), we have
\begin{align}
\frac{1}{{|\xi_i|^2- \omega_{B,1}^c \omega_{B,2}^c}}=\frac{1}{|\xi_i|^2- \omega_{B}^2}+O_\prec(\Psi)\,. \label{110295}
\end{align}
 Then, in light of~(\ref{110270}),~(\ref{110290}) and~(\ref{110295}), it suffices to check 
\begin{align}
\frac{1}{N}\sum_{i=1}^N d_i\Big(G_{ii}\big(|\xi_i|^2- \omega_{B,1}^c \omega_{B,2}^c\big)-\omega_{B,2}^c\Big)=O_\prec(\Psi^2)\,, \label{0913120}
\end{align}
for any deterministic numbers $d_1,\ldots, d_N\in \mathbb{C}$ satisfying $|d_i|  \lesssim  1$
 (here we redefined $d_i$ to  $d_i/( |\xi_i|^2- \omega_{B}^2)$). Using~(\ref{0913101}), we can write
\begin{align}
G_{ii}\big(|\xi_i|^2- \omega_{B,1}^c \omega_{B,2}^c\big)-\omega_{B,2}^c=&-\frac{\omega_{B,2}^c}{\tau_1(G)}\Big(\big(\wt{B}G\big)_{ii}\tau_1(G)-G_{ii}\tau_1(\wt{ B }G)\Big)\nonumber\\
&-\frac{\xi_i}{\tau_2(G)}\Big(\big(\wt{B}G\big)_{\hat{i}i}\tau_2(G)-G_{\hat{i}i}\tau_2(\wt{ B }G) \Big)\,. \label{0913121}
\end{align}
Then, from~(\ref{112133}) and~\eqref{omegaom}, we see that
\begin{align}
&\frac{\omega_{B,2}^c}{\tau_1(G)}=\frac{\omega_B}{m_{\mu_A\boxplus\mu_B}}+O_\prec(\Psi)\,,& &\frac{\xi_i}{\tau_2(G)}=\frac{\xi_i}{m_{\mu_A\boxplus\mu_B}}+O_\prec(\Psi)\,,\nonumber\\
 &\big(\wt{B}G\big)_{ii}\tau_1(G)-G_{ii}\tau_1(\wt{ B }G)=O_\prec(\Psi)\,, & &\big(\wt{B}G\big)_{\hat{i}i}\tau_2(G)-G_{\hat{i}i}\tau_2(\wt{ B }G)=O_\prec(\Psi)\,,  \label{0913122}
\end{align}
 where the second line follows from \eqref{1023110}. 
Thus combining~(\ref{0913121}),~(\ref{0913122}) and~(\ref{0913100}) yields~(\ref{0913120}), which implies~(\ref{080150}) according to the discussion above. 

 Notice that in this argument it was essential that $G_{ii}$ was approximated  in \eqref{080150} not
by $\omega_B/(|\xi_i|^2 - \omega_B^2)$ or by $\omega_B^c/(|\xi_i|^2 - (\omega_B^c)^2)$ but by
$$
   G_{ii} \approx \frac{\omega_{B,2}^c}{|\xi_i|^2- \omega_{B,1}^c \omega_{B,2}^c}\,,
$$
since this latter approximation is precise up to $O_\prec(\Psi^2)$ after averaging,  while the previous ones are
{\it a priori}  correct only with an error
$O_\prec(\Psi)$.

Next, we show that \eqref{080150} nevertheless holds if we  approximate $G_{ii}$ by $\omega_B^c/(|\xi_i|^2 - (\omega_B^c)^2)$.
Choosing all $d_i= 1$ in the first and third inequalities in~(\ref{080150}) and applying~(\ref{neu083140}), we note that 
\begin{align*}
\omega_{B,a}^c=\omega_{B}^c+O_\prec(\Psi^2)\,,\quad\qquad a=1,2\,,
\end{align*}
 so the first approximation in \eqref{omegaom} is actually one order better. 
Thus we get from~(\ref{080150}) that 
\begin{align}
&\Big|\frac{1}{N} \sum_{i=1}^N d_i \Big(G_{ii}-\frac{\omega_{B}^c}{|\xi_i|^2- (\omega_{B}^c)^2}\Big)\Big|\prec \Psi^2,\quad \Big|\frac{1}{N} \sum_{i=1}^N d_i \Big(G_{\hat{i}i}-\frac{\bar{\xi}_i}{ |\xi_i|^2-(\omega_{B}^c)^2}\Big)\Big|\prec \Psi^2,\nonumber\\
&\Big|\frac{1}{N} \sum_{i=1}^N d_i \Big(G_{\hat{i}\hat{i}}-\frac{\omega_{B}^c}{ |\xi_i|^2-(\omega_{B}^c)^2}\Big)\Big|\prec \Psi^2,\quad \Big|\frac{1}{N} \sum_{i=1}^N d_i \Big(G_{i\hat{i}}-\frac{\xi_i}{ |\xi_i|^2-(\omega_{B}^c)^2}\Big)\Big|\prec \Psi^2. \label{083150}
\end{align}
Further, recalling the definitions of  $\mathcal{H}$ and $\mathcal{G}$ in~(\ref{0913125}).  Switching the r\^oles of $A$ and $B$, and also the r\^oles of $U$ and $U^*$ in the above discussions, we  have 
\begin{align}
&\Big|\frac{1}{N} \sum_{i=1}^N d_i \Big(\mathcal{G}_{ii}-\frac{\omega_{A}^c}{|\sigma_i|^2- (\omega_{A}^c)^2}\Big)\Big|\prec \Psi^2,\quad \Big|\frac{1}{N} \sum_{i=1}^N d_i \Big(\mathcal{G}_{\hat{i}i}-\frac{\bar{\sigma}_i}{ |\sigma_i|^2-(\omega_{A}^c)^2}\Big)\Big|\prec \Psi^2,\nonumber\\
&\Big|\frac{1}{N} \sum_{i=1}^N d_i \Big(\mathcal{G}_{\hat{i}\hat{i}}-\frac{\omega_{A}^c}{ |\xi_i|^2-(\omega_{A}^c)^2}\Big)\Big|\prec \Psi^2,\quad \Big|\frac{1}{N} \sum_{i=1}^N d_i \Big(\mathcal{G}_{i\hat{i}}-\frac{\sigma_i}{ |\sigma_i|^2-(\omega_{A}^c)^2}\Big)\Big|\prec \Psi^2.  \label{083151}
\end{align}

Applying~(\ref{083150}) and~(\ref{083151}) to average over the diagonal entries of the Green functions $G$ and $\mathcal{G}$, and also using the fact $\ntr G (z)=\ntr \mathcal{G}(z) =m_H(z)$, we see that
\begin{align*}
m_H(z)=\int_\R  \frac{\omega_{B}^c}{x^2-(\omega_{B}^c)^2} {\rm d} \mu_\Xi(x)+O_\prec(\Psi^2)=\int_\R  \frac{\omega_{A}^c}{x^2-(\omega_{A}^c)^2} {\rm d} \mu_\Sigma (x)+O_\prec(\Psi^2)\,.
\end{align*}
 From this, using
$$
    \frac{\omega_{B}^c}{x^2-(\omega_{B}^c)^2} = \frac{1}{2}\Big[ \frac{1}{x-\omega_{B}^c}  +\frac{1}{-x-\omega_{B}^c}\Big]\,,
$$
we can get 
\begin{align}
m_H(z)=m_A(\omega_{B}^c(z))+O_\prec(\Psi^2)=m_B(\omega_{A}^c(z))+O_\prec(\Psi^2)\,, \label{050980}
\end{align}
where we used the fact $\mu_A\equiv \mu_\Xi^{\text{sym}}$ and $\mu_B\equiv \mu_\Sigma^{\text{sym}}$, in light of ~(\ref{050970}). In addition, we also have~(\ref{091110}).   Summarizing these estimates, we have
$\Phi_{\mu_A,\mu_B}(\omega_A^c,\omega_B^c,z)=O_\prec (\Psi^2)$, \ie compared with \eqref{Phierror},
 we improved the error
in the approximate subordination equations.

 Similarly to the proof of  Lemma~\ref{pro.091379}, we  use the stability of the system $\Phi_{\mu_A,\mu_B}(\omega_A,\omega_B,z)=0$ again, but with the improved error $\Psi^2$. We also note that the estimates from  Theorem~\ref{thm.green function subordination}  and Lemma~\ref{lem.071920}   used in the above discussion hold uniformly on $\mathcal{S}_{\mathcal{I}}(\eta_{\mathrm{m}}, \eta_{\mathrm{M}})$. 
Hence, we can conclude the proof of Theorem~\ref{thm.081501} on $\mathcal{S}_{\mathcal{I}}(\eta_{\mathrm{m}}, \eta_{\mathrm{M}})$.  

At the end, we extend~(\ref{the local law for mH prec inequality}) from $\mathcal{S}_{\mathcal{I}}(\eta_{\mathrm{m}}, \eta_{\mathrm{M}})$ to $\mathcal{S}_{\mathcal{I}}(0, \eta_{\mathrm{M}})$. The extension relies on a standard use of the monotonicity of the Green function: For all $i\in \llbracket 1, N \rrbracket $ and $j=i$ or $\hat{i}$, we~have
\begin{align*}
|G'_{jj}(z)|=\big|\sum_{k=1}^{2N} G_{jk}(z) G_{kj}(z)\big|\leq \sum_{k=1}^{2N} |G_{jk}(z)|^2=\frac{\Im G_{jj}(z)}{\eta},
\end{align*}
where the last step follows from the spectral decomposition.  In addition, note that the function $s\mapsto s\Im G_{jj}(E+\ii s)$ is monotonically increasing. This implies that for any $\eta\in (0, \eta_{\rm m}]$, 
\begin{align}
 \big| G_{jj}  (E+\ii \eta)- G_{jj}(E+\ii\eta_{\rm m})\big|&\leq  \int_{\eta}^{\eta_{\rm m}} \frac{s\Im G_{jj}(E+\ii s)}{s^2} \,{\rm d} s\nonumber\\
 &\leq 2\frac{\eta_{\rm m}}{\eta} \Im G_{jj} (E+\ii \eta_{\rm m})\leq C\frac{N^{\gamma}}{N\eta} \leq CN^{\gamma} \Psi^2\,, \label{112801}
\end{align} 
with high probability, for any $E\in \mathcal{I}$. Here we used $|G_{jj}(E+\ii \eta_{\rm m})|\prec 1$ which follows from the first bound in~(\ref{091601}).  On the other hand, for any $i\in \llbracket 1, N\rrbracket$, we also have 
\begin{align}
\bigg| \frac{\omega_B(E+\ii\eta)}{|\xi_i|^2-\omega_B^2(E+\ii \eta)}- \frac{\omega_B(E+\ii\eta_{\rm m})}{|\xi_i|^2-\omega_B^2(E+\ii \eta_{\rm m})}\bigg|\leq C(\eta_{\rm m}-\eta)\leq \Psi^2, \quad  \label{112802}
\end{align}
$\eta\in (0,\eta_{\rm m}]$, $E\in \mathcal{I}$, for sufficiently small $\gamma$, which follows from the upper bound of $\omega_B'(z)$, the lower bound of $|\xi_i|^2-\omega_B^2(z)$ which follows from the lower bound of $\Im \omega_B$, and also the upper bound of $\omega_B$, in Lemma~\ref{lem.A.1}.  Combining~(\ref{112801}) and~(\ref{112802}), and using~(\ref{111601}),  we conclude that~(\ref{the local law for mH prec inequality})
 holds   uniformly on $\mathcal{S}_{\mathcal{I}}(0, \eta_{\mathrm{M}})$. 
 This completes the proof of  Theorem~\ref{thm.081501} on $\mathcal{S}_{\mathcal{I}}(0, \eta_{\mathrm{M}})$.

Hence, what remains is to prove Lemma ~\ref{lem.071920}.
\subsection{Proof of Lemma ~\ref{lem.071920}}\label{lm81}
Since the proofs for the four estimates in~(\ref{072102}) are nearly the same,  we only present  the details  for the first one. 
First of all, from~(\ref{091601}) and~(\ref{071604}) we have
\begin{align}
\big|T_{ii}\Upsilon_{1}\big|\prec \Psi^2\,. \label{100601}
\end{align}
Hence, it suffices to bound the weighted average of the  following slight modifications of $\mathcal{P}_{ii}$'s:
\begin{align}
\mathcal{Q}_{ii}\equiv\mathcal{Q}_{ii}(z)\deq(\wt{ B }G)_{ii} \tau_1(G) -G_{ii} \tau_1(\wt{ B }G)+ G_{ii} \Upsilon_{1}\,,\qquad i\in \llbracket 1, N\rrbracket\,.  \label{defofQ}
\end{align}
Then we introduce the notation 
\begin{align*}
\mathfrak{m}(k,l)\deq\Big(\frac{1}{N}\sum_{i=1}^N d_i\mathcal{Q}_{ii}\Big)^k\Big({\frac{1}{N}\sum_{i=1}^N \overline{d_i} \,\overline{ \mathcal{Q}_{ii}}}\Big)^{l}\,.
\end{align*}
Similarly to Lemma~\ref{lem.071420}, the main technical task is the following recursive moment estimate.
\begin{thm}[Recursive moment estimate] \label{pro.053020} Suppose that the assumptions in Theorem~\ref{thm.081501} hold. 
Let $\eta_{\mathrm{M}}>0$ be any  (large) constant  and $\gamma>0$  in \eqref{etamdef} be any (small) constant. For any fixed integer $p\geq 1$, we~have
\begin{multline}
\mathbb{E}\big[\mathfrak{m}(p,p)\big]=\mathbb{E}\big[O_\prec (\Psi^2) \mathfrak{m}(p-1,p)\big]+\mathbb{E}\big[ O_\prec(\Psi^4)\mathfrak{m}(p-2,p)\big]\\+\mathbb{E}\big[O_\prec(\Psi^4)\mathfrak{m}(p-1,p-1)\big]\,, \label{053050}
\end{multline}
uniformly on $\mathcal{S}_{\mathcal{I}}(\eta_{\mathrm{m}}, \eta_{\mathrm{M}})$, where we made the convention $\mathfrak{m}(0,0)=1$ and $\mathfrak{m}(-1,1)=0$ if $p=1$. 
\end{thm}

The reason why we prefer to work with   $\mathcal{Q}_{ii}$ instead of $\mathcal{P}_{ii} = \mathcal{Q}_{ii} + T_{ii}\Upsilon_i$ is as follows. To prove Theorem~\ref{pro.053020},  we will follow a similar strategy as the proof of Lemma~\ref{lem.071420}. 
In  Lemma~\ref{lem.071420} and its proof, we worked on  $\mathcal{P}_{ii}$ directly. The derivative $\frac{\partial T_{ii}}{\partial g_{ik}^u}$ was necessary for the proof of  Lemma~\ref{lem.071420}, \cf~(\ref{071520}). However, in the proof of Theorem~\ref{pro.053020}, we would need to consider the derivative $\frac{ \partial T_{ii}}{\partial g_{jk}}$ for all $j\neq i$ if we carry the term $T_{ii}\Upsilon_{1}$ from $\mathcal{P}_{ii}$ in the discussion. Unfortunately, the dependence of the factor $(\mathbf{k}_i^u)^*$ in $T_{ii}$ (\cf~(\ref{053070}))  on $g_{jk}^u$ for $j\neq i$ is difficult to capture. On the other hand,  at this stage of the proof we already 
have the bound~(\ref{100601}) available and this allows us to drop the term $T_{ii}\Upsilon_{1}$ from the beginning.

With the aid of Theorem~\ref{pro.053020}, one can prove Lemma ~\ref{lem.071920}.
\begin{proof}[Proof of Lemma ~\ref{lem.071920}] 
Similarly to the proof of~(\ref{071302}) for $\mathcal{P}_{ii}$ from Lemma~\ref{lem.071420}, one can apply Young's inequality to~(\ref{053050}) and get  $|\frac{1}{N}\sum_{i=1}^N \mathcal{Q}_{ii}|\prec \Psi^2$, which together with~(\ref{100601}) implies 
 the first bound in~(\ref{072102}). The other three in~(\ref{072102}) can be verified analogously. Hence, we completed the proof of Lemma ~\ref{lem.071920}.
\end{proof}

\begin{proof}[Proof of Theorem~\ref{pro.053020}] 
Hence, we start with the  averaged   analogue of~(\ref{071501}), but with  
 $\mathcal{P}_{ii}$'s replaced by $\mathcal{Q}_{ii}$'s. In particular, the term $T_{ii}$ is missing. 
 Following the proof of ~(\ref{071501}) 
with these modifications, we obtain 
\begin{align}
\mathbb{E}&[\mathfrak{m}(p,p)]=\frac{1}{N}\sum_{i}d_i\mathbb{E} \Big[ \frac{1}{\|\mathbf{g}_{i}^u\|_2}\Big(\mathring{T}_{ii}- \frac{1}{N} \sum_{k}^{(i)} \frac{  \partial (\hat{\mathbf{e}}_k^*  G\hat{\mathbf{e}}_i)}{\partial g^{u}_{ik}} \Big)\tau_1(\wt{ B }G)  \mathfrak{m}(p-1,p)\Big]\nonumber\\
& \;-\frac{1}{N^2}\sum_{i}\sum_{k}^{(i)} d_i\mathbb{E} \Big[ \frac{ \partial \| \mathbf{g}^u_i\|_2^{-1}}{\partial g_{ik}^u} \hat{\mathbf{e}}_k^*\wt{ B }^{\la i\ra}  G\hat{\mathbf{e}}_i \tau_1(G) \mathfrak{m}(p-1,p)\Big]\nonumber\\
&\;-\frac{1}{N^2}\sum_i\sum_{k}^{(i)}d_i\mathbb{E}\Big[\frac{\partial \tau_1(G)}{\partial g_{ik}^u}\frac{1}{\|\mathbf{g}_i^u\|_2} \hat{\mathbf{e}}_k^*\wt{ B }^{\la i\ra}  G\hat{\mathbf{e}}_i\mathfrak{m}(p-1,p)\Big]\nonumber\\
&\; -\frac{p-1}{N^2}\sum_{i} \sum_{k}^{(i)} d_i\mathbb{E} \Big[ \frac{1}{\| \mathbf{g}^u_i\|_2} \hat{\mathbf{e}}_k^*\wt{ B }^{\la i\ra}  G\hat{\mathbf{e}}_i \tau_1(G) \Big( \frac{1}{N}\sum_j d_j\frac{\partial \mathcal{Q}_{jj}}{\partial g_{ik}^u}\Big)\mathfrak{m}(p-2,p)\Big]\nonumber\\
&\; -  \frac{p}{N^2}\sum_{i} \sum_{k}^{(i)} d_i\mathbb{E} \Big[\frac{1}{\| \mathbf{g}^u_i\|_2} \hat{\mathbf{e}}_k^*\wt{ B }^{\la i\ra}  G\hat{\mathbf{e}}_i \tau_1(G)  \Big(\frac{1}{N}\sum_j \bar{d}_j\frac{\partial \overline{\mathcal{Q}_{jj}}}{\partial g_{ik}^u}\Big)\mathfrak{m}(p-1,p-1)\Big]\nonumber\\
&\;+\frac{1}{N}\sum_{i} d_i\mathbb{E}\Big[\Big(\varepsilon_{i1} \tau_1(G) -\frac{\varepsilon_{i4}+\varepsilon_{i5}}{\|\mathbf{g}_i^u\|_2}\Big) \mathfrak{m}(p-1,p)\Big]+\mathbb{E}\big[O_\prec(\Psi^2) \mathfrak{m}(p-1,p)\big]\,.\label{083170}
\end{align}
In addition, we also have the averaged analogue of~(\ref{071502}): 
\begin{align}
&\frac{1}{N}\sum_i d_i\mathbb{E} \Big[ \frac{1}{\|\mathbf{g}_{i}^u\|_2}\Big(\mathring{T}_{ii}- \frac{1}{N} \sum_{k}^{(i)} \frac{  \partial (\hat{\mathbf{e}}_k^*  G\hat{\mathbf{e}}_i)}{\partial g^{u}_{ik}} \Big)\tau_1(\wt{ B }G)  \mathfrak{m}(p-1,p)\Big]\nonumber\\
&\quad=\frac{1}{N^2}\sum_i  \sum_{k}^{(i)} d_i\mathbb{E} \Big[\frac{\partial \|\mathbf{g}_{i}^u\|_2^{-2}}{\partial g_{ik}^u} \hat{\mathbf{e}}_k^*  G\hat{\mathbf{e}}_i\tau_1(\wt{ B }G)  \mathfrak{m}(p-1,p)\Big]\nonumber\\
&\qquad+\frac{1}{N^2}\sum_i\sum_{k}^{(i)}d_i\mathbb{E}\Big[ \frac{\partial \tau_1(\wt{ B }G)}{\partial g_{ik}^u} \frac{1}{\|\mathbf{g}_i^u\|_2^2} \hat{\mathbf{e}}_k^*  G\hat{\mathbf{e}}_i \mathfrak{m}(p-1,p)\Big]\nonumber\\
&\qquad +\frac{p-1}{N^2}\sum_i  \sum_{k}^{(i)} d_i\mathbb{E} \Big[  \frac{1}{\|\mathbf{g}_{i}^u\|_2^2} \hat{\mathbf{e}}_k^*  G\hat{\mathbf{e}}_i \tau_1(\wt{ B }G)\Big(\frac{1}{N}\sum_j d_j \frac{\partial \mathcal{Q}_{jj}}{\partial g_{ik}^u}\Big)\mathfrak{m}(p-2,p)\Big]\nonumber\\
&\qquad +\frac{p}{N^2} \sum_i\sum_{k}^{(i)} d_i\mathbb{E} \Big[  \frac{1}{\|\mathbf{g}_{i}^u\|_2^2} \hat{\mathbf{e}}_k^*  G\hat{\mathbf{e}}_i \tau_1(\wt{ B }G)\Big(\frac{1}{N}\sum_j \bar{d}_j \frac{\partial \overline{\mathcal{Q}_{jj}}}{\partial g_{ik}^u}\Big)\mathfrak{m}(p-1,p-1)\Big].\label{083171}
\end{align}

Hence, to show~(\ref{053050}), it suffices to estimate the second to the fifth terms on the right side of~(\ref{083170}), and the terms on the right side of~(\ref{083171}).  
First, we notice that 
\begin{align}
\varepsilon_{i4}=O_\prec(\Psi^2)\,, \label{112131}
\end{align}
 which can be seen from ~(\ref{091601}),~(\ref{071604}), 
 and the facts $\big|\|\mathbf{g}_i^a\|_2^2-1\big|\prec \frac{1}{\sqrt{N}}$ and $|h_{ii}^u|\prec\frac{1}{\sqrt{N}}$.
All the other  desired estimates can be derived from the following lemma.
\begin{lem} \label{lem.072610} Suppose that the assumptions in Theorem~\ref{thm.081501} hold. 
Let $\eta_{\mathrm{M}}>0$ be any  (large) constant  and $\gamma>0$  in \eqref{etamdef} be any (small) constant. Let $\hat{d}_1,\ldots, \hat{d}_N\in\mathbb{C}$ be deterministic numbers with the bound $\max_i|\hat{d}_i|\lesssim 1$ and let $\tilde{d}_1,\ldots, \tilde{d}_N \in\mathbb{C}$ be (possibly random) numbers with the bound $\max_i|\tilde{d}_i|\prec 1$ for all $i\in \llbracket 1, N\rrbracket$. 
Let $Q$ be any deterministic diagonal matrix satisfying $\|Q\|\leq C$ and $X=\hat{I}$ or~$A$, set $X_i=\hat{I}$ or $\wt{ B }^{\la i\ra}$, and let 
\[\mathbf{x}_i,\mathbf{y}_i=\binom{\mathring{\mathbf{g}}_i^u}{\mathbf{0}}\quad \text{or}\quad \binom{\mathbf{0}}{\mathring{\mathbf{g}}_i^v}.\] 
We have the estimates 
\begin{align}
&\frac{1}{N^2} \sum_{i=1}^N  \sum_{k}^{(i)} \tilde{d}_i \frac{\partial \|\mathbf{g}_i^u\|_2^{-1}}{\partial g_{ik}^u} \hat{\mathbf{e}}_k^*X_iG\hat{\mathbf{e}}_i= O_\prec(\frac{1}{N})\,,\nonumber\\
&\frac{1}{N^2}\sum_{i=1}^N\sum_k^{(i)} \tilde{d}_i \frac{\partial \ntr QXG}{\partial g_{ik}^u} \hat{\mathbf{e}}_k^* X_i G\hat{\mathbf{e}}_i=O_\prec(\Psi^4)\,, 
\label{060430}
\end{align}
uniformly on  $\mathcal{S}_{\mathcal{I}}(\eta_{\mathrm{m}}, \eta_{\mathrm{M}})$.
In addition, we also have
\begin{align}
 \frac{1}{N} \sum_{i=1}^N \hat{d}_i&\mathbb{E}\Big[\Big(  \mathbf{x}_i^* X_i\mathbf{y}_i-\mathbb{E}_i\big[\mathbf{x}_i^*X_i\mathbf{y}_i\big]\Big) \mathfrak{m}(p-1,p)\Big]\nonumber\\
 &\quad=\mathbb{E}\Big[O_\prec(\Psi^2)\mathfrak{m}(p-1,p)\Big]+\mathbb{E}\Big[O_\prec(\Psi^4)\mathfrak{m}(p-2,p)\Big]\nonumber\\ &\quad\qquad\qquad+\mathbb{E}\Big[O_\prec(\Psi^4)\mathfrak{m}(p-1,p-1)\Big]\,, \label{072620}
\end{align} 
uniformly on  $\mathcal{S}_{\mathcal{I}}(\eta_{\mathrm{m}}, \eta_{\mathrm{M}})$, 
where $\mathbb{E}_i$ denotes the expectation with respect to $\mathring{\mathbf{g}}_i^u$ and $\mathring{\mathbf{g}}_i^v$.
\end{lem}

With Lemma~\ref{lem.072610}, we can proceed to the proof of  Theorem~\ref{pro.053020} as follows. First of all,  for any diagonal matrix $Q=\text{diag} (q_{1},\ldots, q_{2N})$,  using the first estimate in~(\ref{091601}), 
we have
\begin{align*}
\ntr QG&=\frac{1}{2N} \sum_{i=1}^{N} (q_i+q_{\hat{i}})\frac{\omega_B(z)}{|\xi_i|^2-(\omega_B(z))^2}+O_\prec(\Psi)\,,\nonumber\\
\ntr QAG&=\frac{1}{2N} \sum_{i=1}^N (q_i+q_{\hat{i}})\frac{|\xi_i|^2}{|\xi_i|^2-(\omega_B(z))^2}+O_\prec(\Psi)\,.
\end{align*} 
 Using the upper bound of $\omega_B$ and the lower bound of $\Im \omega_B$ in~(\ref{110270}), we can see that 
\begin{align}
|\ntr QXG|\prec 1\,, \label{112130}
\end{align} 
for diagonal $Q$ with $\|Q\|\leq C$ and $X=\hat{I}$ or $A$. Note that all partial traces such as $\tau_1(G)$, $\tau_1(\wt{B}G)$ can be written as a linear combination of terms of the form $\ntr QXG$ with the aid of the identities in~(\ref{102702}), and thus  for
these partial traces  we have
\[ 
   \tau_1(G)= O_\prec(1)\,, \qquad\quad  \tau_1(\wt{B}G)= O_\prec(1)\,.
\]
 These bounds together with the first estimate in ~(\ref{060430}), imply  the desired estimates for the second term on the right side of~(\ref{083170}) and the first term on the right side of~(\ref{083171}).

 Next, notice that
\begin{align*}
\frac{1}{N}\sum_{j=1}^N d_j \mathcal{Q}_{jj}=\ntr(D\wt{ B }G) \tau_1(G) -\ntr (DG) \tau_1(\wt{ B }G)+ \ntr (DG) \Upsilon_{1},
\end{align*}
where we denoted  the deterministic diagonal matrix $D\deq\text{diag}\big(d_1,\ldots, d_N\big)\oplus 0$,
with $0$ the $N\times N$ zero matrix.
In addition, using~(\ref{102702}), we can see that 
$\frac{1}{N}\sum_j d_j \mathcal{Q}_{jj}$ is a polynomial of $\frac{1}{N} \sum_{j} d_j T_{jj} $ and the terms of the form $\ntr QXG$ for 
some diagonal $Q$ with $\|Q\|\leq C$ and $X=\hat{I}$~or $A$. Here we also used the fact that $\tau_a(\mathcal{D})=\ntr (\hat{I}_a\mathcal{D})$ for any $\mathcal{D}\in M_{2N}(\mathbb{C})$ and $a=1,2$, where $\hat{I}_a$ is defined in~(\ref{071540}). Then the last two estimates in~(\ref{060430}),~(\ref{112130}),   together with the chain rule,
imply  that 
\begin{align}
& \frac{1}{N^3} \sum_{i=1}^{N}  \sum_{k}^{(i)} \tilde{d}_i \hat{\mathbf{e}}_k^*  X_i G\hat{\mathbf{e}}_i \sum_{j=1}^N d_j \frac{\partial \mathcal{Q}_{jj}}{ \partial g_{ik}^u}=O_\prec(\Psi^4)\,. \label{060431}
\end{align}

Similarly, we can prove the same bound if we replace $\mathcal{Q}_{jj}$'s by $\overline{\mathcal{Q}}_{jj}$'s. Hence, the desired estimates for the third to the fifth terms on the right side of~(\ref{083170}), and the last three terms on the right side of~(\ref{083171}) can be obtained from the second estimate in~(\ref{060430}). 

 Hence, what remains is to estimate the sixth term in~(\ref{083170}). First, according to~(\ref{112131}), we can neglect $\varepsilon_{i4}$. Then we recall the definition of $\varepsilon_{i1}$ from~(\ref{083011}).  Using the estimates of $G_{\hat{i}i}$ and $T_{\hat{i}i}$ from the first and the third inequalities in~(\ref{091601}), and the estimates
 \[\ell_i^v=1+O_\prec(\frac{1}{\sqrt{N}})\,,\qquad |h_{ii}^u|\prec \frac{1}{\sqrt{N}}\,,\qquad |(\mathbf{k}_i^u)^*\wt{B}^{\la i\ra}\mathbf{k}_i^v|\prec \frac{1}{\sqrt{N}}\,,\]
we see that 
\begin{align}
\varepsilon_{i1}=\frac{\bar{\xi}_i}{|\xi_i|^2-\omega_B^2} (\mathbf{k}_i^u)^*\wt{B}^{\la i\ra}\mathbf{k}_i^v+O_\prec(\Psi^2)=\frac{\bar{\xi}_i}{|\xi_i|^2-\omega_B^2} (\mathbf{\ell}_i^u)^*\wt{B}^{\la i\ra}\mathbf{\ell}_i^v+O_\prec(\Psi^2)\,, \label{083180}
\end{align}
where we introduced the notations
\begin{align*}
\mathbf{\ell}_i^u\deq\binom{\mathring{\mathbf{g}}_i^u}{\mathbf{0}},\qquad\mathbf{\ell}_i^v\deq \binom{\mathbf{0}}{\mathring{\mathbf{g}}_i^v}\,.
\end{align*}
Then, recall the definition of $\varepsilon_{i5}$ from~(\ref{083190}). Applying the estimate of $G_{ii}$ from the first inequality in~(\ref{091601}), and the second formula in~(\ref{112133}), and the fact $\|\mathbf{g}_i^u\|_2^2=\|\mathbf{\ell}_i^u\|_2^2+O_\prec(\frac{1}{N})=1+O_\prec(\frac{1}{\sqrt{N}})$,   we also have 
\begin{align}
\varepsilon_{i5}= \frac{(z-\omega_B)m_{\mu_A\boxplus \mu_B}\omega_B}{|\xi_i|^2-\omega_B^2}\big(\|\mathbf{\ell}_i^u\|_2^2-1\big)+O_\prec(\Psi^2)\,. \label{083195}
\end{align}

Note that both of the first terms on the right side of~(\ref{083180}) and~(\ref{083195}) are of the form 
$\hat{d}_i(\mathbf{x}_i^*X_i\mathbf{y}_i-\mathbb{E}_i[\mathbf{x}_i^*X_i\mathbf{y}_i])$ for some deterministic $\hat{d}_i$ with  $|\hat{d}_i|\lesssim 1$. Hence, using~(\ref{072620}), we get the desired bound for the sixth term of~(\ref{083170}). This completes the proof of Theorem~\ref{pro.053020} up to the proof of Lemma~\ref{lem.072610}. 
\end{proof}

\begin{proof}[Proof of Lemma~\ref{lem.072610}] The first estimate in ~(\ref{060430}) follows directly from the first estimate in~(\ref{071520}). The second estimate of~(\ref{060430}) is a weighted average of the last estimate in~(\ref{071520}). 

Hence, what remains is to prove~(\ref{072620}). We only show the details for the case $\mathbf{x}_i=\mathbf{\ell}_i^u$ and $\mathbf{y}_i=\mathbf{\ell}_i^v$. The others are similar. Notice that in this case, $\mathbb{E}_i[\mathbf{x}_i^* X_i \mathbf{y}_i]=0$.
Using the integration by parts formula~(\ref{integration by parts formula}) again, we have 
\begin{align*}
\frac{1}{N}\sum_i \hat{d}_i&\mathbb{E}\Big[(\mathbf{\ell}_i^u)^* X_i\mathbf{\ell}_i^v\mathfrak{m}(p-1,p)\Big]=\frac{1}{N}\sum_i\sum_{k}^{(i)} \hat{d}_i \mathbb{E}\Big[\bar{g}_{ik}^u \hat{\mathbf{e}}_k^* X_i\mathbf{\ell}_i^v \mathfrak{m}(p-1,p)\Big]\nonumber\\
&=\frac{p-1}{N^2} \sum_i\sum_k^{(i)} \hat{d}_i \mathbb{E}\Big[ \hat{\mathbf{e}}_k^* X_i\mathbf{\ell}_i^v \frac{1}{N}\sum_j d_j\frac{\partial \mathcal{Q}_{jj}}{\partial g_{ik}^u}\mathfrak{m}(p-2,p)\Big]\nonumber\\
&\qquad + \frac{p}{N^2} \sum_i\sum_k^{(i)} \hat{d}_i \mathbb{E}\Big[ \hat{\mathbf{e}}_k^* X_i\mathbf{\ell}_i^v \frac{1}{N}\sum_j \bar{d}_j\frac{\partial \overline{\mathcal{Q}}_{jj}}{\partial g_{ik}^u}\mathfrak{m}(p-1,p-1)\Big]\,.
\end{align*}
Hence, it suffices to show 
\begin{align}
\Big|\frac{1}{N^3} \sum_i\sum_k^{(i)} \hat{d}_i \hat{\mathbf{e}}_k^* X_i\mathbf{\ell}_i^v \sum_j d_j\frac{\partial \mathcal{Q}_{jj}}{\partial g_{ik}^u}\Big|\prec \Psi^4 \label{090110}
\end{align}
and its complex conjugate analogue. The proof of~(\ref{090110}) is nearly the same as that of~(\ref{060431}). Hence, we omit it. 
Therefore, we completed the proof of Lemma~\ref{lem.072610}.
\end{proof}

\section{Proof of Theorem~\ref{le corollary of 0 in the bulk}}\label{le section of 0 in the bulk}

Theorem~\ref{le corollary of 0 in the bulk} will directly follow from a more detailed result, Theorem~\ref{0 in the bulk}, below. Recall the definitions of $s_+<\infty$ from~\eqref{le s-s+} and of $r_\pm$ from~\eqref{le radii}. Recall further that we assumed $\mathrm{supp}\,\mu_1\subset[-s_+,s_+]$.

Given $r\in(r_-,r_+)$, we define
\begin{align}\label{le sigmas}
 \sigma_-\equiv\sigma_-(r)\deq\left(\frac{r^2-r_-^2}{r_+^2-r_-^2}\right)^{\frac12}\,,\qquad \sigma_+\equiv\sigma_+(r)\deq\left(\frac{r_+^2}{r_+^2-r^2}\right)^{\frac12}\,,
\end{align}
and we note that $\sigma_-(r)\in(0,1)$ and $\sigma_+(r)\in(1,\infty)$. 

The main result of this section is the following bound on $\omega_2$ along the imaginary axis.
\begin{thm}(Bounds on $\omega_2$) \label{0 in the bulk} 
Let $r\in(r_-,r_+)$. Then there are strictly positive constants $b_-\equiv b_-(\mu_1,r)>0$ and $s_-\equiv s_-(\mu_1,r)>0$ such that
\begin{align}\label{le bounds on im omega2}
 C_2^{-1}\sigma_-s_-b_-\min\Big\{1, \frac{\sigma_-s_-}{\eta}\Big\}\le|\omega_2(\ii\eta)-\ii\eta|\le C_2\min\Big\{\sigma_+s_+,\frac{r^2}{\eta}\Big\}\,,
\end{align}
for all $\eta\ge 0$, for a numerical constant $C_2>1$ (independent of $\mu_1$ and $r$). 
\end{thm}
\begin{rem}\label{le remark uniformity of constants}
It follows from the proof of Theorem~\ref{0 in the bulk}, that $s_-(\mu_1,r)$ is a monotonic increasing function in the variable $r\in(r_-,r_+)$, with $0<s_-(\mu_1,r)\le s_+$. The $r$ dependence of $b_-(\mu_1,r)$ (and $a_-(\mu_1,r)$ in~\eqref{le a--}) is then explicit in terms of $s_-(\mu_1,r)$ and $r$. This allows us to obtain uniform bounds for $r\in [r_-+\tau,r_+-\tau]$, with a fixed $\tau>0$, in Theorem~\ref{le corollary of 0 in the bulk}.
\end{rem}

With this proposition we can easily establish all necessary bounds on the free convolution measure and both associated subordination 
functions stated in Theorem~\ref{le corollary of 0 in the bulk}.
\begin{proof}[Proof of Theorem~\ref{le corollary of 0 in the bulk}] Using~\eqref{le omega1 after knowing omega2} and the facts $\omega_1(\ii\eta)=-\overline{\omega_1(\ii\eta)}$ and $\omega_2(\ii\eta)=-\overline{\omega_2(\ii\eta)}$, Theorem~\ref{le corollary of 0 in the bulk} follows readily.
Indeed, the upper bound on $|\omega_2(\ii\eta)|$ 
follows from the upper  bound in \eqref{le bounds on im omega2},
the lower bound on $\im \omega_2(\ii\eta)$ follows from  the lower bound  in \eqref{le bounds on im omega2}
for small~$\eta$ and from~$\im \omega_2(\ii \eta)\ge \eta$, for large $\eta$. The upper and lower bounds on $\im\omega_2(\ii\eta)-\ii \eta$ then
imply a lower and upper bound on $|\omega_1(\ii\eta)|$ by~\eqref{le omega1 after knowing omega2}. Finally,~\eqref{090410}
controls $F_{\mu_1}(\omega_2(\ii \eta))= -1/m_{\mu_1\boxplus\mu_2}(\ii \eta)$ from above
and $\im F_{\mu_1}(\omega_2(\ii\eta)) \ge \im (\omega_2(\ii\eta) -\ii\eta)$ controls it from below by~\eqref{le bounds on im omega2},
which yields \eqref{mbound}. 
\end{proof}

\subsection{Proof of Theorem~\ref{0 in the bulk}}
For the sake of simplicity of presentation, the proof of~Theorem~\ref{0 in the bulk} is accomplished in a sequence of lemmas.

\begin{lem} Let $\mu_1$ be as in Theorem~\ref{0 in the bulk}. Then there exists a symmetric (non-negative) Borel measure $\widetilde\mu_1$ such that
\begin{align}\label{le super nevan}
 F_{\mu_1}(\omega)-\omega=\frac{-r_-^2}{\omega}+\int_\R\frac{\dd\widetilde\mu_1(x)}{x-\omega}\,,\qquad\qquad\omega\in\C^+\,,
\end{align}
with $\widetilde\mu_1(\R)=r_+^2-r_-^2$, $\mathrm{supp}\,\widetilde\mu_1\subset[-s_+,s_+]$ and $\widetilde\mu_1(\{0\})=0$.
\end{lem}
\begin{proof}
Since $\mu_1$ is a symmetric probability measure, $F_{\mu_1}\,:\C^+\rightarrow\C^+$ satisfies~\eqref{le F behaviour at infinity} and there exists a symmetric Borel measure, $\widehat\mu_1$, such that $F_{\mu_1}$ admits the Nevanlinna representation
\begin{align}\label{le symmetric nevan}
 F_{\mu_1}(\omega)=\omega+\int_\R\frac{\dd\widehat\mu_1(x)}{x-\omega}\,,\qquad\qquad\omega\in\C^+;
\end{align}
see \eg Proposition~2.2 in~\cite{Maassen}. We observe that $\widehat\mu_1(\R)=r_+^2$. Indeed, expanding the Stieltjes transform $m_{\mu_1}$ around complex infinity we find
\begin{align*}
 m_{\mu_1}(\omega)=\int_\R\frac{\dd\mu_1(x)}{x-\omega}&=\int_\R\left(\frac{-1}{\omega}-\frac{x^2}{\omega^3}+O(|\omega|^{-5})\right)\dd\mu_1(x)\nonumber\\&=-\frac{1}{\omega}-\frac{r_+^2}{\omega^3}+O(|\omega|^{-5})\,,
\end{align*}
as $|\omega|\to\infty$ in $\C^+$, where we used that $\mu_1$ is symmetric. Thus $F_{\mu_1}(\omega)-\omega=-r_+^2/\omega+O(|\omega|^{-3})$ in the same limit and we conclude by comparing with~\eqref{le symmetric nevan} that $\widehat\mu_1(\R)=r_+^2$.

Since $\mu_1$ is symmetric, its Stieltjes transform satisfies $m_{\mu_1}(\ii\eta)=\ii \,\im m_{\mu_1}(\ii\eta)$, $\eta>0$.
We then obtain
\begin{align}
 \lim_{\eta\searrow 0}\ii\eta (F_{\mu_1}(\ii\eta)-\ii\eta)&=-\lim_{\eta\searrow 0}\frac{\eta}{\im m_{\mu_1}(\ii\eta)}\nonumber\\ &=-\lim_{\eta\searrow 0}\left({\int_\R\frac{\dd\mu_1(x)}{x^2+\eta^2}}\right)^{-1}=-r_-^2\,,
\end{align}
by the definition of $r_-$ in~\eqref{le radii}. Comparing with~\eqref{le symmetric nevan}, we conclude that $\widehat\mu_1(\{0\})=r_-^2$, since for any Borel measure $\nu$ we have $\lim_{\eta\searrow 0} \eta\im m_{\nu}(E+\ii\eta)=\nu(\{E\})$, for all $E\in\R$. Setting $\widetilde\mu_1\deq \widehat\mu_1-r_-^2\delta_0$ we get~\eqref{le super nevan}. Clearly, $\widetilde\mu_1$ is a symmetric (non-negative) Borel measure with $\widetilde\mu_1(\R)=r_+^2-r_-^2$ satisfying $\mathrm{supp}\,\widetilde\mu_1\subset[-s_+,s_+]$. This concludes the proof of the lemma.
\end{proof}

We now introduce $s_-\in\R^+$ as
\begin{align}\label{le s--}
 s_-\deq \sup\left\{x\in\R^+\,:\,\int_0^{x}\dd\widetilde\mu_1(x)\le\frac{r^2-r_-^2}{8}\right\}\,.
\end{align}
Note that, for $r>r_-$, $s_-$ is strictly positive since $\mu_1$ is symmetric and we assume that $\mu_1$ is supported at least at three points. (We assume that $\mu_\sigma$ is supported at least at two points. Thus $\mu_1=\mu_\sigma^\mathrm{sym}$ is supported at least at three points). Note that since $\mathrm{supp}\,\widetilde\mu_1\subset[-s_+,s_+]$, 
 thus $\widetilde\mu_1([0, s_+])= \widetilde\mu_1(\R^+)=\frac{1}{2}(r_+^2-r_-^2)> \frac{1}{8}(r^2-r_-^2)$, we have $s_-\le s_+$.

Equation~\eqref{090410} for $\omega_2(z)$, when combined with~\eqref{le super nevan}, reads
\begin{align}
  F_{\mu_1}(\omega_2(z))-\omega_2(z)=\frac{-r_-^2}{\omega_2(z)}+\int_\R \frac{\dd\widetilde\mu_1(x)}{x-\omega_2(z)}=-z-\frac{r^2}{\omega_2(z)-z}\,,
\end{align}
$z\in\C^+$. We then rewrite this last equation as
\begin{multline}\label{le perturbed}
 (r^2-r_-^2)\omega_2(z)+r_-^2z\\ =-z(\omega_2(z)-z)\omega_2(z)-(\omega_2(z)-z)\omega_2(z)\int_\R\frac{\dd\widetilde\mu_1(x)}{x-\omega_2(z)}\,,
\end{multline}
$z\in\C^+$. Our first goal is to show that $\im\omega_2(0)\equiv \lim_{\eta\searrow 0}\im\omega_2(\ii\eta)$ is strictly positive.

\begin{lem}
 Let $\mu_1$ and $r\in(r_-,r_+)$ be as in Theorem~\ref{0 in the bulk}. Then,
\begin{equation}\label{omega2 at zero}
\im\omega_2(0)>\frac{\sqrt3}{2}\sigma_-s_-\,.
\end{equation} 
\end{lem}
\begin{proof}
By Theorem 2.3 of~\cite{Bel1}, $\omega_2(z)$ extends continuously to the real line. Choosing $z=\ii\eta$ in~\eqref{le perturbed} we can assume that $\lim_{\eta\searrow 0}\im\omega_2(\ii\eta)<\infty$. 
By symmetry, $\omega_2(\ii\eta)=-\overline{\omega_2(\ii\eta)}$, we know that
 $\omega_2(0)$ is purely imaginary. Assume first that $\im \omega_2(0)>0$. Taking the limit $\eta\searrow 0$ in~\eqref{le perturbed} and dividing through $\omega_2(0)$ we get
\begin{align}\label{le new thing}
(r^2-r_-^2)=-\omega_2(0)\int_\R\frac{\dd\widetilde\mu_1(x)}{x-\omega_2(0)}=\int_\R\frac{|\omega_2(0)|^2\dd\widetilde\mu_1(x)}{x^2+|\omega_2(0)|^2}\,,
\end{align}
where we used that $\omega_2(0)$ is purely imaginary. Recalling $s_-$ in~\eqref{le s--}, we further~get
\begin{align}\label{le vega 3}
 \int_\R\frac{|\omega_2(0)|^2\dd\widetilde\mu_1(x)}{x^2+|\omega_2(0)|^2}&\le 2\int_0^{s_-}\frac{|\omega_2(0)|^2\dd\widetilde\mu_1(x)}{x^2+|\omega_2(0)|^2}+2\int_{s_-}^{s_+}\frac{|\omega_2(0)|^2\dd\widetilde\mu_1(x)}{x^2+|\omega_2(0)|^2}\nonumber\\
 &\le  2\int_0^{s_-}\dd\widetilde\mu_1(x)+2|\omega_2(0)|^2\int_{s_-}^{s_+}\frac{\dd\widetilde\mu_1(x)}{x^2}\nonumber\\
 &\le \frac{r^2-r_-^2}{4}+|\omega_2(0)|^2\frac{r_+^2-r_-^2}{s_-^2}\,,
\end{align}
where we also used that $\widetilde\mu_1(\R)=r_+^2-r_-^2$.  Hence from~\eqref{le vega 3} and~\eqref{le new thing}, we conclude that
\begin{align}
 3\frac{r^2-r_-^2}{4}\le |\omega_2(0)|^2\frac{r_+^2-r_-^2}{s_-^2}\,.
\end{align}
Thus, we get
\begin{align}
\frac{\sqrt3}{2} {\sigma_- s_-}\le \im \omega_2(0)\,, 
\end{align}
provided that $\im\omega_2(0)>0$, where we used that $|\omega_2(0)|=\im \omega_2(0)$.

To conclude the proof, we need to show that $\lim_{\eta\searrow 0}\im\omega_2(\ii\eta)>0$. Arguing by contradiction, we assume that $\lim_{\eta\searrow 0}\im\omega_2(\ii\eta)=0$. Choose an arbitrary $\epsilon>0$. Letting $\eta>0$ be sufficiently small, we can assure that
\begin{align}\label{le chliner als epsilon}
 \left|\omega_2(\ii\eta)\int_\R\frac{\dd\widetilde\mu_1(x)}{x-\omega_2(\ii\eta)}\right|=\left| \int_\R\frac{|\omega_2(\ii\eta)|^2\dd\widetilde\mu_1(x)}{x^2+|\omega_2(\ii\eta)|^2}\right|\le \epsilon\,,
\end{align}
where we first used that $\omega_2(\ii\eta)$ is purely imaginary and then used that $0$ is not an atom of the measure $\widetilde\mu_1$. We thus obtain from~\eqref{le perturbed} and~\eqref{le chliner als epsilon} that
\begin{align*}
|(r^2-r_-^2)\omega_2(\ii\eta) |\le |(r^2-r_-^2)\omega_2(\ii\eta)+r_-^2\ii\eta| \le  \eta |\omega_2(\ii\eta)|^2+|\omega_2(\ii\eta)| \epsilon\,,
\end{align*}
for $\eta>0$ sufficiently small, where we used $r>r_-$, $|\omega_2(\ii\eta)|=\im\omega_2(\ii\eta)$, $\im\omega_2(\ii\eta)\ge \eta$ (\cf Theorem~\ref{le prop 1}), so $|\omega_2(\ii \eta)-\ii\eta|\le |\omega_2(\ii \eta)|$. Choosing $\epsilon=(r^2-r_-^2)/2$, we get
\begin{align}
 |(r^2-r_-^2)\omega_2(\ii\eta) |\le 2  \eta |\omega_2(\ii\eta)|^2\,,
\end{align}
for $\eta>0$ sufficiently small, \ie we have $\im\omega_2(\ii\eta)\ge (r^2-r_-^2)/(2\eta)$. Since $r^2-r_-^2>0$, we get a contraction with the assumption that $\lim_{\eta\searrow 0}\omega_2(\ii\eta)=0$. We thus conclude that $\lim_{\eta\searrow 0}\omega_2(\ii\eta)>0$. This completes the proof of the lemma.
 \end{proof}

 We are now prepared to prove the lower bound in~\eqref{le bounds on im omega2}. Recall $s_->0$ from~\eqref{le s--}.
 \begin{lem}
 Let $\mu_1$ and $r\in(r_-,r_+)$ be as in Theorem~\ref{0 in the bulk}. Then, there is a strictly positive constant $b_-\equiv b_-(\mu_1,r)>0$ such that
 \begin{align}\label{le bounds on im omega22}
|\omega_2(\ii\eta)-\ii\eta|\ge  C^{-1}\sigma_-s_- b_-\min\Big\{1, \frac{\sigma_-s_-}{\eta}\Big\}\,,
\end{align}
for all $\eta\ge 0$, where $C>1$ is a numerical constant (independent of $\mu_1$ and~$r$).
\end{lem}
\begin{proof}
Using the definition of $s_-$ in~\eqref{le s--}, we write~\eqref{le perturbed}, the defining equation for $\omega_2(z)$, as 
an equation with a free variable $\omega$:
\begin{align}\label{le alig}
 \frac{r^2-r_-^2}{\omega}+\int_{|x|\le s_-}\frac{\dd\widetilde\mu_1(x)}{(x-\omega)}+\int_{|x|>s_-}\frac{\dd\widetilde\mu_1(x)}{(x-\omega)}=-z-\frac{r^2z}{(\omega-z)\omega}\,,
\end{align}
$z\in\C^+$, whose unique solution on the upper half plane gives $\omega=\omega_2(z)$. Note that the third term on the left side has the expansion
\begin{align}
 \int_{|x|>s_-}\frac{\dd\widetilde\mu_1(x)}{(x-\omega)}&=\int_{|x|>s_-}\frac{\dd\widetilde\mu_1(x)}{x^2}\omega+\int_{|x|>s_-}\frac{\dd\widetilde\mu_1(x)}{x^3(x-\omega)}\omega^3\,,
\end{align}
for $|\omega|<s_-$, where we used that $\widetilde\mu_1$ is symmetric to get the second line. Let
\begin{align}\label{le a--}
 a_-\equiv a_-(\mu_1,r)\deq\int_{|x|>s_-}\frac{\dd\widetilde\mu_1(x)}{x^2}\,.
\end{align}
Note that by the definition of $s_-$ in~\eqref{le s--} we have the bound
\begin{align}\label{le hatomega at zero}
 0< \frac{3}{4}\frac{r^2-r_-^2}{s_+^2}\le a_-\le\frac{r_+^2-r_-^2}{s_-^2}\,.
\end{align}
Let moreover
\begin{align}\label{le def of hatomega}
 \widehat\omega\deq\ii\left(\frac{r^2-r_-^2}{a_-}\right)^{\frac{1}{2}}\,.
\end{align}
Note that from~\eqref{le hatomega at zero}, we have
\begin{align}\label{le size of omega hat}
 \sigma_-s_-\le |\widehat\omega|\,.
\end{align}

Using the definitions of~$\widehat\omega$ in~\eqref{le def of hatomega}
and of~$a_-$ in~\eqref{le a--}, we rewrite~\eqref{le alig} as
\begin{align}\label{final equation}
 (-\widehat\omega^2+\omega^2)(\omega-z)=-\frac{r^2z}{a_-}+\psi(\omega,z)\,,\qquad\qquad z\in\C^+\,,
\end{align}
where we further introduced the shorthand notation
\begin{multline}\label{le psi mm}
 \psi(\omega,z)\deq  - \frac{\omega(\omega-z)z}{a_-}-\frac{\omega^4(\omega-z)}{a_-}\int_{|x|> s_-}\frac{\dd\widetilde\mu_1(x)}{x^3(x-\omega)}\\-\frac{\omega(\omega-z)}{a_-}\int_{|x|\le s_-}\frac{\dd\widetilde\mu_1(x)}{x-\omega}\,.
\end{multline}

Next, we abbreviate $t_-\deq s_-\sigma_-<s_-$ and  define 
\begin{align}\label{defb}
 b_-\equiv b_-(\mu_1,r)\deq \min\Big\{1,a_-,a_-\frac{t_-^2}{r^2}\Big\}>0\,.
\end{align}
Then we introduce the sets
\begin{align}
 {\mathcal{F}}_-\deq\Big\{\omega=\ii|\omega|\in\C^+\,:\,|\omega|^2\le\frac{7}{8}t_-^2\Big\}
\end{align}
\begin{align}\begin{split}
 {\mathcal{E}}_-^{(1)}\deq \Big\{z=\ii\eta\in\C^+\,:\,\eta\le\frac{t_-b_-}{ 64} \Big\}\,.
\end{split}\end{align}

For $\omega\in{\mathcal{F}}_-$ we bound the last term in
the definition of $\psi(\omega,z)$ in~\eqref{le psi mm}~as 
\begin{align}
 \frac{|\omega| |\omega-z|}{a_-}\left|\int_{|x|\le s_-}\frac{\dd\widetilde\mu_1(x)}{x-\omega}\right|&= \frac{|\omega-z|}{a_-}\left|\int_{|x|\le s_-}\frac{|\omega|^2\dd\widetilde\mu_1(x)}{x^2+|\omega|^2}\right|\nonumber\\
 &\le\frac{|\omega-z|}{4a_-}(r^2-r_-^2)=\frac{|\widehat\omega|^2}{4}\,|\omega-z|\,,\qquad z\in\C^+\,,
\end{align}
where we used that $\widetilde\mu_1$ is symmetric and the definitions of~$s_-$ in~\eqref{le s--} and of $\widehat\omega$ in~\eqref{le def of hatomega}.

For $\omega\in {\mathcal{F}}_-$ we bound the second but last term in the definition of $\psi(\omega,z)$~as 
\begin{align}
\frac{|\omega|^4 |\omega-z|}{a_-}\bigg|\int_{|x|>s_-}\frac{\dd\widetilde\mu_1(x)}{x^3(x-\omega)}\bigg|&\le\frac{|\omega|^4|\omega-z| }{a_-}\int_{|x|>s_-}\frac{\dd\widetilde\mu_1(x)}{x^2}\frac{1}{|s_-|^2}\le\frac{7}{8}|\omega|^2\,|\omega-z|\,,
\end{align}
where we used that $|x-\omega|\ge |x|\ge s_-$, as $\omega$ lies on the imaginary axis, the definition of $a_-$ in~\eqref{le a--} and $t_-=\sigma_-s_-\le s_- $.

For the first term on the right side of~$\psi(\omega,z)$, we get for $z\in{\mathcal{E}}_-^{(1)}$ the bound
\begin{align}
 \frac{|\omega(\omega-z)z|}{a_-}\le \frac{|\omega-z| |\omega| t_-}{{ 64}}\,,
\end{align}
where we used that $64|z|\le t_-b_-\le  t_-a_-$ on $\mathcal{E}_-^{(1)}$.

Combining these estimates, we get that, for $\omega\in {\mathcal{F}}_-$ and $z\in{\mathcal{E}}_-^{(1)}$,
\begin{align}
 |\psi(\omega,z)|\le &   \frac{ |\omega| t_-}{ 64}|\omega-z|+\frac{7}{8}|\omega|^2|\omega-z|+\frac{|\widehat\omega|^2}{4}\,|\omega-z|\,.
\end{align}

For the first term on the right side of~\eqref{final equation} we note for $z\in{\mathcal{E}}_-^{(1)}$ the bound
\begin{align}
 \frac{r^2|z|}{a_-}\le \frac{t_-r^2}{{ 64} a_-}b_-\le \frac{t_-^3}{ 64}\,,
\end{align}
where we used that $b_-r^2/a_-\le t_-^2$ on $\mathcal{E}_-^{(1)}$ as follows from~\eqref{defb}.

Thus, for $\omega$ a solution to~\eqref{final equation} in $\mathcal{F}_{-}$ with $z\in{\mathcal{E}}_-^{(1)}$, at least one of the following holds
\begin{align}\label{le first dad}
 |-\widehat\omega^2+\omega^2|\,|\omega-z|\le  \frac{t_-^3}{{ 32}} 
\end{align}
or
\begin{align}\label{le second dad}
 |-\widehat\omega^2+\omega^2|\,|\omega-z|\le\frac{|\omega| t_-}{ 32}|\omega-z| +\frac{7}{4}|\omega|^2|\omega-z|+\frac{|\widehat\omega|^2}{2}\,|\omega-z|\,.
\end{align}

First, assume that~\eqref{le second dad} holds. Then we either have $\omega-z=0$, or  
\begin{align}
 |-\widehat\omega^2+\omega^2|&\le \frac{ t_-^2}{ 64}+\frac{|\omega|^2}{ 64}+\frac{7}{4}|\omega|^2+\frac{|\widehat\omega|^2}{2}\,.
\end{align}
We then absorb the last term on the right side into the left side to get
\begin{align}
\frac{1}{2}|\widehat\omega^2|\le \frac{ t_-^2}{ 64}+{|\omega|^2}+\frac{|\omega|^2}{64}+\frac{7}{4}|\omega|^2\,.
\end{align} We thus find
\begin{align}
 |\widehat\omega^2|<\frac{ t_-^2}{ 32}+{6}{|\omega|^2}\,.
\end{align}
Since $|\widehat\omega|\ge t_-$ by~\eqref{le size of omega hat}, we thus obtain that in this case that either $\omega-z=0$ or
\begin{align}
 |\omega|> {\frac38}t_-\,.
\end{align}

Second, assume that~\eqref{le first dad} holds. Then we can estimate, using that $\omega\in\mathcal{F}_-$,
\begin{align}
 |\omega-z|\le \frac{t_-^3}{ 32}\frac{1}{ |-\widehat\omega^2+\omega^2|}\le { \frac{2}{8}}t_-\,,
\end{align}
where we used that $|\widehat\omega|\ge t_-$ and $|\omega|^2\le 7t_-^2/8$ on $\mathcal{F}_-$. Since $|z|\le t_-/64$ for $z\in\mathcal{E}_-^{(1)}$, we find $|\omega|<{{ 5}t_-}/{ 16}$ in this case.

We conclude that for any $z\in\mathcal{E}_-^{(1)}$ a solution $\omega(z)$ to~\eqref{final equation} in $\mathcal{F}_-$  satisfies either
\begin{align}\label{le amad}
 |\omega(z)|< {\frac{5}{16}}t_-\qquad \textrm{ or }\qquad |\omega(z)|> { \frac{6}{16}}t_-\,.
\end{align}
Also note that if a solution $\omega(z)$ of~\eqref{final equation} satisfies $\omega(z)\not\in\mathcal{F}_-$ for some $z\in\mathcal{E}_-^{(1)}$, then the second alternative in~\eqref{le amad} holds trivially. 

Now, since the subordination function $\eta\rightarrow\mapsto\omega_2(\ii\eta)$ (extends to) a continuous function on $[0,\infty)$ by Theorem~2.3 of~\cite{Bel1}, we can conclude from~\eqref{le amad} that
 \begin{align}\label{le first lower bound on im omega}
|\omega_2(z)|>{ \frac{3}{8}} t_-\,,\qquad\qquad z\in\mathcal{E}_-^{(1)}\,,
 \end{align}
since we already showed in~\eqref{omega2 at zero} that $|\omega_2(0)|\ge { \sqrt{3}t_-/2}$. This proves the lower bound  in~\eqref{le bounds on im omega22} for the small $\eta$ regime.

Next, we introduce the domain which will handle the regime complementary to ${\mathcal{E}}_-^{(1)}$,
\begin{align}
 {\mathcal{E}}_-^{(2)}\deq\Big\{z=\ii\eta\in\C^+\,:\,\eta \ge\frac{t_-b_-}{{ 64}} \Big \}\,.
\end{align}
We claim that $\eta\mapsto \eta\cdot(\im \omega_2(\ii\eta)-\eta)$ is a monotone increasing function for $\eta\in\R^+$. Indeed since the analytic function $\omega_2\,:\,\C^+\rightarrow\C^+$ satisfies~\eqref{le limit of omega} and $\omega_2(\ii\eta)=-\overline{\omega_2(\ii\eta)}$, it has the Nevanlinna representation
\begin{align}\label{le nevi for omega2}
 \omega_2(z)=z+\int_\R\frac{\dd\nu_2(x)}{x-z}\,, \qquad\qquad z\in \C^+\,,
\end{align}
where $\nu_2$ is a finite symmetric Borel measure. 
The claim follows directly by considering the imaginary part of~\eqref{le nevi for omega2}
for $z$ along the positive imaginary axis. Hence, for any $\eta\ge \eta_0>0$, 
\begin{align}
 \im\omega_2(\ii\eta)-\eta\ge\frac{\eta_0}{\eta}(\im\omega_2(\ii\eta_0)-\eta_0)\,.
\end{align}
Choosing $\eta_0={t_-b_-}/{ 64}$ on the boundary of $ {\mathcal{E}}_-^{(1)}$, we can apply~\eqref{le first lower bound on im omega} for $z_0=\ii\eta_0$, and we obtain the estimate
  \begin{align}\label{le second lower bound on im omega}
 \im\omega_2(\ii\eta)-\eta\ge \frac{\eta_0}{\eta}\Big({\frac{3t_-}{8}}-\frac{t_-b_-}{64}\Big)\ge\frac{2t_-\eta_0}{8\eta} \ge \frac{{ 2} t_-^2b_-}{{ 256}\eta}\,,\quad \ii\eta\in {\mathcal{E}}_-^{(2)}\,,
\end{align}
where we used the definition of $b_-$ in~\eqref{defb} to get the second inequality.
Combining~\eqref{le first lower bound on im omega} and~\eqref{le second lower bound on im omega} we get the bound~\eqref{le bounds on im omega22}. 
\end{proof}

We move on to the upper bound in~\eqref{le bounds on im omega2}.  
\begin{lem}
Let $\mu_1$ and $r\in(r_-,r_+)$ be as in Theorem~\ref{0 in the bulk}. Then,
\begin{align}\label{le bounds on im omega23}
 |\omega_2(\ii\eta)-\ii\eta|\le C\min\Big\{\sigma_+s_+,\frac{r^2}{\eta}\Big\}\,,
\end{align}
for all $\eta\ge 0$, for a numerical constant $C<\infty$ (independent of $\mu_1$ and $r$).
\end{lem}
\begin{proof}

Using~\eqref{le symmetric nevan} we write
\begin{align}\label{le super nevan above}
 F_{\mu_1}(\omega)-\omega=-\frac{r_+^2}{\omega}-\frac{1}{\omega^2}\int_\R\frac{x^2\dd\widehat\mu_1(x)}{x-\omega}\,,\qquad\qquad \omega\in\C^+\,.
\end{align}
For $z\in\C^+$, we write~\eqref{le perturbed} with $\omega_2(z)$ replaced by the free variable~$\omega$~as 
\begin{align}\label{le some equation}
-\frac{r_+^2}{\omega}-\frac{\chi(\omega)}{\omega^2}=-z-\frac{r^2}{\omega-z}\,,
\end{align}
where we introduced the short hand notation
\begin{align*}
\chi(\omega)\deq\int_\R\frac{x^2\dd\widehat\mu_1(x)}{x-\omega}\,.
\end{align*}
From~\eqref{le some equation}, we find that
\begin{align}\label{le better way to do it}
 \omega_2(z)&=\frac{r_+^2-r^2+z^2}{2z}\Big( 1-\Big(1-\frac{4z^2r_+^2-4z\chi{(\omega_2(z))}\frac{\omega_2(z)-z}{\omega_2(z)}}{(r_+^2-r^2+z^2)^2} \Big)^{\frac12}\Big)\,,
 \end{align}
where we choose the square root such that $\im \omega_2(z)\ge \im z$.

Abbreviate $t_+\deq \sigma_+s_+$ and partition the positive imaginary axis by introducing 
 \begin{align}\begin{split}
  {\mathcal{E}}_+^{(1)}&\deq \big\{z=\ii\eta\in\C^+\,:\,0<\eta\le\frac{1}{4} \frac{r_+^2-r^2}{t_+}\big\}\,,\\
  {\mathcal{E}}_+^{(2)}&\deq\big\{z=\ii\eta\in\C^+\,:\,\frac{1}{4} \frac{r_+^2-r^2}{t_+}<\eta\le(r_+^2-r^2)^{1/2}\}\,,\\
  {\mathcal{E}}_+^{(3)}&\deq\big\{z=\ii\eta\in\C^+\,:\,(r_+^2-r^2)^{1/2}<\eta\}\,.
 \end{split}\end{align}
We will prove the bound in~\eqref{le bounds on im omega23} separately for these three regimes. 

Choose $z\in{\mathcal{E}}_+^{(1)}$ first. We will argue by contradiction that
$ \im\omega_2(z) \le 2t_+$ for this domain. Assuming that $\omega\in\C^+$ with $\im\omega>2t_+$, we have the simple bound
\begin{align}\label{le ms1}
  \left| \frac{\chi(\omega)}{r_+^2-r^2}\right|=  \left|\frac{1}{r_+^2-r^2}\int_\R\frac{x^2\dd\widehat\mu_1(x)}{x-\omega} \right|\le\frac{s_+^2r_+^2}{r_+^2-r^2}\frac{1}{2t_+}
  = \frac{t_+}{2} \,,
\end{align}
where we used  $\widehat\mu_1(\R)=r_+^2$, and $|x|\le s_+$ on the support of $\widehat\mu_1$ 
in the first inequality and $t_+\ge s_+$ in the second. Now, for $z\in{\mathcal{E}}_+^{(1)}$, we have 
\begin{align}\label{le sas1}
 \frac{4|z|^2r_+^2}{(r_+^2-r^2-|z|^2)^2}\le\frac{r_+^2}{16t_+^2}\frac{(r_+^2-r^2)^2}{(r_+^2-r^2-|z|^2)^2}\le \frac{16}{15^2\sigma_+^2}<\frac{1}{10}\,,
\end{align}
where we use that $t_+=\sigma_+s_+$, $\sigma_+>1$, $s_+\ge r_+ $ and
\begin{align}\label{le sas11}
 \frac{|z|^2}{r_+^2-r^2}\le\frac{r_+^2-r^2}{16t_+^2}< \frac{r_+^2}{16t_+^2}<\frac{1}{16}\,,\qquad\qquad z\in{\mathcal{E}}_{+}^{(1)}\,,
\end{align}
since $t_+>r_+$. For $z\in{\mathcal{E}}_+^{(1)}$, we further have
\begin{align}\label{le sas2}
 \frac{4|z\chi(\omega_2(z))|}{(r_+^2-r^2-|z|^2)^2}\left|\frac{\omega_2(z)-z}{\omega_2(z)}\right|\le \frac{1}{2t_+}\frac{t_+}{2}\frac{(r_+^2-r^2)^2}{(r_+^2-r^2-|z|^2)^2}< \frac{3}{10	}\,,
\end{align}
where we used~\eqref{le ms1}, and $|\omega_2(z)-z| \le |\omega_2(z)|$ and $z\in {\mathcal{E}}_-^{(1)}$
to get the first inequality. 

We then obtain from~\eqref{le better way to do it}, upon expanding the square root using~\eqref{le sas1} and~\eqref{le sas2} that
\begin{align}\label{le sas4}
 |\omega_2(z)|&\le 2\frac{|z|r_+^2}{r_+^2-r^2-|z|^2}+2\frac{|\chi(\omega_2)|}{r_+^2-r^2-|z|^2}\le 2\frac{16|z|r_+^2}{15(r_+^2-r^2)}+\frac{16}{15}t_+<2t_+\,,
\end{align}
where we used~\eqref{le sas1},~\eqref{le sas11} and~\eqref{le sas2} to get the second inequality, and that $z\in{\mathcal{E}}_+^{(1)}$ and $r_+\le t_+$ to get the third. However,~\eqref{le sas4} yields a contradiction with the assumption that $\im\omega_2(z)>2t_+$. We can therefore conclude that
\begin{align}\label{le upper bound on E1}
 \im\omega_2(z)\le 2t_+\,,\qquad z\in{\mathcal{E}}_+^{(1)}\,. 
\end{align}

Choose now~$z\in{\mathcal{E}}_+^{(2)}$. Starting from~\eqref{le better way to do it}, we estimate
\begin{align*}
 |\omega_2(z)|&\le\frac{1}{2|z|}\left(2(r_+^2-r^2)+ 2|z|r_++ 2 \left(\frac{|z|s_+^2r_+^2}{|\omega_2(z)|}\right)^{\frac12}\right)\\
 &\le \frac{4t_+}{(r_+^2-r^2)}(r_+^2-r^2)+2\frac{r_+}{2} + 2\left(\frac{s_+^2r_+^2}{4|z\omega_2(z)|}\right)^{\frac12}\\
 &\le 5 t_++2 \left(\frac{t_+s_+^2r_+^2}{(r_+^2-r^2)|\omega_2(z)|}\right)^{\frac12}\,,\qquad\qquad z\in{\mathcal{E}}_+^{(2)}\,,
\end{align*}
where we used $z^2<0$, $r_+\le t_+$, $ (r_+^2-r^2)/(4t_+)\le |z|$, $|\omega_2(z)-z|\le |\omega_2(z)|$ and
\begin{align*}
 \big|\chi(\omega_2(z))\big|\le\frac{s_+^2r_+^2}{\im \omega_2(z)}\,,
\end{align*}
with $\im\omega_2(z)=|\omega_2(z)|$, for $z\in{\mathcal{E}}_+^{(2)}$. Thus at least one of the following bounds
holds
\begin{align*}
 |\omega_2(z)|&\le 10 t_+\qquad\textrm{or}\qquad |\omega_2(z)|\le 2
  \left(\frac{4t_+s_+^2r_+^2}{(r_+^2-r^2)|\omega_2(z)|}\right)^{\frac12}\,.
\end{align*}
In the latter case we find that
\begin{align*}
 |\omega_2(z)|\le \left(\frac{16 t_+s_+^2r_+^2}{r_+^2-r^2} \right)^{\frac13}=\left(16\sigma_+^2s_+^2 t_+\right)^{\frac13}< 
 3 t_+\,.
\end{align*}
Thus in both cases we have
\begin{align}\label{le upperbound on E2}
 |\omega_2(z)|\le 10t_+\,,\qquad\qquad z\in{\mathcal{E}}_+^{(2)}\,.
\end{align}

Finally, we consider $z\in {\mathcal{E}}_+^{(3)}.$
Since $\im\omega_2(z)\ge \im z$, we can expand~\eqref{le super nevan above} as 
\begin{align}
F_{\mu_1}(\omega_2(z))-\omega_2(z)=-\frac{r_+^2}{\omega_2(z)}+O(|\omega_2(z)|^{-3})=O(|z|^{-1})\,,
\end{align}
as $\im z\nearrow\infty$, which in turn implies through~\eqref{le some equation} that
\begin{align*}
\omega_2(z)=z-\frac{r^2}{z+O\left(|z|^{-1} \right)}\,,
\end{align*}
as $\im z\nearrow\infty$. Comparison with the Nevanlinna representation of $\omega_2(z)$ in~\eqref{le nevi for omega2} reveals that  $\nu_2(\R)=r^2$. Using~\eqref{le nevi for omega2} we can therefore estimate $\omega_2(z)$ from above as
\begin{align}\label{le trivial upper bound}
 |\omega_2(z)-z|\le \frac{r^2}{\im z}\,,\qquad\qquad z\in\C^+\,.
\end{align}
In particular, using $r<r_+\le s_+$ we have, for $\ii\eta\in{\mathcal{E}}_+^{(3)}$, that
\begin{align}\label{le upper bound on E3}
 |\omega_2(\ii\eta)-\ii\eta|\le \frac{r^2}{\eta}\le\sigma_+ s_+=t_+\,,\qquad\qquad z\in{\mathcal{E}}_+^{(3)}\,.
\end{align}

Combining~\eqref{le upper bound on E1},~\eqref{le upperbound on E2} and~\eqref{le upper bound on E3}, 
we see that there is a numeral $C$ such that
\begin{align}
 |\omega_2(\ii\eta)-\ii\eta|\le \min\Big \{C\sigma_+s_+, \frac{r^2}{\eta}\Big\}\le C\sigma_+s_+\,, \qquad\qquad \eta>0\,.
\end{align}
This proves~\eqref{le bounds on im omega23} and concludes the proof of the lemma.
\end{proof}
\begin{proof}[Proof of Theorem~\ref{0 in the bulk}]
Theorem~\ref{0 in the bulk} follows by combining~\eqref{le bounds on im omega22} and \eqref{le bounds on im omega23}, and adjusting the numerical constants.
\end{proof}

\section{Proof of Theorem~\ref{thm.081501} for large $\eta$}\label{s.large eta}

In this section, we prove Theorem~\ref{thm.081501} for spectral parameters $z\in\C^+$ with large imaginary parts, $\eta$. Here, large $\eta$ means  $\eta\ge \eta_\mathrm{M}$, for some $\eta_{\mathrm{M}}\ge 1$ independent of $N$ to be chosen below.

\subsection{Concentration of $m_H$ for large $\eta$}\label{le blink 183}
In this subsection, we fix an arbitrary $L>0$ and a compact interval $\mathcal{I}\subset\R$, and consider the domain $\mathcal{S}_{\mathcal{I}}\big(\eta_{\mathrm{M}}, N^{L}\big)$ introduced in~\eqref{le definition of caS domain}.

\begin{proof}[Proof of~(\ref{the local law for mH prec inequality}) on $\mathcal{S}_{\mathcal{I}}(\eta_{\mathrm{M}}, N^L)$]
In this proof, we choose both matrices $U$ and $V$ to be either Haar distributed on $U(N)$ or $O(N)$, \ie we treat the unitary and orthogonal case at once. For simplicity we refer to $U$ and $V$ as Haar matrices below. 

Our proof consists of two main steps. 
In the first step, we shall show that
\begin{align}
\Big| m_H(z)-m_{A}(\omega_B^c(z))\Big|\prec \frac{1}{N\eta^2}\,,\qquad\Big| m_H(z)-m_{B}(\omega_A^c(z))\Big|\prec \frac{1}{N\eta^2}\,, \label{081920}
\end{align} 
uniformly on $\mathcal{S}_{\mathcal{I}}\big(\eta_{\mathrm{M}}, N^{L}\big)$. In the second step, we 
use the local stability of the system~\eqref{le definiting equations} with the choice $(\mu_1,\mu_2)=(\mu_A,\mu_B)$ to conclude~\eqref{the local law for mH prec inequality} from~\eqref{081920} for large $\eta$.

 {\it Step 1: Proof of \eqref{081920}.}
 This step is based on the Gromov-Milman concentration inequality.  Let $\mathfrak{M}\equiv \mathfrak{M}(N)$ stand for the fundamental representation of either $U(N)$ or $O(N)$ on $M_N(\C)$, and let $\mathfrak{M}_1\equiv \mathfrak{M}_1(N)$ stand for the fundamental representation of either $SU(N)$ or $SO(N)$ on $M_N(\C)$, all endowed with the Riemann metric $\|{\rm d} s\|_2$ inherited from $M_N(\mathbb{C})$ (equipped with the Hilbert-Schmidt norm $\|\cdot\|_2$). We denote by $\P_{\mathfrak{M}}$, $\P_{\mathfrak{M}_1}$ (the push-forwards of)
 the Haar measure on $\mathfrak{M}$, $\mathfrak{M}_1$ respectively. We use the following version of the Gromov-Milman concentration inequality formulated as Corollary~4.4.28 in~\cite{AGZ}. If $g: (\mathfrak{M}(N), \|{\rm d} s\|_2)\to \C$ is an $\mathcal{L}$-Lipschitz function  then
\begin{align}
\mathbb{P}_{\mathfrak{M}}\Big(\Big|g(\,\cdot\,)-\int_{\mathfrak{M}_1}g(W\,\cdot\,)\dd \P_{\mathfrak{M}_1}(W)\Big|>\delta\Big)\le C\e{-c\frac{N\delta^2}{\mathcal{L}^2}}\,, \label{081505}
\end{align}
for all $\delta>0$, where $c>0$ and $C$ are numerical constants. 

To apply~\eqref{081505} with the Haar matrices $U$ and $V$ at once, we extend~\eqref{081505} to the direct product group $\mathfrak{M}\times\mathfrak{M}$ by adjusting the constants $c>0$ and $C$; see \eg~Theorem~1.11~of~\cite{Le01}.

For any deterministic matrix $Q\in M_{2N}(\mathbb{C})$, we introduce
\begin{align}\label{le f function}
 f(Q,\mathcal{U},z)\deq\ntr QG(z)\,,\qquad\qquad z\in\C^+\,,
\end{align}
where $\mathcal{U}$ is given in terms of $U$ and $V$ as in~\eqref{050970}, \ie is a Haar matrix on $\mathfrak{M}\times\mathfrak{M}$. 
 We will view $f(Q,\mathcal{U},z)$ as a random variable on $\mathfrak{M}\times\mathfrak{M}$. To apply~\eqref{081505}, we estimate the Lipschitz constant of $ f(Q,\,\cdot\,,z)\,:\, \mathfrak{M}\times\mathfrak{M}\rightarrow \C$.

Denote by $\mathfrak{m}\oplus\mathfrak{m}$ the (fundamental representation of the) Lie algebra of $\mathfrak{M}\times\mathfrak{M}$. Note that $X\in\mathfrak{m}\oplus\mathfrak{m}$ is a blockdiagonal matrix satisfying $X=-X^*$. For $X\in\mathfrak{m}\oplus\mathfrak{m}$ let $\mathrm{ad}_X\,:\,\mathfrak{m}\oplus\mathfrak{m}\rightarrow \mathfrak{m}\oplus\mathfrak{m}$, $Y\mapsto [X,Y]$ with $[\,\cdot\,,\,\cdot\,]$  the Lie bracket of $\mathfrak{m}\oplus\mathfrak{m}$, \ie the commutator on $M_{2N}(\C)$.
 Let~$B$ be as in~\eqref{050970}. Then for $X\in\mathfrak{m}\oplus\mathfrak{m}$ and $t\in\R$, we have $\e{t\mathrm{ad}_X}(\mathcal{U}B\mathcal{U}^*)=(\e{tX}\mathcal{U})B (\e{t X}\mathcal{U})^*$, where we used that $X=-X^*$. Furthermore, note that
    \begin{align}\label{le derivata rule}
  \frac{\dd }{\dd t}\e{t\mathrm{ad}_X}(\mathcal{U}B\mathcal{U}^*)=\e{t\mathrm{ad}_X}\mathrm{ad}_X (\mathcal{U}B\mathcal{U}^*)\,.
 \end{align}

For $X\in\mathfrak{m}\oplus\mathfrak{m}$, we then compute, using~\eqref{le derivata rule} and $\wt{B}=\mathcal{U}B\mathcal{U}^*$, that
\begin{align}
 \frac{\dd}{\dd t}f(Q, \e{tX}\mathcal{U},z)\Big|_{t=0}=-\ntr QG(\mathrm{ad}_X \wt{B})G=-\frac{1}{N}\mathrm{Tr} QG(\mathrm{ad}_X \wt{B})G\,. \label{081520}
\end{align}
We thus get the bound
\begin{align}
 \bigg| \frac{\dd}{\dd t}f(Q, \e{tX}\mathcal{U},z)\Big|_{t=0}\bigg|\leq \frac{2}{N}\| B\|\,\|QG\|\,\|GX\|_1\le \frac{C \|QG\|}{N} \|G\|_2\,\|X\|_2\,, \label{081521}
\end{align}
where $\|\cdot\|_1$ denotes  the trace norm. We used Schwarz inequality and $\|B\|\le C$ by assumption to get the last inequality. Since $|G(z)|^2=\frac{\im G(z)}{\eta}$ and $\|G(z)\|\le \eta^{-1}$, we get from~\eqref{081521} that
\begin{align}
  \bigg| \frac{\dd}{\dd t}f(Q, \e{tX}\mathcal{U},z)\Big|_{t=0}\bigg|\le \frac{C \|Q\|}{\sqrt{N}\eta^2}\|X\|_2\,,\qquad\qquad z\in\C^+\,.
\end{align}
Thus the Lipschitz constant of $f(Q)$ is bounded above by $C\|Q\|/(\sqrt{N}\eta^2)$. We therefore obtain from~\eqref{081505} the concentration inequality
\begin{align}\label{le partial average concentation}
 \bigg|f(Q,\mathcal{U},z)-\int_{\mathfrak{M}_1\times\mathfrak{M}_1} f(Q,\mathcal{W}\cdot\mathcal{U},z)\dd \P_{\mathfrak{M}_1\times\mathfrak{M}_1}(\mathcal{W})\bigg|\prec \frac{\|Q\|}{N\eta^2}\,,
\end{align}
where the randomness behind the notation $\prec$ is provided by the Haar measure on $\mathfrak{M}\times\mathfrak{M}$.

We next identify the average appearing on the left side of~\eqref{le partial average concentation}. For a function $g\,:\,\mathfrak{M}\times\mathfrak{M}\rightarrow \C$, $\mathcal{U}\mapsto g(\mathcal{U})$, we introduce the shorthand
\begin{align}
 \widetilde \E g(\mathcal{U})\deq\int_{\mathfrak{M}_1\times\mathfrak{M}_1} g(\mathcal{W}\cdot\mathcal{U})\,\dd \P_{\mathfrak{M}_1\times\mathfrak{M}_1}(\mathcal{W})\,.
\end{align}
Using the invariance of Haar measure on $\mathfrak{M}_1\times\mathfrak{M}_1$, we are going to compute $\widetilde\E\ntr G(z)$. Denote by $\frak{m}_1\oplus\frak{m}_1$ the Lie algebra of $\frak{M}_1\times\frak{M}_1 $. The following argument is essential due to~\cite{VP}; see also~\cite{BG,Kargin} for similar arguments. Viewing the Green function 
as a function (random variable) on $\mathfrak{M}\times\mathfrak{M}$, 
 $G(\cdot,z)\,:\mathfrak{M}\times\mathfrak{M}\rightarrow M_N(\C)$, we compute, using~\eqref{le derivata rule}, that
\begin{align}
 \widetilde{\E}\Big[\frac{\dd}{\dd t} G(\e{tX}\mathcal{U},z)\Big|_{t=0}\Big]=-\widetilde{\E}\Big[G(\mathcal{U},z) \mathrm{ad}_X(\wt B)G(\mathcal{U},z)\Big]\,, \label{102901}
\end{align}
for any $X\in\mathfrak{m}_1\oplus\mathfrak{m}_1$, where $\wt B\equiv\mathcal{U}B\mathcal{U}^*$. On the other hand, by the left-invariance of Haar measure, we also have $\frac{\dd}{\dd t}\widetilde{\E} G(\e{tX}\mathcal{U},z)=0$, for all $t\in\R$ and all $X\in\mathfrak{m}_1\oplus\mathfrak{m}_1$. Thus we get from~\eqref{102901} that $\widetilde{\E}[G(\mathcal{U},z) \mathrm{ad}_X(\wt B)G(\mathcal{U},z)]=0$, for any $X\in\mathfrak{m}_1\oplus\mathfrak{m}_1$,
 \ie  we have
\begin{align}\label{le pastur trick}
 \widetilde{\E}[G(\mathcal{U},z) [X,\wt B]G(\mathcal{U},z)]=0\,.
\end{align}
Such formulas originating from basic symmetries of the model are often called {\it Ward identities} in physics.
Let now $Y\deq \hat{\mathbf{e}}_i\hat{\mathbf{e}}_k^*$, with $i\not=k$, $i,k\in\llbracket 1,N\rrbracket$. We then note that we can decompose $Y=\frac{1}{2}X_1+\frac{1}{2\ii}X_2$, where $X_1\deq Y-Y^*$ and $X_2\deq(\ii Y+\ii Y^*)$. Note that $X_1,X_2\in \mathfrak{m}_1\oplus\mathfrak{m}_1$. Thus we have from~\eqref{le pastur trick} that
\begin{align}\label{le pastur trick 2}
 \widetilde{\E}[G(\mathcal{U},z) [X_\iota,\wt B]G(\mathcal{U},z)]=0\,,\qquad\qquad \iota=1,2\,.
\end{align}
Since $Y$ is a linear combination of $X_1$ and $X_2$, we conclude by the linearity of the commutator and~\eqref{le pastur trick 2} that, for $i\not=k$, 
\begin{align}\label{le pastur 1}
\widetilde{\E}[G(\mathcal{U},z) [ \hat{\mathbf{e}}_i\hat{\mathbf{e}}_k^*,\wt B]G(\mathcal{U},z)] =0\,.
\end{align}

Next, recall the notational convention $\hat i\equiv i+N$, for $i\in\llbracket 1,N\rrbracket$. Using exactly the same argument as above we infer, for $i\not=k$, $i,k\in\llbracket 1,N\rrbracket$, that
\begin{align}\label{le pastur 2}
\widetilde{\E}[G(\mathcal{U},z)[\hat{\mathbf{e}}_{\hat{i}}\hat{\mathbf{e}}_{\hat{k}}^*,\wt B]G(\mathcal{U},z)]=0\,.
\end{align}

Thus, taking matrix elements of~\eqref{le pastur 1} and~\eqref{le pastur 2}, we obtain, for all $i\in\llbracket 1,N\rrbracket$, $ j=i,\hat{i}$,
\begin{align}
\Big|\wt\E\Big[  \tau_1(G(\mathcal{U},z))(G(\mathcal{U},z)\wt{B})_{ji}-\tau_1(\wt{ B }G(\mathcal{U},z)) G_{j i}(\mathcal{U},z) \Big]\Big|\le \frac{C}{N\eta^2}\,, \label{081690}
\end{align}
and
\begin{align}
\Big|\wt\E\Big[  \tau_2(G(\mathcal{U},z))(G(\mathcal{U},z)\wt{B})_{j\hat{i}}-\tau_2(\wt{B}G(\mathcal{U},z))G_{j\hat{i}}(\mathcal{U},z) \Big]\Big|\le \frac{C}{N\eta^2}\,,\label{081691}
\end{align} 
for some constant $C$ depending only on $\|B\|$, where the error terms result from coincidences among 
indices when using~\eqref{le pastur 1} and~\eqref{le pastur 2}. Here we also used $\|G(z)\|\le \eta^{-1}$.

Suppressing for simplicity the $z$- and $\mathcal{U}$-dependences in the notation for the Green function, we next note the identities 
\begin{align}
&(G\wt{ B })_{ii}=1+zG_{ii}-\bar{\xi}_i G_{i\hat{i}}\,,\qquad & &(G\wt{B})_{i\hat{i}}=-\xi_iG_{ii}+z G_{i\hat{i}}\,, \nonumber\\ 
&(G\wt{ B })_{\hat{i}i}=-\bar{\xi}_i G_{\hat{i}\hat{i}}+zG_{\hat{i}i}\,,\qquad & &(G\wt{ B })_{\hat{i}\hat{i}}=1+zG_{\hat{i}\hat{i}}-\xi_i G_{\hat{i}i}\,,
 \label{071701}
\end{align}
 for all $i\in \llbracket 1, N\rrbracket$, which follow from~(\ref{102702}). Plugging~(\ref{071701}) into~(\ref{081690}) and~(\ref{081691}) we get
\begin{align}\begin{split}
\Big|\wt\E\Big[ \big(1+zG_{ii}-\bar{\xi}_i G_{i\hat{i}}\big)\tau_1(G) - G_{ii}\tau_1(\wt{ B }G)\Big]\Big|&\le \frac{C}{N\eta^2}\,,\\
\Big|\wt\E\Big[ \big(-\bar{\xi}_i G_{\hat{i}\hat{i}}+zG_{\hat{i}i}\big) \tau_1(G) - G_{\hat{i}i}\tau_1(\wt{ B }G)\Big]\Big|&\le \frac{C}{N\eta^2}\,, \\
\Big|\wt\E\Big[ \big(-\xi_i G_{ii}+zG_{i\hat{i}}\big)\tau_2(G) - G_{i\hat{i}}\tau_2(\wt{ B }G)\Big]\Big|&\le \frac{C}{N\eta^2}\,, \\
\Big|\wt\E\Big[ \big(1+zG_{\hat{i}\hat{i}}-\xi_i G_{\hat{i}i}\big)\tau_2(G) - G_{\hat{i}\hat{i}}\tau_2(\wt{ B }G)\Big]\Big|&\le \frac{C}{N\eta^2}\,. \label{081693}
\end{split}\end{align}

Next, by~\eqref{le partial average concentation} we have the concentration inequalities
\begin{align*}
&\Big|\tau_a(G)-\wt\E\big[\tau_a(G)\big]\Big|\prec \frac{1}{N\eta^2}\,,\qquad \Big|\tau_a(\wt{B}G)-\wt\E\big[\tau_a(\wt{B}G)\big]\Big|\prec \frac{1}{N\eta^2}\,.
\end{align*}
For the second estimate we used that $\tau_a(\wt BG)$ can be brought into the form $\ntr QG$ with a deterministic $Q$
with the help of~(\ref{102702}). Hence, we can go back and forth between the tracial quantities $\tau_a(G)$, $\tau_a(\wt B G)$ and their partial expectations $\wt\E\tau_a(G)$ and $\wt\E\tau_a(\wt BG)$, up to an error $O_\prec(\frac{1}{N\eta^2})$ in the following discussion. Pulling out the expectation of the tracial quantities 
 and combining the first and the third equations in~(\ref{081693})  we eliminate $\wt\E G_{i\hat{i}}$ and get an equation for
$\wt\E G_{ii}$. After solving for $\wt\E G_{ii}$, we may remove the partial expectation~$\wt\E$ from the tracial quantities. We get
\begin{multline*}
 \Big( \big(z\tau_1(G)-\tau_1(\wt{B}G)\big)\big(z\tau_2(G)-\tau_2(\wt{B}G)\big)-|\xi_i|^2  \tau_1(G)\tau_2(G) 
\Big)\wt{\E}\big[G_{ii}\big]\nonumber\\
\qquad\qquad+\tau_1(G)(z\tau_2(G)-\tau_2(\wt{B}G))=O_\prec\Big(\frac{1}{N\eta^2}\Big)\,.
\end{multline*}
Here we used once more the bound $\|G\|\le 1/\eta$. Dividing the above equation by $\tau_1(G)\tau_2(G)$ and using the fact $|\tau_a(G)- \frac{{ \ii}}{\eta}|\le O(\eta^{-2})$, we obtain 
\begin{align}
 \big(\omega_{B,1}^c\omega_{B,2}^c-|\xi_i|^2\big)\wt\E\big[G_{ii}\big]+\omega_{B,2}^c=O_\prec\Big(\frac{1}{N}\Big)\,, \label{081910}
\end{align}
for all $z\in\mathcal{S}_\mathcal{I}(\eta_{\mathrm{M}},N^L)$, by choosing $\eta_M>0$ sufficiently large. Here we introduced the auxiliary subordination functions
\begin{align}
\omega_{B,a}^c(z)\deq z-\frac{\tau_a(\wt{ B }G(z))}{\tau_a(G(z))}\,,\qquad\qquad a=1,2\,,\qquad z\in\C^+\,,  \label{102401}
\end{align}
which are defined using the partial traces $\tau_a$ instead of the full traces as in~\eqref{091107}.

We further observe that a large $z$ expansion in $\C^+$ of the resolvent yields
\begin{align}\label{dae johan mues halt robotere}
\tau_a(\wt{B}G(z))=-\frac{\tau_a(\wt{B})}{z}+O\Big(\frac{1}{|z|^2}\Big)=O\Big(\frac{1}{|z|^2}\Big)\,,
\end{align}
as $|z|\rightarrow\infty$, where we used that $\tau_a(\wt B)=0$. Thus from~\eqref{102401} we find that
\begin{align}
\omega_{B,a}^c(z)=z+O\big(|z|^{-1}\big)\,,\qquad a=1,2\,,  \label{081911}
\end{align}
as $|z|\to\infty$. Combining~(\ref{081910}) and~(\ref{081911}) we find
\begin{align}
\wt\E\big[G_{ii}(z)\big]-\frac{\omega_{B,2}^c(z)}{|\xi_i|^2-\omega_{B,1}^c(z)\omega_{B,2}^c(z)}=O_\prec\Big(\frac{1}{N\eta^2}\Big)\,,\qquad \forall i\in \llbracket 1,N\rrbracket\,, \label{081940}
\end{align}
for all $z\in\mathcal{S}_\mathcal{I}(\eta_{\mathrm{M}},N^L)$, for sufficiently large $\eta_M>0$. Analogously, we have 
\begin{align}
\wt\E\big[G_{\hat{i}\hat{i}}(z)\big]-\frac{\omega_{B,1}^c(z)}{|\xi_i|^2-\omega_{B,1}^c(z)\omega_{B,2}^c(z)}=O_\prec\Big(\frac{1}{N\eta^2}\Big)\,,\qquad \forall i\in \llbracket 1,N\rrbracket\,, \label{081941}
\end{align}
for all $z\in\mathcal{S}_\mathcal{I}(\eta_{\mathrm{M}},N^L)$. From $\tau_1(G)=\tau_2(G)$, see~\eqref{neu083140}, and from~\eqref{081911}, we  obtain from~(\ref{081940})  and~(\ref{081941}) that
\begin{align*}
\omega_{B,1}^c=\omega_{B,2}^c+O_\prec\Big(\frac{1}{N}\Big)=\omega_{B}^c+O_\prec\Big(\frac{1}{N}\Big)\,,
\end{align*}
where $\omega_B^c$ is defined in~\eqref{091107}. The second equality follows from the fact that 
$\tau_1(G)=\tau_2(G)$ implies that this common value is $\ntr G$. Hence
$\omega_{B,1}^c\approx \omega_{B,2}^c$ implies $\tau_1(\widetilde BG) \approx \tau_2(\widetilde BG) $, hence
both are close to their average, $\ntr \widetilde B G$.  We therefore also  have
\begin{align}
&\wt\E\big[G_{ii}(z)\big]-\frac{\omega_{B}^c(z)}{|\xi_i|^2-(\omega_{B}^c(z))^2}=O_\prec\Big(\frac{1}{N\eta^2}\Big)\,,\nonumber \\
&\wt\E\big[G_{\hat{i}\hat{i}}(z)\big]-\frac{\omega_{B}^c(z)}{|\xi_i|^2-(\omega_{B}^c(z))^2}=O_\prec\Big(\frac{1}{N\eta^2}\Big)\,,\qquad\quad \forall i\in \llbracket 1,N\rrbracket\,, \label{081943}
\end{align}
uniformly on $z\in\mathcal{S}_\mathcal{I}(\eta_{\mathrm{M}},N^L)$ by choosing $\eta_M>0$ sufficiently large. Averaging~\eqref{081943} over~$i$ and using the concentration estimate~\eqref{le partial average concentation} with $Q=\hat{I}$, we obtain the first estimate in~(\ref{081920}). The second estimate in~\eqref{081920} is obtained in the same way by interchanging 
the r\^oles of~$A$ and~$B$. This completes the first step of the argument.

\medskip

{\it Step 2: Stability analysis.}
We move on to check the stability of the system  
\begin{align}
\Phi_{\mu_A,\mu_B} \big(\omega_A,\omega_B, z\big)=0\,,\label{le voxx}
\end{align}
for $z\in \mathcal{S}_{\mathcal{I}}(\eta_{\mathrm{M}},N^L)$; see~\eqref{le H system defs} for the definition of $\Phi_{\mu_A,\mu_B}$. First, we will show that~$\omega_A^c(z)$ and~$\omega_B^c(z)$ approximately solve~\eqref{le voxx}. Then we will conclude from Lemma~A.2 of~\cite{BES17} that~$\omega_A^c(z)$ and~$\omega_B^c(z)$ are close to~$\omega_A(z)$ and~$\omega_B(z)$.

Using that $F_{\mu_A}(\omega_B^c(z))=-1/m_{\mu_A}(\omega_B^c(z))$ and the identity~(\ref{091110}), we can write
\begin{align}
F_{\mu_A}(\omega_B^c(z))-\omega_A^c(z)-\omega_B^c(z)+z&=\frac{1}{m_H(z)m_A(\omega_B^c(z))}\Big(m_A(\omega_B^c(z))-m_H(z)\Big)\,. \label{081950}
\end{align}
From the resolvent expansion $G(z)= -1/z + O(1/|z|^2)$ in the large $|z|$ regime we have
 that $ |m_H(z)-\ii \eta^{-1}|\le C\eta^{-2}$ and $|\omega_B^c(z)-\ii\eta|\le C\eta^{-1}$; for the latter estimate we also
 used  $\ntr \wt B=0$ in \eqref{091107}. Thus together with 
  the estimates in~(\ref{081920}), we get from~\eqref{081950} that
\begin{align*}
F_{\mu_A}(\omega_B^c(z))-\omega_A^c(z)-\omega_B^c(z)+z=O_\prec\Big(\frac{1}{N}\Big)\,,
\end{align*}
uniformly in $z\in\mathcal{S}_\mathcal{I}(\eta_{\mathrm{M}},N^L)$, by choosing $\eta_M>0$ large enough. Analogously, we also have
\begin{align*}
F_{\mu_B}(\omega_A^c(z))-\omega_A^c(z)-\omega_B^c(z)+z=O_\prec\Big(\frac{1}{N}\Big)\,,
\end{align*}
on the same domain. Hence we have 
\begin{align}\label{le phi system large eta}
\|\Phi(\omega_A^c(z),\omega_B^c(z),z)\|_2\prec N^{-1}\,,
\end{align}
for all $z\in\mathcal{S}_{\mathcal{I}}(\eta_\mathrm{M},N^L)$. 

Next, observe that  the deterministic bounds
\begin{align}
|\omega_A^c(z)-z|\leq \frac{C}{|z|}, \qquad |\omega_B^c(z)-z|\leq \frac{C}{|z|} \label{17101201}
\end{align}
hold uniformly for all $|z|\geq \eta_{M}$ with sufficiently large $\eta_M$. This follows from~\eqref{091107}, a large $z$ expansion of $G(z)$ and $\ntr A=\ntr B=0$.  Consequently, it is easy to check the following deterministic bound also holds uniformly for all $z$ with $|z|\geq \eta_M$
\begin{align}
\|\Phi(\omega_A^c(z),\omega_B^c(z),z)\|_2\leq  \frac{C}{|z|}.  \label{17101202}
\end{align}
Then we apply Lemma~A.2 of~\cite{BES17}. Thanks to ~\eqref{17101201} and~\eqref{17101202}, the assumptions of Lemma~A.2 of~\cite{BES17} are satisfied and we further conclude from ~\eqref{le phi system large eta} that
\begin{align}
\big|\omega_a^c(z)-\omega_a(z)\big|\prec \frac{1}{N}\,,\qquad\quad a=A,B\,, \label{081971}
\end{align}
uniformly in $z\in\mathcal{S}_\mathcal{I}(\eta_\mathrm{M},N^L)$ by slightly adjusting the value of $\eta_M$.

Combining~\eqref{081971} with~\eqref{081920} and $|m_A'(\omega)|= O( |\omega|^{-2}) = O(\eta^{-2}) $ in the regime where $\omega \approx  \ii\eta$ and $\eta$ is large, we find that
\begin{align*}
m_H(z)-m_A(\omega_B(z))=O_\prec\Big(\frac{1}{N\eta^2}\Big)\,,\qquad\qquad z\in\mathcal{S}_\mathcal{I}(\eta_\mathrm{M},N^L)\,.
\end{align*}
Finally, since $m_{\mu_A\boxplus\mu_B}=m_A(\omega_B)$ we conclude the proof of~(\ref{the local law for mH prec inequality}) for $z\in\mathcal{S}_\mathcal{I}(\eta_\mathrm{M},N^L)$.
\end{proof}

\subsection{Green function subordination for large $\eta$} 
In this subsection, we show the following subordination property for the Green function entries when $\eta$ is large. Recall from~\eqref{071603} the control parameter $\Lambda_d$.

\begin{lem}\label{lem.081970} Under the conditions and with the notations of Theorem~\ref{thm.081501} there is a (large) constant~$\eta_{\mathrm{M}}$ such that
\begin{align}
\Lambda_{\rm d}(z)\prec\frac{1}{\sqrt{N\eta^4}}\,, \label{081983}
\end{align}
uniformly on $\mathcal{S}_{\mathcal{I}}(\eta_{\mathrm{M}}, N^L)$.
\end{lem}
\begin{proof}[Proof of Lemma~\ref{lem.081970}]
Let $\eta_{\mathrm{M}}$ be as in Subsection~\ref{le blink 183}.
From~(\ref{081943}) and~(\ref{081971}) we directly get
\begin{align}
&\wt\E\big[G_{ii}(z)\big]-\frac{\omega_{B}(z)}{|\xi_i|^2-\omega^2_{B}(z)}=O_\prec\big(\frac{1}{N\eta^2}\big)\,,\quad \wt\E\big[G_{\hat{i}\hat{i}}(z)\big]-\frac{\omega_{B}(z)}{|\xi_i|^2-\omega^2_{B}(z)}=O_\prec\big(\frac{1}{N\eta^2}\big)\,,\nonumber\\
&\wt\E\big[G_{i\hat{i}}(z)\big]-\frac{\xi_i}{|\xi_i|^2-\omega_{B}^2(z)}=O_\prec\big(\frac{1}{N\eta^2}\big)\,,\quad \wt\E\big[G_{\hat{i}i}(z)\big]-\frac{\bar{\xi}_i}{|\xi_i|^2-\omega_{B}^2(z)}=O_\prec\big(\frac{1}{N\eta^2}\big)\,, \label{081981}
\end{align}
for all $z\in\mathcal{S}_{\mathcal{I}}(\eta_\mathrm{M},N^L)$. Hence, it remains to show the concentration of these entries of the Green function. To this end, we regard, as in Subsection~\ref{le blink 183}, the Green function entries as functions of $\mathcal{U}$, and use the Gromov-Milman concentration inequality in~\eqref{081505}. The Lipschitz constant of $G_{ij}(\cdot, z)\,:\,\mathfrak{M}\times\mathfrak{M}\rightarrow \C$, $\mathcal{U}\mapsto G_{ij}(\mathcal{U},z)$ is estimated by bounding, for $X\in\mathfrak{m}\oplus\mathfrak{m}$, 
\begin{align*}
\bigg|\frac{\dd G_{ij}(\e{tX}\mathcal{U},z)}{\dd t} \Big|_{t=0}\bigg|= \bigg|\hat{\mathbf{e}}_i^*G(\mathcal{U},z)  \mathrm{ad}_X \wt{B} G(\mathcal{U},z)\hat{\mathbf{e}}_j\bigg|\le \frac{C \|X\|_2}{\eta^2}\,,
\end{align*}
 with a constant $C$ depending only on $\|B\|$, where we first used~\eqref{le derivata rule} and then Schwarz inequality.
  Thus by~\eqref{081505}, 
\begin{align*}
\Big|G_{ij}(z)-\wt\E\big[G_{ij}(z)\big]\Big|\prec \frac{1}{\sqrt{N\eta^4}}\,,\qquad\qquad z\in\mathcal{S}_{\mathcal{I}}(\eta_{\mathrm{M}},N^L)\,. 
\end{align*} 
Combining these concentration results
with~\eqref{081981} we find~\eqref{081983}. To obtain uniform bounds in $z \in
\mathcal{S}_{\mathcal{I}}(\eta_{\mathrm{M}},N^L)$, we can apply a simple lattice argument using the Lipschitz continuity of the Green function $G(z)$ and of the two subordination functions~$\omega_A(z)$ and~$\omega_B(z)$. See the proof of Theorem~\ref{main theorem} in Section~\ref{s. proof of main theorem} for a similar argument. The uniform Lipschitz continuity of the subordination functions follows directly from their analyticity on $\mathcal{S}_{\mathcal{I}}(\eta_{\mathrm{M}},N^L)$. This completes the proof of Lemma~\ref{lem.081970}.
\end{proof}

  \appendix

\section{} \label{appendix A}

\subsection{Stochastic domination and large deviation properties}\label{stochastic domination section}
Recall the stochastic domination in Definition~\ref{definition of stochastic domination}. The relation $\prec$ is transitive and it satisfies the following arithmetic rules: if $X_1\prec Y_1$ and $X_2\prec Y_2$ then $X_1+X_2\prec Y_1+Y_2$ and $X_1 X_2\prec Y_1 Y_2$.  Further assume that $\Phi(v)\ge N^{-C}$ is deterministic and that~$Y(v)$ is a non-negative random variable satisfying $\E [Y(v)]^2\le N^{C'}$ for all~$v$. Then $Y(v) \prec \Phi(v)$, uniformly in $v$, implies $\E [Y(v)] \prec \Phi(v)$, uniformly in~$v$.

Gaussian vectors have well-known large deviation properties. We will use them in the following form
whose proof is standard.  
\begin{lem} \label{lem.091720} Let $X=(x_{ij})\in M_N(\C)$ be a deterministic matrix and let $\bs{y}=(y_{i})\in\C^N$ be a deterministic complex vector. For a Gaussian real or complex random vector $\mathbf{g}=(g_1,\ldots, g_N)\in \mathcal{N}_{\mathbb{R}}(0,\sigma^2 I_N)$ or $\mathcal{N}_{\mathbb{C}}(0,\sigma^2 I_N)$, we have
 \begin{align}\label{091731}
  |\bs{y}^* \bs{g}|\prec\sigma \|\bs{y}\|_2\,,\qquad\qquad  |\bs{g}^* X\bs{g}-\sigma^2N \ntr X|\prec \sigma^2\| X\|_2\,.
 \end{align}
\end{lem}

\subsection{Bounds on subordination functions}\label{bounds on subordination}Let $\mu_\alpha,\mu_\beta$ be two $N$-independent probability measures on $\mathbb{R}$ which are compactly supported: there exists a constant $L<\infty$ such that
\begin{align}
\text{supp}(\mu_\alpha)\subset [-L,L]\,,\qquad \text{supp} (\mu_\beta) \subset [-L,L]\,. \label{110250}
\end{align}
Let $\omega_\alpha, \omega_\beta$ be the subordination functions defined via the system of equations~\eqref{le definiting equations}.  The following result is proved in \cite{BES15}.
\begin{lem}[Lemma 5.1 and Corollary 5.2 in~\cite{BES15}] \label{lem.A.1} Suppose that neither  $\mu_\alpha$  nor $\mu_\beta$ is a single point mass and at least of one of them is supported at more than two points. Assume in addition that ~(\ref{110250}) holds. Let $\mathcal{I}\subset \mathcal{B}_{\mu_\alpha\boxplus \mu_\beta}$ be a compact non-empty interval in the  bulk of $\mu_\alpha\boxplus\mu_\beta$. Fix any $0<\eta_{\rm M}<\infty$.  Let $\mu_A$, $\mu_B$ be ($N$-dependent) probability measures on $\R$.
Then there exist constants $b_0>0$, $k>0$ and $K<\infty$, $S<\infty$, which depend only on $\eta_{\rm M}$, $L$ in~\eqref{110250}, the  interval $\mathcal{I}$ and the measures $\mu_\alpha$ and $\mu_\beta$, such that whenever
\begin{align}\label{le assumption on levy distances 2}
 \mathrm{d}_{\mathrm{L}}(\mu_A,\mu_\beta)+\mathrm{d}_{\mathrm{L}}(\mu_B,\mu_\beta)\le b_0\,,
\end{align}
then
\begin{align}
&\max_{z\in \mathcal{S}_{\mathcal{I}}(0,\eta_{\rm M})} |\omega_A(z)| \leq K\,,\qquad \max_{z\in \mathcal{S}_{\mathcal{I}}(0,\eta_{\rm M})} |\omega_B(z)| \leq K\,,\nonumber\\
& \min_{z\in \mathcal{S}_{\mathcal{I}}(0,\eta_{\rm M})} \Im \omega_A(z) \geq k\,,\qquad  \min_{z\in \mathcal{S}_{\mathcal{I}}(0,\eta_{\rm M})} \Im \omega_B(z) \geq k\,, \nonumber\\
&  \max_{z\in \mathcal{S}_{\mathcal{I}}(0,\eta_{\rm M})} |\omega_A'(z)| \leq S\,,\qquad  \max_{z\in \mathcal{S}_{\mathcal{I}}(0,\eta_{\rm M})} |\omega_B'(z)| \leq S\,,  \label{110270}
\end{align}
for all $N\geq N_0$ with  some sufficiently large $N_0$ depending only on $\eta_{\rm M}$, $L$ in~\eqref{110250}, the interval~$\mathcal{I}$ and the measures $\mu_\alpha$ and $\mu_\beta$. Here  $\omega_A$, $\omega_B$ denote the subordinations functions defined via~\eqref{le definiting equations} for the choice $(\mu_1,\mu_2)=(\mu_A, \mu_B)$.
\end{lem}
\subsection{Bounded rank perturbation estimate}
At various places, we use the following perturbation estimate; see Section 3.2 of \cite{BES15b} for proof, for instance.
\begin{lem}\label{finite rank perturbation}
Let $D\in M_N(\C)$ be Hermitian and let $Q\in M_N(\C)$ be arbitrary. Then, for any Hermitian matrix $R\in M_N(\C)$, we have
\begin{align}
\big|\ntr \big(Q(D+R-z)^{-1}\big)-\ntr \big(Q(D-z)^{-1}\big)\big| &\leq \frac{\mathrm{rank}(R)\|Q\|}{N\eta}\,,\qquad z=E+\ii\eta\in\C^+\,.
 \label{091002}
\end{align}
\end{lem}
 Lemma~\ref{finite rank perturbation} also has the following corollary.
\begin{cor}\label{cor.finite rank} Let $Q\in M_{2N}(\C)$ be arbitrary matrix. Then there is a numerical constant $C$  such that, with the notations defined in~(\ref{052704}),~(\ref{052703}) and~(\ref{0916110}), we have
\begin{align}
\big| \ntr QG-\ntr QG^{\la i\ra}\big| &\leq \frac{C\|Q\|}{N\eta}\,,\qquad &  \big|\ntr Q\wt{B}G- \ntr Q\wt{B}^{\la i\ra} G\big|&\leq  \frac{C\|Q\|}{N\eta}\,,\nonumber\\
\big|\ntr Q\wt{B}G- \ntr Q\wt{B}^{\la i\ra} G^{\la i\ra}\big|&\leq  \frac{C\|Q\|}{N\eta}\,, \qquad & \big|\ntr Q\wt{B}G\wt{B}- \ntr Q\wt{B}^{\la i\ra} G^{\la i\ra}\wt{B}^{\la i\ra}\big|&\leq  \frac{C\|Q\|}{N\eta}\,. \label{0916136}
\end{align}
\end{cor}
\begin{proof}
The first inequality follows from~(\ref{091002}) directly  since $H^{\la i\ra}$ is a bounded rank perturbation of $H$. Next, we show the second inequality. Note that
\begin{align}
\ntr Q\wt{B}^{\la i\ra} G-\ntr Q\wt{B}G=\ntr Q\wt{B}^{\la i\ra} G-\ntr Q\mathcal{R}_i\wt{B}^{\la i\ra} \mathcal{R}_iG\,. \label{0916120}
\end{align}
Denote by $\hat{\mathbf{r}}_i^u=\ell_i^u\big(\hat{\mathbf{e}}_i+\mathbf{k}_i^u\big)$ and $\hat{\mathbf{r}}_i^v=\ell_i^v\big(\hat{\mathbf{e}}_{\hat{i}}+\mathbf{k}_i^v\big)$. By the definition in~(\ref{def of R}) and~(\ref{052704}), we have 
 $\mathcal{R}_i=\hat{I}-\hat{\mathbf{r}}_i^u(\hat{\mathbf{r}}_i^u)^*-\hat{\mathbf{r}}_i^v(\hat{\mathbf{r}}_i^v)^*$.
Then it is easy to check the right side of~(\ref{0916120}) is a sum of the terms of the form 
\begin{align}
\frac{\wt{d}_i}{N} (\hat{\mathbf{r}}_i^a)^* \wt{B}^{\la i\ra} GQ\hat{\mathbf{r}}_i^b\,,\qquad \frac{\wt{d}_i}{N} (\hat{\mathbf{r}}_i^a)^*  GQ \wt{B}^{\la i\ra}\hat{\mathbf{r}}_i^b\,,\qquad  \frac{\wt{d}_i}{N} (\hat{\mathbf{r}}_i^a)^*  GQ\hat{\mathbf{r}}_i^b\,,  \label{0916135}
\end{align}
 or products of some of them, for some $\wt{d}_i$ which could be different from one to another, up to the bound $|\wt{d}_i|\leq C$.  Here $a, b=u,v$. 
 Clearly, the terms in~(\ref{0916135}) are all bounded by $\frac{C\|Q\|}{N\eta}$.  This proves the second estimate in~(\ref{0916136}).  The third bound in~(\ref{0916136}) follows from the second one and~(\ref{091002}) immediately. The last  one  can also be proved analogously. 
\end{proof}

\section{}\label{Appendix B}
In this appendix, we bound the terms involving $\Delta_R^u(i,k)$, \ie  the terms in~(\ref{083050}), the last term of~(\ref{100501}), ~(\ref{112450}) and the last term of~(\ref{112455}). We summarize the bound in the next lemma.
\begin{lem} \label{lem.100655} Let $Q\in M_{2N}(\mathbb{C})$ be arbitrary, with $\|Q\|\prec 1$. Let $X_i=\hat{I}$ or $\wt{B}^{\la i\ra}$ and~$X=\hat{I}$ or~$A$. Suppose that the assumptions in Theorem~\ref{071601} hold.  Then,
\begin{align}
&|\varepsilon_{i2}|\prec \Psi^2,\qquad |\varepsilon_{i3}|\prec \Psi^2\,,\label{120501}\\
& \Big|\frac{1}{N} \sum_{k}^{(i)} \hat{\mathbf{e}}_i^*X\Delta_{G}^u(i,k)\hat{\mathbf{e}}_i \hat{\mathbf{e}}_k^*  X_i G\hat{\mathbf{e}}_i \Big|\prec \Psi^2\,,\label{120502}\\
& \Big|\frac{1}{N} \sum_{k}^{(i)} (\mathbf{k}_i^u)^*\Delta_{G}^u(i,k)\hat{\mathbf{e}}_i \hat{\mathbf{e}}_k^*  X_i G\hat{\mathbf{e}}_i \Big|\prec \Psi^2\,, \label{120503}\\
& \Big| \frac{1}{N} \sum_{k}^{(i)} \ntr QX\Delta_{G}^u(i,k) \hat{\mathbf{e}}_k^*  X_i G\hat{\mathbf{e}}_i\Big|\prec \Psi^4\,. \label{100657}
\end{align}
\end{lem}
\begin{proof}[Proof of Lemma~\ref{lem.100655}] Recalling~(\ref{0916100}), we see that $\Delta_R^u(i,k)$ is a sum of terms of the form
\begin{align*}
\wt{d}_i\bar{g}_{ik}^u \boldsymbol{\alpha}_i \boldsymbol{\beta}_i^*\,,
\end{align*}
for some $\wt{d}_i\in \mathbb{C}$ satisfying $|\wt{d}_i|\prec 1$, and $\boldsymbol{\alpha}_i,\boldsymbol{\beta}_i=\mathbf{e}_i$ or $\mathbf{h}_i^u$.  Hereafter $\wt{d}_i$ can be different from line to line, up to the bound $\wt{d}_i\prec 1$ uniformly on $\mathcal{S}_{\mathcal{I}}(\eta_{\mathrm{m}}, \eta_{\mathrm{M}})$.  By~(\ref{0916101}), we see that $\Delta_G^u(i,k)$ is  a sum of the terms of the form 
\begin{align}
\wt{d}_i \bar{g}_{ik}^u G\hat{\boldsymbol{\alpha}}_i\hat{\boldsymbol{\beta}}_i^* \wt{B}^{\la i\ra}\mathcal{R}_i G\,,\qquad \wt{d}_i \bar{g}_{ik}^u G \mathcal{R}_i \wt{B}^{\la i\ra} \hat{\boldsymbol{\alpha}}_i\hat{\boldsymbol{\beta}}_i^* G\,, \label{100650}
\end{align}
where $ \hat{\boldsymbol{\alpha}}_i, \hat{\boldsymbol{\beta}}_i= \hat{\mathbf{e}}_i$ or ${\mathbf{k}}_i^u$. Then, by the definition in~(\ref{112470}), we see that $\varepsilon_{i2}$ is a sum of the terms of the form 
\begin{align*}
\frac{1}{N}\wt{d}_i (\mathring{\mathbf{k}}_i^u)^*  \wt{B}^{\la i \ra}  G\hat{\boldsymbol{\alpha}}_i\hat{\boldsymbol{\beta}}_i^* \wt{B}^{\la i\ra}\mathcal{R}_i G\hat{\mathbf{e}}_i\,,\qquad   \frac{1}{N}\wt{d}_i (\mathring{\mathbf{k}}_i^u)^*\wt{B}^{\la i\ra} G \mathcal{R}_i \wt{B}^{\la i\ra} \hat{\boldsymbol{\alpha}}_i\hat{\boldsymbol{\beta}}_i^* G \hat{\mathbf{e}}_i\,.
\end{align*}
Note that  using the trivial bound $\| G\|\le 1/\eta$, the terms above are stochastically dominated by  
\begin{align*}
\frac{1}{N\eta}\big|\hat{\boldsymbol{\beta}}_i^* \wt{B}^{\la i\ra}\mathcal{R}_i G\hat{\mathbf{e}}_i\big|\,,\qquad \frac{1}{N\eta}\big|\hat{\boldsymbol{\beta}}_i^* G \hat{\mathbf{e}}_i\big|
\end{align*}
respectively. It is easy to check 
\begin{align}
\big|\hat{\boldsymbol{\beta}}_i^* \wt{B}^{\la i\ra}\mathcal{R}_i G\hat{\mathbf{e}}_i\big|\prec 1\,,\qquad \big|\hat{\boldsymbol{\beta}}_i^* G \hat{\mathbf{e}}_i\big|\prec 1\,, \label{100670}
\end{align}
 for $\hat{\boldsymbol{\beta}}_i=\hat{\mathbf{e}}_i$ or $\mathbf{k}_i^u$. This can been seen from the facts ~(\ref{071553}) and ~(\ref{0913101}), 
and also the bounds~(\ref{112401}), which hold under the assumption~(\ref{071422}). From the above discussion, we can see that $|\varepsilon_{i2}|\prec \frac{1}{N\eta}$. This proves the first estimate in~(\ref{120501}). The second estimate 
 on $\varepsilon_{i3}$  can be verified in the same way. We omit the details. 

Now, we  prove ~(\ref{120502}). 
According to~(\ref{100650}), the left side of~(\ref{120502}) is a sum of terms of the~form 
\begin{align*}
\frac{1}{N} \wt{d}_i \hat{\mathbf{e}}_i^*X G \hat{\boldsymbol{\alpha}}_i  \;\hat{\boldsymbol{\beta}}_i^* \wt{B}^{\la i\ra} \mathcal{R}_i G \hat{\mathbf{e}}_i\; (\mathring{\mathbf{k}}_i^u)^* X_i G \hat{\mathbf{e}}_i\,,\\ \frac{1}{N} \wt{d}_i\hat{\mathbf{e}}_i^*X G \mathcal{R}_i \wt{B}^{\la i\ra} \hat{\boldsymbol{\alpha}}_i \; \hat{\boldsymbol{\beta}}_i^* G\hat{\mathbf{e}}_i\; (\mathring{\mathbf{k}}_i^u)^* X_i G \hat{\mathbf{e}}_i\,.
\end{align*}
Using~(\ref{XG}),~(\ref{100670}) and the bound  $|\hat{\mathbf{e}}_i^*X G \hat{\boldsymbol{\alpha}}_i|\prec \frac{1}{\eta}$ and $\big|\hat{\mathbf{e}}_i^*X G \mathcal{R}_i \wt{B}^{\la i\ra} \hat{\boldsymbol{\alpha}}_i \big|\prec \frac{1}{\eta}$, we can get 
(\ref{120502}). Then,~(\ref{120503}) can be proved similarly to~(\ref{120502}). Hence, we omit the details.  
Finally, we show~(\ref{100657}). According to~(\ref{100650}), $\frac{1}{N} \sum_{k}^{(i)} \ntr QX\Delta_{G}^u(i,k) \hat{\mathbf{e}}_k^*  X_i G\hat{\mathbf{e}}_i$ is a sum of terms of the~form 
\begin{align*}
\frac{1}{N^2} \wt{d}_i \hat{\boldsymbol{\beta}}_i^* \wt{B}^{\la i\ra} \mathcal{R}_i G  QX G \hat{\boldsymbol{\alpha}}_i\;  (\mathring{\mathbf{k}}_i^u)^* X_i G \hat{\mathbf{e}}_i\,, \\  \frac{1}{N^2} \wt{d}_i \hat{\boldsymbol{\beta}}_i^* G QX G \mathcal{R}_i \wt{B}^{\la i\ra}  \hat{\boldsymbol{\alpha}}_i\; (\mathring{\mathbf{k}}_i^u)^* X_i G \hat{\mathbf{e}}_i\,. 
\end{align*}
Then~(\ref{100657}) follows from~(\ref{XG}) and the trivial bounds 
\begin{align*}
\big| \hat{\boldsymbol{\beta}}_i^* \wt{B}^{\la i\ra} \mathcal{R}_i G  QX G \hat{\boldsymbol{\alpha}}_i\big|\prec \frac{1}{\eta^2}\,,\qquad  \big| \hat{\boldsymbol{\beta}}_i^* G QX G \mathcal{R}_i \wt{B}^{\la i\ra}  \hat{\boldsymbol{\alpha}}_i\big|\prec \frac{1}{\eta^2}\,. 
\end{align*}
Hence, we concluded the proof of Lemma~\ref{lem.100655}.
\end{proof}

  \section{} \label{Appendix C}

In this appendix, we explain how to modify our discussions  in Sections~\ref{s.green function subordination} and~\ref{s.strong law} 
to adapt to the orthogonal setup.  Recall our partial randomness decomposition of Haar unitary matrices~$U$ and~$V$ in~(\ref{050202}). For Haar orthogonal matrices $U$ and $V$, we  refer to Appendix A of \cite{BES15b} for an analogous decomposition, with the phases of the $i$-th components of $\mathbf{u}_i$ and $\mathbf{v}_i$ replaced by the signs of them. We then inherit all the notations introduced in Sections~\ref{s.green function subordination} and~\ref{s.strong law}. Under the orthogonal setting,  instead of \eqref{integration by parts formula}, 
we need to use the following integration by parts formula for real Gaussian random variables
\begin{align*}
\int_{\mathbb{R}} gf(g) \e{-\frac{g^2}{2\sigma^2}} {\rm d} g=\sigma^2 \int_\mathbb{R} f'(g) \e{-\frac{g^2}{2\sigma^2}} {\rm d} g\,,
\end{align*}
for differentiable functions $f: \mathbb{R}\to \mathbb{R}$. Consequently, instead of~(\ref{050903}), here we have, for $k\neq i$,
\begin{align*}
\frac{\partial R_i^a}{\partial g_{ik}^a}=&-\frac{(\ell_i^a)^2}{\|\mathbf{g}_i^a\|_2} \mathbf{e}_k\big(\mathbf{e}_i+\mathbf{h}_i^a\big)^*-\frac{(\ell_i^a)^2}{\|\mathbf{g}_i^a\|_2} \big(\mathbf{e}_i+\mathbf{h}_i^a\big)\mathbf{e}_k^*+2\Delta_R^a(i,k)\,,\qquad a=u,v\,,
\end{align*}
where $\Delta_R^a(i,k)$ is defined in~(\ref{0916100}). Thence we have the following modification of~(\ref{050915}):
\begin{align}
\frac{\partial G}{\partial g_{ik}^u}=\text{ right side of ~(\ref{050915}) }  &+c_i^u  G\big(\hat{\mathbf{e}}_i+\mathbf{k}_i^u\big)\hat{\mathbf{e}}_k^* \wt{ B }^{\la i\ra} \mathcal{R}_i G\nonumber\\ &+c_i^uG\mathcal{R}_i \wt{ B }^{\la i\ra} \big(\hat{\mathbf{e}}_i
+\mathbf{k}_i^u\big)\hat{\mathbf{e}}_k^* G\,. \label{101601}
\end{align}
The remaining task is to go through all the discussions in Sections~\ref{s.green function subordination} and~\ref{s.strong law} again, and show that all the estimates which involve the last two terms in~(\ref{101601})  are negligible at the right order.

To get through the discussions in Section~\ref{s.green function subordination} for orthogonal case, it suffices to take the last two terms in~(\ref{101601}) into the account of the derivation of the equations~(\ref{071440}) and  ~(\ref{071441}), as well as the last two estimates in~(\ref{071520}).   

Using~(\ref{101601}), we will have the following modification of~(\ref{071440}):
\begin{align}
\frac{1}{N}\sum_{k}^{(i)} \frac{  \partial (\hat{\mathbf{e}}_k^*\wt{ B }^{\la i\ra}  G\hat{\mathbf{e}}_i)}{\partial g^{u}_{ik}}= &\text{right side of~(\ref{071440})  }+\frac{c_i^u}{N} \hat{\mathbf{e}}_i^* G\mathcal{R}_i \wt{B}^{\la i\ra} \hat{I}_{1}^{\la i\ra} \wt{B}^{\la i\ra} G(\hat{\mathbf{e}}_i+\mathbf{k}_i^u)\nonumber\\
&\qquad +\frac{c_i^u}{N} \hat{\mathbf{e}}_i^* G \hat{I}_{1}^{\la i\ra} \wt{B}^{\la i\ra} G R_i \wt{B}^{\la i\ra} (\hat{\mathbf{e}}_i+\mathbf{k}_i^u)\,. \label{101803}
\end{align}
 Notice that the new terms are qualitatively different from the ones already present in~(\ref{071440}). In the new terms
the  summation over $k$ could be directly performed since $\mathbf{e}_k$ and $\mathbf{e}_k^*$ appear  directly next to each other,
 yielding the almost identity $\hat{I}_{1}^{\la i\ra}$. The analogous
sums in the old terms, explicitly seen in \eqref{083015}, result in a partial trace.

We will show that the last two terms above are of order $O_\prec(\Psi^2)$.  For the first one, note~that 
\begin{align}
\big|\hat{\mathbf{e}}_i^* G\mathcal{R}_i \wt{B}^{\la i\ra} \hat{I}_{1}^{\la i\ra} \wt{B}^{\la i\ra} G(\hat{\mathbf{e}}_i+\mathbf{k}_i^u)\big|\leq C\|G \hat{\mathbf{e}}_i\|_2 \big(\|G\hat{\mathbf{e}}_i\|_2+\|G\mathbf{k}_i^u\|_2\big)\,, \label{101801}
\end{align}
for some constant $C$. For the last term in~\eqref{101803}, 
using~(\ref{050805}), ~(\ref{0830200}) and also $\mathcal{R}_i^2=\hat{I}$, we get
\begin{align}
G \mathcal{R}_i \wt{B}^{\la i\ra} (\hat{\mathbf{e}}_i+\mathbf{k}_i^u)=-\tilde{\sigma}_i^* G\mathbf{k}_i^u-G\wt{B}\hat{\mathbf{e}}_i&=-\tilde{\sigma}_i^* G\mathbf{k}_i^u-\hat{\mathbf{e}}_i+G(A-z)\hat{\mathbf{e}}_i\nonumber\\
&= \tilde{\sigma}_i^* G\mathbf{k}_i^u-\hat{\mathbf{e}}_i-G\hat{\mathbf{e}}_i+\xi_i^*G\hat{\mathbf{e}}_{\hat{i}}\,. \label{101610}
\end{align}
 Thus, applying~(\ref{101610}), for the  last  term in \eqref{101803} we  have
\begin{align}
\big|\hat{\mathbf{e}}_i^* G \hat{I}_{1}^{\la i\ra} \wt{B}^{\la i\ra} G \mathcal{R}_i \wt{B}^{\la i\ra} (\hat{\mathbf{e}}_i+\mathbf{k}_i^u)\big|\leq C\|G\hat{\mathbf{e}}_i\|_2\big(\|G\hat{\mathbf{e}}_i\|_2+\|G\hat{\mathbf{e}}_{\hat{i}}\|_2+\|G{\mathbf{k}}_i^u\|_2\big)\,. \label{101802}
\end{align}

According to~(\ref{101803}),~(\ref{101801}) and~(\ref{101802}), it suffices to prove 
\begin{align}
\|G\hat{\mathbf{e}}_i\|_2\prec \frac{1}{\sqrt{\eta}}\,,\qquad \|G\hat{\mathbf{e}}_{\hat{i}}\|_2\prec \frac{1}{\sqrt{\eta}}\,,\qquad  \|G{\mathbf{k}}_i^u\|_2\prec \frac{1}{\sqrt{\eta}}  \label{101804}
\end{align}
 to get a bound $O_\prec(\Psi^2)$ for the last two terms in \eqref{101803}. 
To show~(\ref{101804}), we use the identities
\begin{align}
\|G\hat{\mathbf{e}}_i\|_2^2=\frac{1}{\eta} \Im G_{ii}\,,\qquad  \|G\hat{\mathbf{e}}_{\hat{i}}\|_2^2=\frac{1}{\eta} \Im G_{\hat{i}\hat{i}}\,,\qquad    \|G{\mathbf{k}}_i^u\|_2^2=\frac{1}{\eta} \Im ({\mathbf{k}}_i^u)^* G{\mathbf{k}}_i^u\,.  \label{101807}
\end{align}
Applying assumption~(\ref{071422}) and~(\ref{101807}), we can get the first two estimates in~(\ref{101804}). Using the last identity of~(\ref{101807}),~(\ref{112230}) and~(\ref{112460}), we can get the last estimate in~(\ref{101804}). The necessary modification for the proofs of~(\ref{071441}) and  the last three estimates in~(\ref{071520}) can be done in the same way, we thus omit the details. 

For the discussions in  Section~\ref{s.strong law}, in the orthogonal case,
 the  averaged  analogue  of~(\ref{071501}), \ie ~(\ref{083170}), still holds. That is because  the last two terms in~(\ref{101803})  and their analog in the equation for $\frac{1}{N}\sum_{k}^{(i)} \frac{  \partial (\hat{\mathbf{e}}_k^*  G\hat{\mathbf{e}}_i)}{\partial g^{u}_{ik}}$ are of order $O_\prec(\Psi^2)$. So  the contribution of these additional terms in~(\ref{083170})
can be  absorbed   into the last term of~(\ref{083170}). Thence, the remaining proof  is the same as the unitary case. Hence, we completed the necessary modifications for the orthogonal setup.

\end{document}